\documentclass{siamltex} 

\title{A comparison of mixed precision iterative refinement approaches for least-squares problems}
\author{Erin Carson\thanks{Faculty of Mathematics and Physics, Charles University, carson@karlin.mff.cuni.cz.} \and 
Ieva Dau\v{z}ickait\.{e}\thanks{CERFACS, Toulouse, France, dauzickaite@cerfacs.fr.}\\
Both authors were supported by Charles University Research Centre program No. UNCE/24/SCI/005, the Exascale Computing Project (17-SC-20-SC), a collaborative effort of the U.S. Department of Energy Office of Science and the National Nuclear Security Administration, and by the European Union (ERC, inEXASCALE, 101075632). Views and opinions expressed are those of the authors only and do not necessarily reflect those of the European Union
or the European Research Council. Neither the European Union nor the granting authority can be held responsible
for them.}

\usepackage{amsmath}
\usepackage{amssymb}
\usepackage{comment}
\usepackage{graphicx}
\usepackage{epstopdf}
\usepackage{color}
\usepackage{graphicx}
\usepackage{lipsum}
\usepackage{graphicx}
\usepackage{subcaption}
\usepackage{hyperref}
\usepackage{algorithmicx}
\usepackage{algpseudocode}
\usepackage[T1]{fontenc}
\usepackage[dvipsnames]{xcolor}
\usepackage{listings}
\usepackage{mathtools}
\usepackage{multirow} 
\usepackage{soul}
\usepackage{etoolbox}
\usepackage{tikz,forest}
\usetikzlibrary{backgrounds}
\usetikzlibrary{trees,positioning}
\usetikzlibrary{arrows.meta}
\tikzset{
  treenode/.style = {shape=rectangle, rounded corners,
                     draw, align=center, 
                     minimum height=2ex, text depth=0.25ex,
                     top color=white, bottom color=blue!20},
  root/.style     = {treenode, font=\Large\rmfamily, bottom color=red!30},
  env/.style      = {treenode, font=\ttfamily\normalsize},
}

\newcommand{\Atld}{\widetilde{A}}

\newcommand{\Rhat}{\widehat{R}}

\newcommand{\fhat}{\widehat{f}}

\newcommand{\xhat}{\widehat{x}}
\newcommand{\rhat}{\widehat{r}}
\newcommand{\tg}{\tilde{\gamma}}

\newtheorem{observation}[theorem]{Observation}

\usepackage[section]{algorithm}
\renewcommand{\algorithmicrequire}{\textbf{Input: }}
\renewcommand{\algorithmicensure}{\textbf{Output: }}

\begin{document}

\maketitle

\renewcommand{\thefootnote}{\fnsymbol{footnote}}

\begin{abstract}
Various approaches to iterative refinement (IR) for least-squares problems have been proposed in the literature and it may not be clear which approach is suitable for a given problem. We consider three approaches to IR for least-squares problems when two precisions are used and review their theoretical guarantees, known shortcomings and when the method can be expected to recognize that the correct solution has been found, and extend uniform precision analysis for an IR approach based on the semi-normal equations to the two-precision case. We focus on the situation where it is desired to refine the solution to the working precision level. It is shown that the IR methods exhibit different sensitivities to the conditioning of the problem and the size of the least-squares residual, which should be taken into account when choosing the IR approach. We also discuss a new approach that is based on solving multiple least-squares problems.
\end{abstract}

\begin{keywords}
mixed precision, iterative refinement, least-squares, semi-normal equations
\end{keywords}

\begin{AMS}
65F05, 65F20
\end{AMS}

\section{Introduction}
Iterative refinement is used to reduce the impact of finite precision errors in computations. When solving linear systems of equations $Ax=b$, the solution can be refined by repeatedly computing the residual in a possibly higher precision and solving the system with the residual on the right-hand side to obtain an update to the solution, that is, computing $x_0$ in some way and repeating the following steps:
\begin{enumerate}
    \item Compute $r_i = b - A x_i$.
    \item Solve $ A \delta x_i = r_i$. \label{ir_step_correction}
    \item Update $ x_{i+1} = x_i + \delta x_i$.
\end{enumerate}
Various approaches for solving the correction equation in step~\ref{ir_step_correction} with convergence guarantees depending on the condition number of the coefficient matrix $A$ can be used, e.g. \cite{amestoy2024five, carson2018accelerating}. 

More care may be needed in the case of least-squares problems
\begin{equation}\label{eq:LS} 
    \min_x \Vert b - Ax \Vert_2, 
\end{equation} 
where $A \in \mathbb{R}^{m \times n}$, $m \geq n$, is full rank, $x  \in \mathbb{R}^n$, and $b \in \mathbb{R}^m$. The least-squares residual $r = b - Ax$ may be non-zero and this invalidates some of the convergence guarantees that are available in the linear system case.

There is a renewed interest in iterative refinement for least-squares problems \cite{epperly2024fast}, and laterally in employing low precision computations for such problems \cite{chen2023half,li2024double}. In this note, we thus examine three approaches to least-squares iterative refinement (LSIR) in two precisions proposed in the literature, show how two of the approaches can be combined so that the solution updates come from solving a few least-squares problems independently, and discuss their (lack of) convergence guarantees. We also investigate the conditions under which the methods can be guaranteed to recognize that a solution meeting the required accuracy has been found, ensuring that the new computed update does not corrupt the solution.
Further, the forward error analysis for the semi-normal equations LSIR approach is extended from the uniform precision case to the case where two precisions are used (Section~\ref{sec:semi-normal}). In this case, we determine the limiting accuracy and the convergence conditions. 
In the original form, LSIR strategies employ a direct solve for computing the update and are based on a QR decomposition of $A$ or the $R$-factor of $A$. This may be appropriate when $A$ is dense, but iterative methods may be preferred for sparse $A$. We include some numerical experiments to illustrate the behavior of LSIR when iterative methods are used, but theoretical convergence guarantees are not provided. 
 
The discussed LSIR procedures were originally introduced assuming the problem data $A$ and $b$ are in single precision and all computations (except computing the residual) are performed in single precision as well \cite{bjorck1967iterative,golub1966note}. The higher precision for the residual computations is achieved by accumulating the inner products in double precision, that is, the scalar products needed to compute the matrix-vector product of a single precision $A$ and a single precision vector $x$ were computed and added in double precision, and then rounded to single precision, which was faster than doing all the computations in double precision on hardware at the time; see, e.g. \cite[Section 2.3(f)]{von1947numerical}. Note that the units roundoff for IEEE double and single precisions are of the order $10^{-16}$ and $10^{-8}$, respectively. We present general versions of these LSIR procedures in two precisions with units roundoff $u$ (working precision) and $u_r$ (residual computations) with $u_r \leq u^2$. 

We note that fixed precision iterative refinement, where uniform precision is used for all computations, can be used to refine a solution obtained by a possibly not backward-stable solver; the goal is to obtain a solution of the same quality as when a backward-stable solver is used. The mixed precision iterative refinement however goes further and aims to remove the effect of the condition number of the coefficient matrix on the solution. This of course comes at the cost of using a higher precision for the residual computations. The focus of our convergence analysis is the relative forward error and when it can be 
guaranteed to reach the working precision level $\mathcal{O}(u)$. One could argue that solving ill-conditioned problems to high accuracy is an unnecessary task; note, however, that if the working precision is set to a low precision then the notion of an ill-conditioned problem changes accordingly. For example, when $u$ is set to IEEE half precision then problems with condition number as low as $10^3$ can be considered to be ill-conditioned in the working precision even though it is a moderately conditioned problem in the usual double precision. 

This note is structured as follows. We discuss the least-squares system approach to LSIR in Section~\ref{sec:LS_approach}. The semi-normal equations and augmented system approaches follow in Sections~\ref{sec:semi-normal} and \ref{sec:augmented}, respectively. A short discussion on using iterative solvers for LSIR is presented in Section~\ref{sec:iterative}. We comment on how to combine the least-squares system and augmented system approaches in Section~\ref{sec:combining_ls_augmented}. Numerical experiments illustrate the theory in Section~\ref{sec:numerics} and we summarize and give insight into choosing a solver in practice in Section~\ref{sec:summary}.

\subsection{Notation}
In the following, hats are used to refer to the computed quantities; for example, $\rhat$ is the computed version of $r$. We use a standard model of floating point arithmetic and define 
\begin{equation*}
    \tg_n = \frac{cnu}{1-cnu} \quad \text{and} \quad \tg_n^{(r)} = \frac{cnu_r}{1-cnu_r},
\end{equation*}
where $c$ is a small constant independent of $n$; see, e.g., \cite{high:ASNA2}. When $c=1$ we use $\gamma_n$ or $\gamma_n^{(r)}$. We use $\lesssim$ when we omit higher order terms which are insignificant in the expression. $\kappa(A)$ denotes the 2-norm condition number of $A$, that is, $\kappa(A) = \Vert A \Vert_2 \Vert A^{\dagger} \Vert_2 $, where $A^{\dagger}$ is the Moore-Penrose pseudoinverse of $A$.

\section{Least-squares system approach}\label{sec:LS_approach}

Golub \cite{golub1965numerical} proposed a straightforward extension of the IR procedure to least-squares problems, which we specify in two precisions in Algorithm~\ref{alg:LS_approach}. We call it the LS approach.

\begin{algorithm}
\caption{Two-precision LSIR based on the least-squares system approach}\label{alg:LS_approach}
\algorithmicrequire  $A \in \mathbb{R}^{m \times n}$, $b \in \mathbb{R}^m$, initial solution $x_0 \in \mathbb{R}^n$, precisions $u_r$ and $u$, where $u_r \leq u^2$ \\
\algorithmicensure approximate solution $x$
\begin{algorithmic}[1]
\State $i=0$
\While{not converged}
\State Compute $r_i = b - A x_i$ \Comment{$u_r$}
\State Solve $\min_{\delta x_i} \Vert r_i - A \delta x_i \Vert_2$ via QR decomposition \Comment{$u$} \label{alg_step:LS_approach_correction}
\State Update $ x_{i+1} = x_i + \delta x_i$ \Comment{$u$} 
\State $i=i+1$
\EndWhile
\end{algorithmic}
\end{algorithm}
No convergence guarantees are provided by Golub, but it is noted that this LSIR procedure is likely to converge only if the initial approximation $x_0$ is of sufficient accuracy. Businger and Golub \cite{businger1965linear} implement this procedure and point out that ``there is no assurance, however, that all digits of the final solution will be correct.''
The theoretical analysis by Golub and Wilkinson \cite{golub1966note} reveals some shortcomings of the approach in Algorithm~\ref{alg:LS_approach}. It is shown that when Householder QR decomposition is used to solve the correction equation in step~\ref{alg_step:LS_approach_correction}, the correct solution may not be recognized (as described below) unless the system is nearly compatible, that is, $\Vert r \Vert_2$ is nearly zero. In the following, we restate and explain    their argument in the general case, i.e., assuming that the solve in step~\ref{alg_step:LS_approach_correction} is performed in some precision $u$. 
The error in computing the residual $r_i$ is ignored.

First, we note that the correction $\delta \xhat_i$ computed in step~\ref{alg_step:LS_approach_correction} is the exact solution of a perturbed problem (see, e.g., \cite[Theorem 20.3]{high:ASNA2})
\begin{gather*}
    \min_{\delta x_i} \Vert r_i + \Delta r_i - (A + \Delta A_i)\delta \xhat_i \Vert_2, \textrm{ where} \\
    \Vert \Delta A_i \Vert_F \leq \tg_{mn}  \Vert A \Vert_F \textrm{ and } \Vert \Delta r_i \Vert_2 \leq \tg_{mn}  \Vert r_i \Vert_2,
\end{gather*}
and $\Delta A_i$ is the perturbation matrix associated with $\delta \xhat_i$. Then we can bound the error between the true correction $\delta x_i$ and the computed $\delta \xhat_i$ using \cite[Theorem 20.1]{high:ASNA2}, where we set $\epsilon = \tg_{mn}$ and assume that $\tg_{mn} \kappa(A)<1$. Ignoring small constant factors we obtain
\begin{equation}\label{eq:LS_forward_error_bound}
    \Vert \delta \xhat_i - \delta x_i \Vert_2 \leq \frac{\tg_{mn} }{1- \tg_{mn} \kappa(A)} \left( \kappa(A)\Vert \delta x_i\Vert_2 + \kappa(A)^2 \frac{\Vert  r_{\delta x_i } \Vert_2}{\Vert A \Vert_2  } \right),
\end{equation}
where $r_{\delta x_i }= r_i - A\delta x_i$. 

To illustrate the possible failing of Algorithm~\ref{alg:LS_approach}, assume that $x_i =x^*$, where $x^*$ is the minimizer of \eqref{eq:LS}. In this case, we have $r_i = b-Ax^* = r^*$, and thus $\delta x_i = 0$ and $\delta r_i = r^*$.
Ideally in this case, we would have $\Vert \delta \xhat_i \Vert_2/\Vert x^*\Vert_2  = \mathcal{O}(u) $ for the \emph{computed} $\delta\xhat_i$, since we would like to recognize that we have at least found the solution to working accuracy and thus the refinement should be terminated at this point.
Unfortunately, from \eqref{eq:LS_forward_error_bound}, we can only say that 
\begin{equation*}
    \Vert \delta \xhat_i \Vert_2 \leq \frac{\tg_{mn} }{1- \tg_{mn} \kappa(A)} \kappa(A)^2 \frac{\Vert r^*\Vert_2}{\Vert A \Vert_2  }.
\end{equation*}
This means that in order to guarantee that $\Vert \delta \xhat_i\Vert_2/\Vert x^*\Vert_2 = O(u)$ is achievable, we must have that 
\begin{equation}\label{eq:gol_wilk_condition}
        \kappa(A)^2 \frac{\Vert r^* \Vert_2}{\Vert A \Vert_2 \Vert x^* \Vert_2 } < 1.
\end{equation}
When the least-squares problem is almost compatible, that is, $ \Vert r^* \Vert_2$ is small, the condition \eqref{eq:gol_wilk_condition} is satisfied for a wider range of coefficient matrices $A$. Note that the condition is independent of $u$. Golub and Wilkinson \cite{golub1966note} thus conclude that ``whatever precision of computation is used there will be right-hand sides for which iterative refinement will never give solutions which are correct to working accuracy.''

\section{Semi-normal equations approach}\label{sec:semi-normal}

Golub and Wilkinson \cite{golub1966note} briefly discuss an approach to iterative refinement that employs the semi-normal equations
\begin{equation*}
    R^T R \delta x_i = A^T r_i,
\end{equation*}
where $R$ is the $R$-factor of the QR decomposition of $A$. The approach is attributed to Kahan. We present it in Algorithm~\ref{alg:semi-normal_approach}.

\begin{algorithm}
\caption{Two-precision LSIR based on the semi-normal equations approach}\label{alg:semi-normal_approach}
\algorithmicrequire  $A \in \mathbb{R}^{m \times n}$, $b \in \mathbb{R}^m$, $R \in \mathbb{R}^{n \times n}$ factor of the Householder QR decomposition of $A$, initial solution $x_0 \in \mathbb{R}^n$, precisions $u_r$ and $u$, where $u_r \leq u^2$ \\
\algorithmicensure approximate solution $x$
\begin{algorithmic}[1]
\State $i=0$
\While{not converged}
\State Compute $r_i = b - A x_i$ \Comment{$u_r$} \label{alg:sn_step_r}
\State Solve $R^T R \delta x_i = A^T r_i$ \Comment{Computing $A^T r_i$ in $u_r$, solving in $u$ } \label{eq:correction_step_semi_normal}
\State Update $ x_{i+1} = x_i + \delta x_i$ \Comment{$u$} \label{alg:sn_step_x_upd}
\State $i=i+1$
\EndWhile
\end{algorithmic}
\end{algorithm}
In \cite{bjorck1987stability}, Bj\"{o}rck derives a bound on the forward error for what he calls the ``corrected'' semi-normal equations approach, in which one step of iterative refinement is performed and a uniform precision is used throughout the computation. His goal is to obtain a solution with forward error equivalent to the error given by a backward stable LS solver rather than obtain a forward error to the level $\mathcal{O}(u)$. Recently Epperly \cite{epperly2024fast} extended Bj\"{o}rck's analysis to the case when $R$ is computed via a QR factorization $\Omega A = QR$, where $\Omega$ is a random sketching matrix, and $x_0$ is obtained as $x_0 = R^{-1} Q^T \Omega b$; if a high quality $R$ is obtained, then the method is forward stable.
In the following, we discuss some qualities of the computed $R$-factor of $A$ that is used in the correction equation, consider when the correct solution is recognized by the method, and analyze the convergence of the forward error in the semi-normal equations approach to LSIR, adapting the analysis for the corrected semi-normal equations in uniform precision of \cite{bjorck1987stability} to the mixed precision case described in Algorithm~\ref{alg:semi-normal_approach}.

\subsection{Computed $R$-factor of $A$}
We first consider $\Rhat$, which is the computed $R$-factor of $A$. It is assumed that $\Rhat$ is computed via the backward stable Householder QR algorithm in precision with unit roundoff $u$ and hence we can write (see, e.g., \cite[Theorem 19.4]{high:ASNA2})
\begin{gather}
    A+E = \tilde{Q} \Rhat, \textrm{ where }  \label{eq:HHQR_pert}\\
    \Vert E \Vert_F \leq \tg_{mn} \Vert A \Vert_F, \label{eq:E_norm}
\end{gather}
and $\tilde{Q}$ is orthogonal. 
It follows that
\begin{gather}
\Vert \Rhat \Vert_2 = \Vert \tilde{Q} \Rhat \Vert_2 = \Vert A + E \Vert_2 \leq \left( 1 + n^{1/2}\tg_{mn} \right) \Vert A \Vert_2. \label{eq:Rhat_bound}
\end{gather}
We assume that $\kappa(A)$ and $u$ are such that 
\begin{equation*}
    c(m,n) u \kappa(A) < 1,
\end{equation*}
where $c(m,n)$ is a constant depending on $m$ and $n$. Then $\Rhat$ is invertible and we can obtain an upper bound for $\Vert \Rhat^{-1} \Vert_2$ following the approach detailed in \cite[eq. 6.1-6.3]{bjorck1967solving}, that is, we write
\begin{gather*}
   \Rhat^T \Rhat = \Rhat^T \tilde{Q}^T \tilde{Q} \Rhat = (A+E)^T (A+E) = R^T (I + F) R, \quad \text{where } \\
   F = (Q^T E R^{-1})^T + Q^T E R^{-1} + (E R^{-1} )^T E R^{-1}.
\end{gather*}
Using \eqref{eq:E_norm} and $\Vert R^{-1} \Vert_2 = \Vert A^{\dagger} \Vert_2$ it follows that
\begin{equation}\label{eq:bound_norm_F}
    \Vert F \Vert_2 \leq 2 n^{1/2}\tg_{mn} \kappa(A) + n \tg_{mn}^2 \kappa(A)^2 < 1.
\end{equation}
We thus can use a first order approximation $(I + F)^{-1} \approx I - F$ in the following:
\begin{equation*}
    \Vert \Rhat^{-1} \Vert_2^2 = \Vert (\Rhat^T \Rhat)^{-1} \Vert_2 \approx \Vert  R^{-1} (I - F) R^{-T} \Vert _2 \leq \left( 1 +  \Vert F \Vert_2 \right) \Vert R^{-1} \Vert_2^2.
\end{equation*}
Taking the square root, employing \eqref{eq:bound_norm_F}, and ignoring the second order terms gives
\begin{equation}\label{eq:Rhatinverse_bound}
     \Vert \Rhat^{-1} \Vert_2 \leq \left( 1 +  2 n^{1/2}\tg_{mn} \kappa(A) \right)\Vert R^{-1} \Vert_2.
\end{equation}
We combine this with \eqref{eq:Rhat_bound} to obtain
\begin{equation}\label{eq:cond_no_Rhat_bound_A}
    \kappa(\Rhat) = \kappa(A) \left( 1 + \mathcal{O}(u)\kappa(A) \right).
\end{equation}
Then we also have
\begin{equation}\label{eq:norm_ARhatinv}
    \Vert A \Rhat^{-1} \Vert_2 = \Vert\tilde{Q} - E \Rhat^{-1} \Vert_2 \leq 1 + n^{1/2}\tg_{mn} \Vert A \Vert_2 \Vert \Rhat^{-1} \Vert_2 
    \lesssim 1 + n^{1/2}\tg_{mn}  \kappa(A).
\end{equation}

\subsection{Correction to the true solution}
The update in step~\ref{eq:correction_step_semi_normal} of Algorithm~\ref{alg:semi-normal_approach} is computed by solving two triangular systems via forward and backward substitutions. Golub and Wilkinson explore when this approach recognizes the correct solution $x^*$ \emph{if} it is found. They consider the case when $r_i$ and $A^T r_i$ are computed in higher (double) precision $u_r$ and only the computed $A^T r_i$ is cast to the lower (single) precision $u$. Note that if $x_i = x^*$ and the error in computing $fl(A^T r_i)$ is ignored, then the right-hand side of the correction equation in step~\ref{eq:correction_step_semi_normal} of Algorithm~\ref{alg:semi-normal_approach} is zero and thus $\delta x_i = 0$. If $x_i = x^*$, $\Vert A \Vert_2 = 1$ and we account for the error in computing $fl(A^T r_i) = A^T  r^* + A^T h_i = A^T h_i$ in $u_r$, it is shown that the true solution $\delta \widetilde{x}_i$ to the correction equation
\begin{equation*}
    \Rhat^T \Rhat \delta \widetilde{x}_i = A^T h_i
\end{equation*}
satisfies
\begin{equation*}
    \Vert \delta \widetilde{x}_i \Vert_2 \leq u_r \kappa(A)^2 \Vert r^* \Vert_2.
\end{equation*}
We can write this in a form similar to \eqref{eq:gol_wilk_condition}. If we set $u_r = u^2$, then we can guarantee $\Vert \delta \widetilde{x}_i  \Vert_2 / \Vert x^* \Vert_2 = \mathcal{O}(u)$ when
\begin{equation}\label{eq:condition_recognise_semi_normal}
   \kappa(A)^2 \frac{ \Vert r^* \Vert_2}{\Vert x^* \Vert_2} < u^{-1}.
\end{equation}
A bound for the computed solution $ \delta \xhat_i$ (where the errors in the triangular solvers are also accounted for) is not provided in \cite{golub1966note}. The forward error $ \Vert  \delta \widetilde{x}_i -  \delta \xhat_i \Vert_2$ for two triangular solves with $\Rhat$ and $\Rhat^T$ in precision $u$ can be bounded using 
\begin{equation*}
    \delta \xhat_i = (\Rhat + \Delta R_2)^{-1} (\Rhat^T + \Delta R_1)^{-1} A^T h_i,
\end{equation*}
where $\Vert \Delta R_i \Vert_2 \leq n^{1/2} \gamma_n \Vert \Rhat \Vert_2$. 
Then employing a first order approximation to the inverses, $\Vert h_i \Vert_2 \leq u_r \Vert r^* \Vert_2$, \eqref{eq:Rhatinverse_bound}, \eqref{eq:cond_no_Rhat_bound_A}, \eqref{eq:norm_ARhatinv} and ignoring higher order terms gives
\begin{equation*}
    \Vert  \delta \widetilde{x}_i -  \delta \xhat_i \Vert_2 \lesssim  n^{1/2} \gamma_n u_r \kappa(A)^2 \Vert r^* \Vert_2.
\end{equation*}
We can guarantee $\Vert  \delta \widetilde{x}_i -  \delta \xhat_i \Vert_{\infty} / \Vert x^* \Vert_2  = \mathcal{O}(u)$ if 
\begin{equation*}
     \kappa(A)^2 \frac{ \Vert r^* \Vert_2}{\Vert x^* \Vert_2} < u_r^{-1},
\end{equation*} 
which clearly holds whenever \eqref{eq:condition_recognise_semi_normal} holds. Note that unlike \eqref{eq:gol_wilk_condition} the condition \eqref{eq:condition_recognise_semi_normal} depends on the precision used and thus problems caused by large residual or ill-conditioning can be addressed by increasing the precision.

\subsection{The forward error}
We bound the relative forward error in the following theorem, followed by comments on the results. The proof of the theorem is provided after the comments. 
\begin{theorem}\label{th:semi_normal}
    Let $\xhat_{i}$ be the solution computed in iteration $i$ of Algorithm~\ref{alg:semi-normal_approach} and $x^*$ be the true solution of \eqref{eq:LS}. Then the relative forward error after $i$ iterations is bounded as
    \begin{align}
     \frac{\Vert x^* - \xhat_{i} \Vert_2}{\Vert x^* \Vert_2} \leq & \, c_1(m,n) u \kappa(A)^2 \frac{\Vert  x^* - \xhat_{i-1} \Vert_2}{\Vert x^* \Vert_2}  \nonumber \\
     & \, + u \frac{\Vert \xhat_{i} \Vert_2}{\Vert x^* \Vert_2}  + c_2(m,n) u_r \kappa(A)^2 \frac{\Vert A \Vert_2 \Vert \xhat_{i-1} \Vert_2 + \Vert b \Vert_2}{\Vert A \Vert_2 \Vert x^* \Vert_2}, \label{eq:semi_normal_fe_bound}
\end{align}
where $c_1(m,n)$ and $c_2(m,n)$ are constants depending on $m$ and $n$.
\end{theorem}

\begin{observation}
    We can thus conclude that the solution is improved with every LSIR iteration if $ c_1(m,n) u \kappa(A)^2<1$, which is essentially
\begin{equation}\label{eq:kappaA_condition_seminormal}
    \kappa(A) < u^{-1/2}.
\end{equation}
When we approach convergence we have $\xhat_i = \xhat_{i+1} + \mathcal{O}(u) =  x^* + \mathcal{O}(u)$, that is,  $\Vert \xhat_i \Vert_2 \approx \Vert \xhat_{i+1} \Vert_2 \approx \Vert x^* \Vert_2$, and by plugging this and $\Vert b \Vert_2 \leq \Vert r^* \Vert_2 + \Vert A \Vert_2 \Vert x^* \Vert_2$ into \eqref{eq:semi_normal_fe_bound} we can conclude that the limiting normwise relative error is of the order
\begin{equation*}
u + c_2(m,n) u_r \kappa(A)^2 \left(2 +\frac{ \Vert r^* \Vert_2}{\Vert A \Vert_2 \Vert x^* \Vert_2 }\right).
\end{equation*}
We thus achieve $\mathcal{O}(u)$ relative error if 
\begin{equation}
    u_r \kappa(A)^2 \left(2 +\frac{ \Vert r^* \Vert_2}{\Vert A \Vert_2 \Vert x^* \Vert_2 }\right) < u
    \label{eq:cond_sneir}
\end{equation}
and if $u_r = u^2$ we require
\begin{equation*}
     \kappa(A)^2 \left(2 +\frac{ \Vert r^* \Vert_2}{\Vert A \Vert_2 \Vert x^* \Vert_2 }\right) < u^{-1},
\end{equation*}
which is satisfied when \eqref{eq:kappaA_condition_seminormal} holds and 
\begin{equation*}
    \frac{ \Vert r^* \Vert_2}{\Vert A \Vert_2 \Vert x^* \Vert_2} = \mathcal{O}(1).
\end{equation*}
Notice that as long as the above holds, $\Vert r^* \Vert_2$ does not affect the limiting error.
\end{observation}

We continue with the proof of Theorem~\ref{th:semi_normal}.

\begin{proof}
We bound the error $\Vert x^* - \xhat_{i+1} \Vert_2$ by accounting for the finite precision error in steps~\ref{alg:sn_step_r} to \ref{alg:sn_step_x_upd} of Algorithm~\ref{alg:semi-normal_approach} as follows: 
\begin{gather}
    \rhat_i = b - A \xhat_i + f, \label{eq:sn_rhat}\\
    \Rhat^T \Rhat \delta \xhat_i = A^T \rhat_i + k, \label{eq:sn_triang_solves} \\
    \xhat_{i+1} = \xhat_i + \delta \xhat_i  + e, \label{eq:sn_update_xhat}
\end{gather}
where the vector $f$ accounts for the finite precision error in computing $fl(r_i)$, $k$ accounts for the error in computing $fl(A^T \rhat_i)$ and the triangular solves with $\Rhat$, and $e$ accounts for the error in adding $\xhat_i $ and $ \delta \xhat_i $. Using \eqref{eq:sn_rhat}-\eqref{eq:sn_update_xhat}, we write 
\begin{align}
    x^* - \xhat_{i+1} = & \, x^* - \xhat_i - \delta \xhat_i  - e \nonumber \\
    = & \, x^* - \xhat_i - (\Rhat^T \Rhat )^{-1} \left( A^T \rhat_i + k \right) - e \nonumber\\
    = & \, x^* - \xhat_i - (\Rhat^T \Rhat )^{-1} \left( A^T ( b - A \xhat_i + f) + k \right) - e \nonumber \\
    = & \, M ( x^* - \xhat_i ) - (\Rhat^T \Rhat )^{-1} \left( A^T f + k \right) - e, \label{eq:error_xtrue_x_i+1}
\end{align}
where $M = I - (\Rhat^T \Rhat )^{-1} A^T A$ and we used $A^T b = A^T A x$ in the last equation. We first aim to bound $\Vert M \Vert_2$ and for this we define
\begin{equation*}
    \bar{Q} = A \Rhat^{-1}. 
\end{equation*}
Note that $\Vert \bar{Q} \Vert_2$ is bounded in \eqref{eq:norm_ARhatinv}.
We then write
\begin{equation}\label{eq:M_definition}
    M = \Rhat^{-1} \left( I -  \bar{Q}^T \bar{Q} \right) \Rhat.
\end{equation}
Note that 
\begin{equation*}
     I -  \bar{Q}^T \bar{Q} = \tilde{Q}^T E \Rhat^{-1} + \Rhat^{-T} E^T \tilde{Q} - \Rhat^{-T} E^T E \Rhat^{-1}.
\end{equation*}
Taking the norm and using \eqref{eq:E_norm} and \eqref{eq:Rhatinverse_bound} gives
\begin{align*}
    \Vert  I -  \bar{Q}^T \bar{Q}  \Vert_2 \leq & \, 2 \Vert E \Vert_2 \Vert \Rhat^{-1} \Vert_2 + \Vert E \Vert_2^2 \Vert \Rhat^{-1} \Vert_2^2 \nonumber\\
    \lesssim  & \, 2 n^{1/2} \tg_{mn} \kappa(A) + n \tg_{mn}^2 \kappa(A)^2. \nonumber 
\end{align*}
Combining this with \eqref{eq:M_definition} and \eqref{eq:cond_no_Rhat_bound_A} results in
\begin{align}
    \Vert  M  \Vert_2 \leq & \,  2 n^{1/2} \tg_{mn} \kappa(A) \kappa(\Rhat)  + n \tg_{mn}^2 \kappa(A)^2 \kappa(\Rhat) \nonumber\\
    \leq & \, 2 n^{1/2} \tg_{mn}  \kappa(A)^2 \left( 1 + \mathcal{O}(u\kappa(A)) \right)^2 + n \tg_{mn}^2  \kappa(A)^3 \left( 1 + \mathcal{O}(u\kappa(A)) \right)^3 \nonumber\\
    \lesssim & \,  2 n^{1/2} \tg_{mn} \kappa(A)^2.  \label{eq:M_norm_bound}
\end{align}
We proceed by bounding the norm of $ (\Rhat^T \Rhat )^{-1} \left( A^T f + k \right)$. The error $f$ can be bounded as
\begin{equation}\label{eq:f_bound}
\Vert f \Vert_2 \leq n^{1/2} \tg_n^{(r)} \left( \Vert A \Vert_2 \Vert \xhat_i \Vert_2 + \Vert b \Vert_2\right).    
\end{equation}
We write the error vector $k$ as $ k = k_1 + k_2$, where $k_1$ accounts for the error in computing $fl(A^T \rhat_i)$ and $k_2$ accounts for the error in two triangular solves. Then
\begin{equation}\label{eq:boubd_k1}
    \Vert k_1 \Vert_2 \leq n^{1/2} \gamma_n^{(r)} \Vert A \Vert_2 \Vert \rhat \Vert_2 \lesssim n^{1/2} \gamma_n^{(r)} \Vert A \Vert_2  \left( \Vert A \Vert_2 \Vert \xhat_i \Vert_2 + \Vert b \Vert_2 \right),
\end{equation}
where we ignore the terms in $u_r^2$. To bound $\Vert k_2 \Vert_2$, we use standard results on the backward error in triangular solves, e.g. \cite[Theorem 8.5]{high:ASNA2}, and obtain
\begin{equation}\label{eq:trinag_solves_with_backward_error}
    \left( \Rhat^T + \Delta R_1^T \right)\left( \Rhat + \Delta R_2 \right) \delta \xhat_i = A^T \rhat_i + k_1,
\end{equation}
where $\Vert \Delta R_j \Vert_F \leq \gamma_n \Vert \Rhat \Vert_F$ for $j=1,2$. Then using \eqref{eq:sn_triang_solves} we have
\begin{equation*}
    k +  \left( \Rhat^T \Delta R_2  + \Delta R_1^T \Rhat + \Delta R_2 \Delta R_1^T \right) \delta  \xhat_i = k_1
\end{equation*}
and
\begin{equation*}
    k_2 = k - k_1 = - \left( \Rhat^T \Delta R_2  + \Delta R_1^T \Rhat + \Delta R_2 \Delta R_1^T \right) \delta  \xhat_i.
\end{equation*}
Multiplying with $\Rhat^{-T}$ on the left and moving $\Rhat$ gives
\begin{equation*}
    \Rhat^{-T} k_2 = - \left( \Delta R_2 \Rhat^{-1} +  \Rhat^{-T} \Delta R_1^T  +  \Rhat^{-T} \Delta R_2 \Delta R_1^T \Rhat^{-1} \right) \Rhat \delta  \xhat_i.
\end{equation*}
We thus have
\begin{equation}\label{eq:Rhat-T_k2_bound_interim}
    \Vert \Rhat^{-T} k_2 \Vert_2 \leq \left( 2 n^{1/2} \gamma_n  \kappa(\Rhat) + n \gamma_n^2  \kappa(\Rhat)^2 \right) \Vert \Rhat \delta  \xhat_i \Vert_2
\end{equation}
and need to bound $\Vert \Rhat \delta  \xhat_i \Vert_2$. To do this, we write \eqref{eq:trinag_solves_with_backward_error} as 
\begin{equation*}
     \Rhat^T \left(I +  \Rhat^{-T} \Delta R_1^T \right)\left( I + \Delta R_2 \Rhat^{-1} \right) \Rhat \delta \xhat_i = A^T \rhat_i + k_1.
\end{equation*}
Assuming that $\Vert \Rhat^{-T} \Delta R_1^T \Vert_2 < 1$ (which is essentially equivalent to $\kappa(A) < u^{-1}$), we use the approximation $\left(I +  \Rhat^{-T} \Delta R_1^T \right)^{-1} \approx I - \Rhat^{-T} \Delta R_1^T $ and an equivalent argument for $\left( I + \Delta R_2 \Rhat^{-1} \right)^{-1}$ to obtain
\begin{equation*}
    \Rhat \delta \xhat_i \approx \left( I - \Delta R_2 \Rhat^{-1} \right) \left(I - \Rhat^{-T} \Delta R_1^T \right) \left(   \Rhat^{-T} A^T \rhat_i +  \Rhat^{-T} k_1 \right).
\end{equation*}
Taking the norm and using $A^T r^* = 0$ gives
\begin{align*}
    \Vert \Rhat \delta \xhat_i \Vert_2 \leq & \, \left( 1 + n^{1/2} \gamma_n \kappa(\Rhat) \right)^2 \left(  \Vert   \Rhat^{-T} A^T \rhat_i \Vert_2  + \Vert  \Rhat^{-T} k_1 \Vert_2 \right) \\
    = & \, \left( 1 + n^{1/2} \gamma_n \kappa(\Rhat) \right)^2 \left(  \Vert   \bar{Q}^T (\rhat_i - r^*) \Vert_2  + \Vert  \Rhat^{-T} k_1 \Vert_2 \right).
\end{align*}
Combining this with \eqref{eq:Rhat-T_k2_bound_interim}, \eqref{eq:boubd_k1} and ignoring terms in $u^2\kappa(\Rhat)^2$ and $u u_r \kappa(\Rhat)^2$ gives
\begin{align}
    \Vert \Rhat^{-T} k_2 \Vert_2 \lesssim & \, 2 n^{1/2} \gamma_n  \kappa(\Rhat) \Vert \bar{Q}^T \Vert_2 \Vert \rhat_i - r^* \Vert_2  \nonumber \\
    \lesssim & \,  2 n^{1/2} \gamma_n  \kappa(\Rhat) \Vert r^* - (b  - A \xhat_i) \Vert_2. \label{eq:bound_Rhat_-T_k2}
\end{align}
Now using \eqref{eq:error_xtrue_x_i+1}, \eqref{eq:M_norm_bound}, \eqref{eq:f_bound}, \eqref{eq:boubd_k1}, \eqref{eq:bound_Rhat_-T_k2} and 
$\Vert e \Vert_2 \leq u \Vert \xhat_{i+1} \Vert_2 $ \cite[eq. 2.5]{high:ASNA2} we obtain
\begin{align*}
    \Vert x^* - \xhat_{i+1} \Vert_2 \leq & \, \Vert M ( x^* - \xhat_i )\Vert_2 + \Vert (\Rhat^T \Rhat )^{-1} \left( A^T f + k \right)\Vert_2 + \Vert e \Vert_2 \nonumber \\
    \leq & \, \Vert M \Vert_2 \Vert x^* - \xhat_i \Vert_2 +  \Vert \Rhat^{-1} \bar{Q}^T f \Vert_2 + \Vert \Rhat^{-1} \Rhat^{-T} k \Vert_2 + \Vert e \Vert_2 \nonumber \\
    \lesssim & \, 2 n^{1/2} \tg_{mn}  \kappa(A)^2 \Vert  x^* - \xhat_i \Vert_2 \nonumber \\
    & \, + n^{1/2} \tg_n^{(r)} \Vert \Rhat^{-1} \Vert_2 \left( \Vert A \Vert_2 \Vert \xhat_i \Vert_2 + \Vert b \Vert_2\right) \nonumber \\
    & \, + n^{1/2} \gamma_n^{(r)} \Vert \Rhat^{-1} \Vert_2^2 \Vert A \Vert_2  \left( \Vert A \Vert_2 \Vert \xhat_i \Vert_2 + \Vert b \Vert_2 \right) \nonumber \\
    & \, + 2 n^{1/2} \gamma_n \kappa(\Rhat) \Vert \Rhat^{-1} \Vert_2  \Vert r^* - (b  - A \xhat_i) \Vert_2 \nonumber \\
    & \, + u \Vert \xhat_{i+1} \Vert_2 \nonumber \\
    \leq & \, 2 n^{1/2} \tg_{mn} \kappa(A)^2 \Vert  x^* - \xhat_i \Vert_2 \nonumber \\
    & \, + 2 n^{1/2} \gamma_n  \kappa(A)^2 \frac{\Vert r^* - (b  - A \xhat_i) \Vert_2 }{\Vert A \Vert_2 }\nonumber \\
     & \, +  u \Vert \xhat_{i+1} \Vert_2 + n^{1/2} \gamma_n^{(r)} \kappa(A)^2 \frac{\Vert A \Vert_2 \Vert \xhat_i \Vert_2 + \Vert b \Vert_2}{\Vert A \Vert_2}. \nonumber 
\end{align*}
Using $r^* - (b  - A \xhat_i) = A(\xhat_i - x^*)$ gives
\begin{align}\label{eq:forward_error_change_seminormal}
     \Vert x^* - \xhat_{i+1} \Vert_2 \leq & \, n^{1/2} \tg_{mn}  \kappa(A)^2 \Vert  x^* - \xhat_i \Vert_2 \nonumber \\
     & \, + u \Vert \xhat_{i+1} \Vert_2 + n^{1/2} \gamma_n^{(r)} \kappa(A)^2 \frac{\Vert A \Vert_2 \Vert \xhat_i \Vert_2 + \Vert b \Vert_2}{\Vert A \Vert_2}.
\end{align}
\end{proof}

\subsubsection{Comparison with corrected semi-normal equations}
We now contrast the results in the previous section with the well-known results for the corrected semi-normal equations method of Bj\"{o}rck \cite{bjorck1987stability}, which is equivalent to running Algorithm \ref{alg:semi-normal_approach} with only one refinement iteration, computing $x_0$ by solving the semi-normal equations in precision $u$, and using $u_r=u$. 
Under these assumptions on $x_0$ but keeping $u_r$ general, combining the bound for $\Vert x^* - x_0 \Vert_2$ in \cite[eq. 3.11]{bjorck1987stability} with \eqref{eq:forward_error_change_seminormal} gives the following bound for the forward error after one iteration of LSIR:
\begin{align*}
    \Vert x^* - \xhat_1 \Vert_2 \leq & \, c(m,n) u^2 \kappa(A)^3 \left( \Vert  x^*  \Vert_2  + \frac{\Vert b \Vert_2}{\Vert A \Vert_2 }\right) \nonumber \\
     & \, + u \Vert \xhat_1 \Vert_2 + n^{1/2} \gamma_n^{(r)} \kappa(A)^2 \frac{\Vert A \Vert_2 \Vert \xhat_0 \Vert_2 + \Vert b \Vert_2}{\Vert A \Vert_2}.
\end{align*}
We can thus expect LSIR to reach $ \Vert x^* - \xhat_1 \Vert_2 / \Vert x^* \Vert_2 = \mathcal{O}(u)$ and thus converge after one iteration if $u_r = u^2$ and
\begin{equation}
    u \kappa(A)^3 \left( 2  + \frac{\Vert r^* \Vert_2}{\Vert A \Vert_2 \Vert  x^*  \Vert_2}\right) < 1,
    \label{eq:cond_csne}
\end{equation}
which is equivalent to $\kappa(A) < u^{-1/3}$ when $\frac{\Vert r^* \Vert_2}{\Vert A \Vert_2 \Vert  x^*  \Vert_2} = \mathcal{O}(1)$, and $\kappa(A) \frac{\Vert r^* \Vert_2^{1/3}}{\Vert A \Vert_2^{1/3} \Vert  x^*  \Vert_2^{1/3}} < u^{-1/3}$ for large $\Vert r^* \Vert_2$.

Note that the condition \eqref{eq:cond_csne} is stricter than the condition \eqref{eq:cond_sneir} in the case that $u_r=u^2$.
 If $u_r=u$, then the accuracy of Algorithm \ref{alg:semi-normal_approach} is limited by the last term in \eqref{eq:forward_error_change_seminormal}, which can be written as
\begin{equation*}
    u \kappa(A)^2 \left( 2  + \frac{\Vert r^* \Vert_2}{\Vert A \Vert_2 \Vert  x^*  \Vert_2}\right)
\end{equation*}
and we can expect $ \Vert x^* - \xhat_1 \Vert_2 / \Vert x^* \Vert_2 = \mathcal{O}(u)$ only if $\kappa(A) = \mathcal{O}(1)$ and $\frac{\Vert r^* \Vert_2}{\Vert A \Vert_2 \Vert  x^*  \Vert_2} = \mathcal{O}(1)$. In the first iterations of LSIR however we can expect the first term to dominate and thus we can get some improvement if \eqref{eq:kappaA_condition_seminormal} holds.

It is thus clear that in many cases, there is a benefit to allowing more than one iteration of refinement, as mentioned in \cite[Chapter 6]{bjorck1996numerical} for the uniform precision case. We note that there are no assumptions on how $x_0$ is computed in our general analysis in the previous section. 
If it is of poor quality, LSIR may indeed need a few iterations to correct the error.

We note that Bj\"{o}rck mentions performing fixed precision iterative refinement with normal equations, where the $R$ factor is obtained via a Cholesky decomposition of $A^TA$ \cite[Chapter 6]{bjorck1996numerical}. Our numerical experiments (not reported here) show that two precision LSIR with normal equations stops working for problems with slightly smaller condition number than in experiments with LSIR with semi-normal equations in Section~\ref{sec:numerics}.


\section{Augmented system approach}\label{sec:augmented}
Motivated by the sensitivity of the LS system approach to the size of the residual, Bj\"{o}rck \cite{bjorck1967iterative} exploits the fact that $x$ is the solution of an overdetermined least-squares problem \eqref{eq:LS} if and only if it satisfies the orthogonality property \cite[Theorem 1.1.2]{bjorck1996numerical} 
\begin{equation*}
    A^T r = 0. \label{eq:orthog_prop}
\end{equation*}
Bj\"{o}rck proposes performing iterative refinement on an augmented $(m +n) \times (m+n)$ linear system of equations
\begin{equation}\label{eq:augmented_system}
\underbrace{\begin{bmatrix}
 I & A\\ A^T & 0
\end{bmatrix}}_{\Atld}
\begin{bmatrix}
r\\  x
\end{bmatrix}
=
\begin{bmatrix}
 b \\0
\end{bmatrix},    
\end{equation}
which is equivalent to the normal equations. 
We describe the approach in Algorithm~\ref{alg:augmented_approach}; see \cite{demmel2009extra} for a variant of this algorithm and some implementation details. Note that both $r$ and $x$ are refined. In order to compare this approach with the LS system approach further, in the following sections we explore the conditions for recognizing the correct solution $x_i$ and the available convergence guarantees when the correction equation for the system \eqref{eq:augmented_system} is solved via a QR factorization of $A$. A general stability analysis when \eqref{eq:augmented_system} is solved via different QR factorizations is provided in \cite{bjorck1994solution}.

\begin{algorithm}
\caption{Two-precision LSIR based on the augmented system approach}\label{alg:augmented_approach}
\algorithmicrequire  $A \in \mathbb{R}^{m \times n}$, $b \in \mathbb{R}^m$, initial solution $x_0$, precisions $u_r$ and $u$ where $u_r \leq u^2$ \\
\algorithmicensure approximate solution $x$
\begin{algorithmic}[1]
\State $i=0$
\State Compute $r_0 = b - A x_0$ \Comment{$u$ }
\While{not converged}
\State Compute    $
        \begin{bmatrix}
 f_i \\g_i
\end{bmatrix}
= \begin{bmatrix}
 b \\0
\end{bmatrix}
-
\begin{bmatrix}
 I & A\\ A^T & 0
\end{bmatrix}
\begin{bmatrix}
r_i \\  x_i
\end{bmatrix}
$ \Comment{$u_r$}
\State Solve \begin{equation}\label{eq:correction_equation_augmented}
\begin{bmatrix}
 I & A\\ A^T & 0
\end{bmatrix}
\begin{bmatrix}
\delta r_i \\ \delta x_i
\end{bmatrix}
=
\begin{bmatrix}
 f_i \\ g_i
\end{bmatrix}
\end{equation} \Comment{$u$ }
\State Update $
\begin{bmatrix}
r_{i+1} \\  x_{i+1} 
\end{bmatrix} = 
\begin{bmatrix}
r_i \\  x_i 
\end{bmatrix}
+ 
\begin{bmatrix}
\delta r_i \\ \delta x_i 
\end{bmatrix} $ \Comment{$u$} 
\State $i=i+1$
\EndWhile
\end{algorithmic}
\end{algorithm}

We note that Bj\"{o}rck discusses other LSIR approaches based on the remainders $f_i$ and $g_i$ as in Algorithm~\ref{alg:augmented_approach} and using either the semi-normal equations, where the triangular factor is obtained via QR, or the normal equations method employing the Cholesky decomposition \cite{bjorck1978comment}. A detailed convergence analysis of these schemes is not provided in \cite{bjorck1978comment} and our numerical experiments, which are not reported here, show that it performs similarly to the semi-normal equations approach described in the previous section. We thus do not explore this approach further in this note. 

\subsection{Solving via QR decomposition}\label{sec:augmented_solving_via_qr}
We now detail how the augmented system can be solved via a QR decomposition of $A$ and state bounds for the forward error in the update to the solution $\delta x$ and the update to the residual $\delta r$.

Let $A = Q \begin{pmatrix}
    R \\ 0
\end{pmatrix} = Q \widetilde{R}$
be the QR decomposition of $A$, where $Q \in \mathbb{R}^{m \times m}$ and $R \in \mathbb{R}^{n \times n}$. We can write the block $LDL^T$ decomposition of $\Atld$ as
\begin{equation*}
    \Atld = \underbrace{\begin{pmatrix}
        I & 0 \\  
        \widetilde{R}^T Q^T & I
    \end{pmatrix}}_L 
    \underbrace{\begin{pmatrix}
        I & 0 \\ 
    0 & - R^T R
    \end{pmatrix}}_D
    \begin{pmatrix}
        I & Q  \widetilde{R}\\ 
        0 & I 
    \end{pmatrix}. 
\end{equation*}
Then the solution of \eqref{eq:correction_equation_augmented} is 
\begin{align}
    \begin{bmatrix}
\delta r_i \\ \delta x_i
\end{bmatrix}
=  & \, L^{-T} D^{-1} L^{-1} 
\begin{bmatrix}
 f_i \\ g_i
\end{bmatrix} \nonumber \\
=  & \, \begin{pmatrix}
        I & - Q \widetilde{R} \\  
       0 & I
    \end{pmatrix} 
    \begin{pmatrix}
        I & 0 \\ 
    0 & - R^{-1} R^{-T}
    \end{pmatrix}
    \begin{pmatrix}
        I & 0 \\  
       - \widetilde{R}^T Q^T & I
    \end{pmatrix} 
    \begin{bmatrix}
 f_i \\ g_i
\end{bmatrix} \nonumber \\
=  & \, \begin{bmatrix}
Q \begin{pmatrix}
    0 & 0 \\ 0 & I
\end{pmatrix} Q^T f_i  + Q \begin{pmatrix}
     R^{-T} \\ 0
 \end{pmatrix} g_i \\ \begin{pmatrix}
     R^{-1} & 0
 \end{pmatrix} Q^T f_i - R^{-1} R^{-T}  g_i
\end{bmatrix}, \label{eq:augmented_exact_update}
\end{align}
which can be computed in the following procedure proposed in \cite{bjorck1967iterative}:
\begin{align}
    h = & \, R^{-T} g_i, \label{eq:solve_augmented_QR_first}\\
    k = &  \,Q^T f_i = \begin{bmatrix}
        k_1 \\ k_2
    \end{bmatrix} \\
    \delta r = & \, Q \begin{bmatrix}
        h \\ k_2
    \end{bmatrix}, \\
    \delta x = &  \, R^{-1}(k_1 - h). \label{eq:solve_augmented_QR_last}
\end{align}
The forward error bounds for both $\delta x$ and $\delta r$ are stated in the following lemma.
\begin{lemma}
    Let $A \in \mathbb{R}^{m \times n}$ with $n \leq m$ be full rank and the system \eqref{eq:correction_equation_augmented} is solved via a Householder QR factorization of $A$ as detailed in \eqref{eq:solve_augmented_QR_first}-\eqref{eq:solve_augmented_QR_last}.
    Then if \[ \beta = (\sqrt{2}+1) m n^{3/2} \tg_m \kappa(A) < 1, \] we set \[\chi = (1 - \beta)^{-1/2} \kappa(A)\]    
    and we have 
\begin{equation}
    \Vert A \Vert_2  \Vert \delta \xhat - \delta x \Vert_2 
    \leq \, m n^{3/2} \tg_m \chi \left( \chi (\Vert \delta r \Vert_2 + \Vert \delta \rhat \Vert_2 ) +  \Vert A \Vert_2 \Vert \delta x \Vert_2 + \Vert \fhat \Vert_2 + \Vert \delta \rhat \Vert_2 \right) \label{eq:fe_bound_dx_bjorck}
\end{equation}
and
\begin{equation}
     \Vert \delta \rhat - \delta r \Vert_2 
     \leq  \, m n^{3/2} \tg_m \left( \chi (\Vert \delta r \Vert_2 + \Vert \delta \rhat \Vert_2 ) + \Vert A \Vert_2 \Vert \delta x \Vert_2 + \Vert \fhat \Vert_2 + \Vert \delta \rhat \Vert_2 \right). \label{eq:fe_bound_dr_bjorck}
\end{equation} 
\end{lemma}
\begin{proof}
    The proof follows from combining the perturbation result for a saddle point system as in \eqref{eq:correction_equation_augmented} in \cite[Theorem 1]{bjorck1967iterative} with the result for the finite precision analysis of \eqref{eq:solve_augmented_QR_first}-\eqref{eq:solve_augmented_QR_last} in \cite[Theorem 20.4]{high:ASNA2}.
\end{proof}

\subsection{Correction to true solution}
We wish to understand if the augmented system approach avoids the restrictive condition \eqref{eq:gol_wilk_condition} for recognizing the correct solution in the LS system approach. Assume that $x_i = x^*$. In order to determine what the ideal update should be, we first ignore the errors in computing $f_i$ and $g_i$ and in the solve. 
Then we have
\begin{equation*}
     \begin{bmatrix}
 f_i \\g_i
\end{bmatrix}
= \begin{bmatrix}
 b \\0
\end{bmatrix}
-
\begin{bmatrix}
 I & A\\ A^T & 0
\end{bmatrix}
\begin{bmatrix}
r_i \\  x^*
\end{bmatrix} = \begin{bmatrix}
 r^* - r_i \\- A^T r_i
\end{bmatrix}.
\end{equation*}
Combining this with \eqref{eq:augmented_exact_update} gives
\begin{equation}\label{eq:dr_dx_explicit_exact_x}
    \begin{bmatrix}
\delta r_i \\ \delta x_i
\end{bmatrix} = \begin{bmatrix}
Q \begin{pmatrix}
    0 & 0 \\ 0 & I
\end{pmatrix} Q^T ( r^* - r_i)  - Q \begin{pmatrix}
     R^{-T} \\ 0
 \end{pmatrix} A^T r_i \\ \begin{pmatrix}
     R^{-1} & 0
 \end{pmatrix} Q^T ( r^* - r_i) + R^{-1} R^{-T}  A^T r_i
\end{bmatrix} = \begin{bmatrix}
r^* - r_i \\ 0
\end{bmatrix},
\end{equation}
where we use $Q  =  \begin{bmatrix}
    Q_1 & Q_2
\end{bmatrix} $ with $Q_1 \in \mathbb{R}^{m \times n}$ and $Q_2 \in \mathbb{R}^{m \times (m-n)}$ and the fact that $Q_1^T r^* = 0$. Thus the true solution to the correction equation is $\delta x_i = 0$.

We hence want the computed solution to satisfy $\Vert \delta \xhat_i  \Vert_2 / \Vert x^* \Vert_2 = \mathcal{O}(u)$.
Using \eqref{eq:fe_bound_dx_bjorck} with $\delta x_i = 0$ and the fact that $\chi >1 $, the condition for recognizing the correct solution is
\begin{equation}\label{eq:condition_bjorck}
     \chi^2\frac{\Vert \delta r_i \Vert_2 + 2 \Vert \delta \rhat_i \Vert_2  +   \Vert \fhat_i \Vert_2}{\Vert A \Vert_2 \Vert x^* \Vert_2}  < 1.
\end{equation}
Comparing the condition \eqref{eq:condition_bjorck} with the equivalent condition for the LS approach in \eqref{eq:gol_wilk_condition}, we observe that the numerator of \eqref{eq:condition_bjorck} includes the \emph{corrections} to the residual whereas the numerator of \eqref{eq:gol_wilk_condition} includes the residual itself. The corrections to the residual can be of a similar size as the residual in the first refinement iteration if a very poor quality initial solution $x_0$ is used, but the error in the residual is reduced in the following iterations; see the comment on the initial rate of improvement in the following section. Regarding the relevant condition for the semi-normal equations case \eqref{eq:condition_recognise_semi_normal}, note that the right-hand side of \eqref{eq:condition_recognise_semi_normal} includes the unit roundoff whereas \eqref{eq:condition_bjorck} is independent of it. Thus \eqref{eq:condition_bjorck} is more restrictive than \eqref{eq:condition_recognise_semi_normal} unless $\Vert \delta r_i \Vert_2 \approx u$.

\subsection{Initial rate of convergence and sufficient conditions for convergence}
Bj\"{o}rck analyzes the convergence of LSIR when the augmented system is solved via the Householder QR approach detailed in Section~\ref{sec:augmented_solving_via_qr} \cite[Section 6]{bjorck1967iterative}. It is assumed that the computations are performed in single precision and the inner products are accumulated in double precision throughout the computations except when updating $r_{i+1}$ and $x_{i+1}$. Bj\"{o}rck points out that using higher precision is only essential in the computation of the residuals and if the accumulation is omitted from other computations, then the error bounds would increase by at most a factor of $m$. 

Bj\"{o}rck shows that the initial rate of improvement is proportional to the term 
\begin{equation*}
    \rho_i = \frac{\Vert \hat{x}_i - x^* \Vert_2}{\Vert \hat{x}_{i-1} - x^* \Vert_2} < c(m,n) u \kappa(A),
\end{equation*}
where $c(m,n)$ is a constant. The bound for the initial rate of improvement for the residual is the same. Note that the bound does not depend on the size of the residual and thus it is possible to get some improvement in the first iterations even for problems with large $\Vert r^* \Vert_2$. The dependency on $\kappa(A)$ rather than $\kappa(A)^2$ also shows that initial improvement is possible even for highly ill-conditioned problems.

However, $\kappa(A)^2$ arises in the conditions that guarantee a small relative error in the solution. It is shown that 
\begin{equation*}\label{eq:lim_x}
    \lim_{i \rightarrow \infty} \Vert \xhat_i - x^* \Vert_2 = \mathcal{O}(u) \Vert x^* \Vert_2 
\end{equation*}
is guaranteed to hold if
\begin{equation}\label{eq:condition_x_convergence}
     \chi^2 \frac{  \Vert r^* \Vert_2}{\Vert A \Vert_2 \Vert x^* \Vert_2 } \leq \frac{1}{{c}_1(m,n,u,u_r)}, 
\end{equation}
where ${c}_1(m,n,u,u_r)$ is a constant that depends on $m$, $n$, $u$, and $u_r$.
Ignoring the dimensional constants, we can approximate ${c}_1(m,n,u,u_r)$ by its leading term $u$, and also approximate $\chi$ by $\kappa(A)$ to obtain the simplified convergence condition
\begin{equation}\label{eq:condition_x_convergence_simplified}
    \kappa(A)^2 \frac{  \Vert r^* \Vert_2}{\Vert A \Vert_2 \Vert x^* \Vert_2 } \leq u^{-1}. 
\end{equation}
Bj\"{o}rck also shows that a sufficient condition for
\begin{equation*}\label{eq:lim_residual}
    \lim_{i \rightarrow \infty} \Vert \rhat_i - r^* \Vert_2 = \mathcal{O}(u) \Vert r^* \Vert_2 
\end{equation*}
to hold is
\begin{equation}\label{eq:condition_r_convergence}
    \frac{\Vert A \Vert_2 \Vert x^* \Vert_2 }{\Vert r^* \Vert_2} \leq {c}_2(m,n,u,u_r),
\end{equation}
where ${c}_2(m,n,u,u_r)$ is another constant that depends on $m$, $n$, $u$, and $u_r$.

The following sensitivities of the augmented system approach to LSIR can be observed.
\begin{itemize}
    \item The right-hand side of \eqref{eq:condition_x_convergence} is inversely proportional to $\chi^2$ which is proportional to $\kappa(A)^2$. We can thus expect the procedure to become less reliable in correcting $x$ when $\kappa(A)$ increases.
    \item The left-hand side of \eqref{eq:condition_r_convergence} grows when $\Vert r^* \Vert_2$ is reduced. It is thus suggested that the refinement of the residual is sensitive to the size of it and if the residual norm is very small then the procedure may fail to correct it to the required accuracy.
\end{itemize}

\textbf{NB:} Bj\"{o}rck points out that for any matrix $A$ there exist right-hand sides $b$ such that either \eqref{eq:condition_x_convergence} or \eqref{eq:condition_r_convergence} does not hold. However, at least one of \eqref{eq:condition_x_convergence} or \eqref{eq:condition_r_convergence} \emph{always} holds. 

We now consider how the above convergence guarantees compare to the guarantees for other approaches. Recall that there are no convergence guarantees for the LS approach. Convergence of the solution of the semi-normal equations approach is guaranteed if \eqref{eq:cond_sneir} holds. When $u_r=u^2$ the condition \eqref{eq:cond_sneir} is more limiting than \eqref{eq:condition_x_convergence}, because a small $\Vert r^* \Vert_2$ allows \eqref{eq:condition_x_convergence} to hold for problems with  $\kappa(A)>u^{-1/2}$. If we have $\kappa(A)>u^{-1/2}$ then we need to set $u_r < u^2$ for \eqref{eq:cond_sneir} to hold independent of $\Vert r^* \Vert_2$.

\subsection{Scaling of the augmented system}\label{sec:scaling_augmented_system}

Theory for the iterative refinement of linear systems of equations can also be applied to analyze the refinement of \eqref{eq:augmented_system}. When \eqref{eq:correction_equation_augmented} in Algorithm~\ref{alg:augmented_approach} is solved via backward and forward substitution using the $LDL^T$ decomposition of $\Atld$ computed in precision $u$ or using the QR factors of $A$ as described in Section~\ref{sec:augmented_solving_via_qr}, then the forward and backward error of the augmented system is guaranteed to converge if $\kappa(\Atld)$ is sufficiently less than $u^{-1}$; see, e.g., \cite[Chapters 12 and 20.5]{high:ASNA2}. The convergence is thus independent of the size of the residual, but depends on $\kappa(\Atld)$ instead of $\kappa(A)$. Bj\"{o}rck \cite{bjorck1967iterative} showed that $\kappa(\Atld)$ relates to the singular values of $A$ and can be changed significantly with the scaling 
\begin{equation*}
   \begin{pmatrix}
        I & 0 \\ 0 & \alpha^{-1} I
    \end{pmatrix}
    \begin{pmatrix}
        I & A \\ A^T & 0
    \end{pmatrix}
       \begin{pmatrix}
        \alpha I & 0 \\ 0 & I
    \end{pmatrix} =     \begin{pmatrix}
       \alpha  I & A \\ A^T & 0
    \end{pmatrix} = \Atld_{\alpha},
\end{equation*}
where $\alpha \in \mathbb{R}$ is positive. 
Then
\begin{equation*}
    \kappa(\Atld_{\alpha}) = \frac{ \alpha + \sqrt{ \alpha^2 + 4 \sigma_{max}(A)^2}}{\min \{ 2 \alpha, \sqrt{ \alpha^2 + 4 \sigma_{min}(A)^2} -  \alpha\}},
\end{equation*}
where $\sigma_{max}(A)$ and $\sigma_{min}(A)$ are the largest and smallest singular values of $A$, respectively. If we set $\alpha = 2^{-1/2}\sigma_{min}(A)$, then $ \kappa(\Atld_{\alpha})$ reaches its minimum value and $\kappa(A) \leq \kappa(\Atld_{\alpha}) \leq 2 \kappa(A)$; note that $\sigma_{min}(A)$ is expensive to compute. The case without scaling, that is, setting $\alpha = 1$ gives
\begin{equation*}
    \kappa(\Atld_{\alpha}) = \frac{ 1 + \sqrt{ 1 + 4 \sigma_{max}(A)^2}}{\min \{ 2 , \sqrt{ 1 + 4 \sigma_{min}(A)^2} -  1\}},
\end{equation*}
which can be close to $\kappa(A)^2$ if $\sigma_{min}(A) \ll 1$. 

Arioli, Duff, and de Rijk \cite{arioli1989augmented} note that scaling not only reduces $\kappa(\Atld)$, but also may improve the quality of the computed $LDL^T$ factorization. They show that scaling (even if $\alpha$ is set to a non-optimal value) can enable convergence for some ill-conditioned problems or at least improve the quality of the computed solution if IR does not converge. The authors also propose a heuristic for choosing $\alpha$.

\section{Iterative solvers}\label{sec:iterative}

The theoretical analysis in the previous sections assumes that the least-squares problems are solved via Householder QR decomposition. The cost of computing this decomposition may be prohibitive when $A$ is very large or sparse. In this case, iterative solvers may be used for the correction equations. 

Irrespective of the solver choice, we are not aware of the existence of conditions that guarantee convergence for the LS system approach to LSIR, and as we have already mentioned, Golub and Wilkinson note that the method will not converge for some right-hand sides. If the LS approach is chosen regardless, then LSQR, which is a popular iterative method and only accesses $A$ via a procedure for computing matrix-vector products with $A$ and $A^T$, may be preferred by the user \cite{paige1982lsqr}. 

The semi-normal equations approach is essentially direct and requires a good quality $R$ factor of the QR decomposition. We note that the analysis can be applicable when $R$ is computed via methods other than Householder QR as long as they are backward stable, that is, \eqref{eq:HHQR_pert} and \eqref{eq:E_norm} hold. Iterative solvers could, however, be used to find $\delta x_i$ via the normal equations $A^T A \delta x_i = A^T b$ instead of the semi-normal equations in step~\ref{eq:correction_step_semi_normal} of Algorithm~\ref{alg:semi-normal_approach}. $A^TA$ is a symmetric positive definite matrix and thus the CG method can be used (see e.g. \cite{saad2003iterative}), although for CG there is no guarantee of backward stability in the usual sense; this approach may be preferable over the LS approach with LSQR because using $A^T r_i$ as the right-hand side in the correction equation reduces the impact of a large residual. Backward stability can be guaranteed for reasonably well-conditioned $A$ if preconditioned GMRES is used; see the discussion in \cite[p. A268]{higham2021exploiting}.

The augmented system approach with an iterative solver for the correction equations has been studied before. Various solvers and preconditioners (possibly computed in a low precision $u_f$) can be used 
as long as the solution satisfies certain conditions \cite{carson2020three}; the authors in \cite{carson2020three} use left-preconditioned GMRES. Scott and T\r{u}ma \cite{scott2022computational} use this approach with split-preconditioned GMRES and MINRES to solve sparse-dense LS problems. As we discussed in Section~\ref{sec:scaling_augmented_system}, convergence theory for iterative refinement for linear systems of equations can be applied in this case and thus convergence can be guaranteed depending on the condition number of the (potentially preconditioned) augmented matrix. The performance of the iterative solver itself will also depend on the properties of this matrix.  
We thus note that computing a good scaling factor and preconditioner is paramount to ensure good performance of the scheme. 

\section{Combining least-squares and augmented system approaches}\label{sec:combining_ls_augmented}

We further explore whether both $x$ and $r$ can be updated by solving least-squares problems, which also allows the use of various iterative solvers. We do this by considering the inverse of the augmented system matrix
\begin{equation*}
    \Atld^{-1} = \begin{pmatrix}
        P_N & (A^\dagger)^T \\
        A^\dagger & - (A^T A)^{-1}
    \end{pmatrix},
\end{equation*}
where $P_N = I - A A^{\dagger}$ is an orthogonal projector onto the null-space of $A^T$. The update in \eqref{eq:correction_equation_augmented} is then
\begin{equation}\label{eq:corrections_mixed_approach}
    \begin{bmatrix}
\delta r_i \\ \delta x_i
\end{bmatrix}
= \begin{pmatrix}
        P_N f_i  + (A^\dagger)^T g_i\\
        A^\dagger f_i - (A^T A)^{-1} g_i
    \end{pmatrix}.
\end{equation}
We observe that $(A^\dagger)^T  g_i$ is a minimum-norm solution to the under-determined least-squares problem 
\begin{equation}\label{eq:ls_problem_dr1}
    \min_{y_1} \Vert g_i - A^T y_1 \Vert_2.
\end{equation}
Note that $A^\dagger f_i$ is the solution of the least-squares problem
\begin{equation}\label{eq:ls_problem_dx1}
    \min_{y_2} \Vert f_i - A y_2 \Vert_2,
\end{equation}
and $- (A^T A)^{-1} g_i = - R^{-1} R^{-T} g_i$, where $R$ comes from the QR decomposition of $A$, is the solution of 
\begin{equation*}
    R^T R y_3 = - g_i, 
\end{equation*}
which is equivalent to the semi-normal equation approach to solving 
\begin{equation}\label{eq:ls_problem_dx2}
  \min_{y_3} \Vert r_i - A y_3 \Vert_2  
\end{equation}
when the right-hand side of the semi-normal equations is computed in a higher precision. Note that, however, since $r_i$ is used instead of $A^T r_i$, it is likely that the method is sensitive to the size of $\Vert r \Vert_2$.

We can thus obtain $\delta r_i$ and $\delta x_i$ in \eqref{eq:corrections_mixed_approach} via solving the least-squares problems \eqref{eq:ls_problem_dr1}, \eqref{eq:ls_problem_dx1}, and \eqref{eq:ls_problem_dx2} employing a least-squares solver of choice, and computing the projection $P_N f_i$. Note that the latter can be computed as $P_N f_i = ( I - A A^{\dagger})f_i = f_i - Ay_2 $, where $y_2$ is the solution to \eqref{eq:ls_problem_dx1}, that is, $P_N f_i$ is the residual of \eqref{eq:ls_problem_dx1}. 
We summarize the procedure in Algorithm~\ref{alg:combined_approach}. Note that all the least-squares problems can be solved at the same time and a block least-squares solver may be employed.

\begin{algorithm}
\caption{Two-precision LSIR combining least-squares and augmented system approaches}\label{alg:combined_approach}
\algorithmicrequire  $A \in \mathbb{R}^{m \times n}$, $b \in \mathbb{R}^m$, initial solution $x_0\in \mathbb{R}^n$, projector $P_N$ onto the the null-space of $A^T$, precisions $u_r$ and $u$ where $u_r \leq u$ \\
\algorithmicensure approximate solution $x$
\begin{algorithmic}[1]
\State $i=0$
\State Compute $r_0 = b - A x_0$ \Comment{$u$ }
\While{not converged}
\State Compute    $
        \begin{bmatrix}
 f_i \\g_i
\end{bmatrix}
= \begin{bmatrix}
 b \\0
\end{bmatrix}
-
\begin{bmatrix}
 I & A\\ A^T & 0
\end{bmatrix}
\begin{bmatrix}
r_i \\  x_i
\end{bmatrix}
$ \Comment{$u_r$}
\State Solve $\min_{\delta x_i^{(1)}} \Vert f_i - A \delta x_i^{(1)} \Vert_2$ \Comment{$u$} 
\State Solve $ \min_{\delta x_i^{(2)}} \Vert r_i - A \delta x_i^{(2)} \Vert_2$ \Comment{$u$} 
\State Solve $ \min_{\delta r_i^{(1)}} \Vert g_i - A^T \delta r_i^{(1)} \Vert_2$ \Comment{$u$} \label{step:combined_underdetermined_LS}
\State Compute $ \delta r_i^{(2)} = P_N f_i$ \Comment{$u$} 
\State Update $
\begin{bmatrix}
r_{i+1} \\  x_{i+1} 
\end{bmatrix} = 
\begin{bmatrix}
r_i \\  x_i 
\end{bmatrix}
+ 
\begin{bmatrix}
\delta r_i^{(1)} + \delta r_i^{(2)} \\ \delta x_i^{(1)} + \delta x_i^{(2)}
\end{bmatrix} $ \Comment{$u$} 
\State $i=i+1$
\EndWhile
\end{algorithmic}
\end{algorithm}

\section{Numerical experiments}\label{sec:numerics}

In order to compare the various methods numerically and to confirm the behavior described by the analysis, we consider a simple class of least-squares problems using MATLAB R2021a\footnote{Our code is available at \url{https://github.com/dauzickaite/LSIR}}. A $10^3 \times 10$ coefficient matrix $A$ is generated as
\[ A = \text{gallery}(\text{'randsvd'},[1e3,1e1],k,3)\]
and we change the $\kappa(A)$ values (the $k$ parameter above). We generate the right-hand side $b$ as 
\[ b = mp(A) mp(y) + mp(e), \]
where $mp(.)$ simulates precision accurate to 64 decimal digits via the Advanpix toolbox \cite{advanpix}, $y $ is a normalized random vector with entries that are uniformly distributed on the interval $(0,1)$, and $e$ is such that $A^T e = 0$. We then cast $b$ to double or single. We vary the size of $e$ and thus obtain results for different residual norms. The true solution $x^*$ and true residual $r^*$ are computed as
\begin{align*}
    x^* = & \, mp(A) \backslash mp(b); \\
    r^* = & \, mp(b) - mp(A)*x^*.
\end{align*}
Note that because we are aiming for the unit roundoff level of accuracy, the results are very sensitive to the setup. We explore two sets of precisions $u_r$ for computing the residuals and $u$ for other computations: $u_r$ set to quadruple (simulated via Advanpix) and $u$ set to double, and $u_r$ set to double and $u$ set to single. The units roundoff for these IEEE floating point arithmetics are $9.6 \times 10^{-35}$ for quadruple (fp128), $1.1 \times 10^{-16}$ for double (fp64), and $6.0 \times 10^{-8}$ for single (fp32).

We declare that LSIR converged if \[  \Vert x_i - x^* \Vert / \Vert x^* \Vert \leq \tau.  \]
We set $\tau = 8 u$, where $u$ is the unit roundoff for double or single precision. If this condition is satisfied with $x_0$, then LSIR is not performed. LSIR is run for a maximum of 30 iterations. 

\subsection{Solving via QR}
We compute the full QR decomposition in the working precision $u$ as $[\widetilde{Q},\widetilde{R}] = qr(A)$ and set $R = \widetilde{R}(1:n,1:n)$ and $Q_1 = \widetilde{Q}(:,1:n)$. The initial estimate is $x_0 = R \backslash (Q_1'*b)$, where $Q_1'$ denotes $Q_1^T$. We show when the condition for recognizing the true solution in the LS approach is satisfied in Figure~\ref{fig:recog_true_sol_criteria_LS_approach}. 
We use condition \eqref{eq:condition_x_convergence_simplified} to estimate for which instances the augmented system LSIR approach is expected to converge; see Figure~\ref{fig:x_convergence_condition_augmented}. Note that we expect the semi-normal equations approach to be insensitive to $\Vert r \Vert_2$ and converge for problems with $\kappa(A) < u^{-1/2}$.

\begin{figure}
    \centering
 \includegraphics[width=0.45\linewidth]{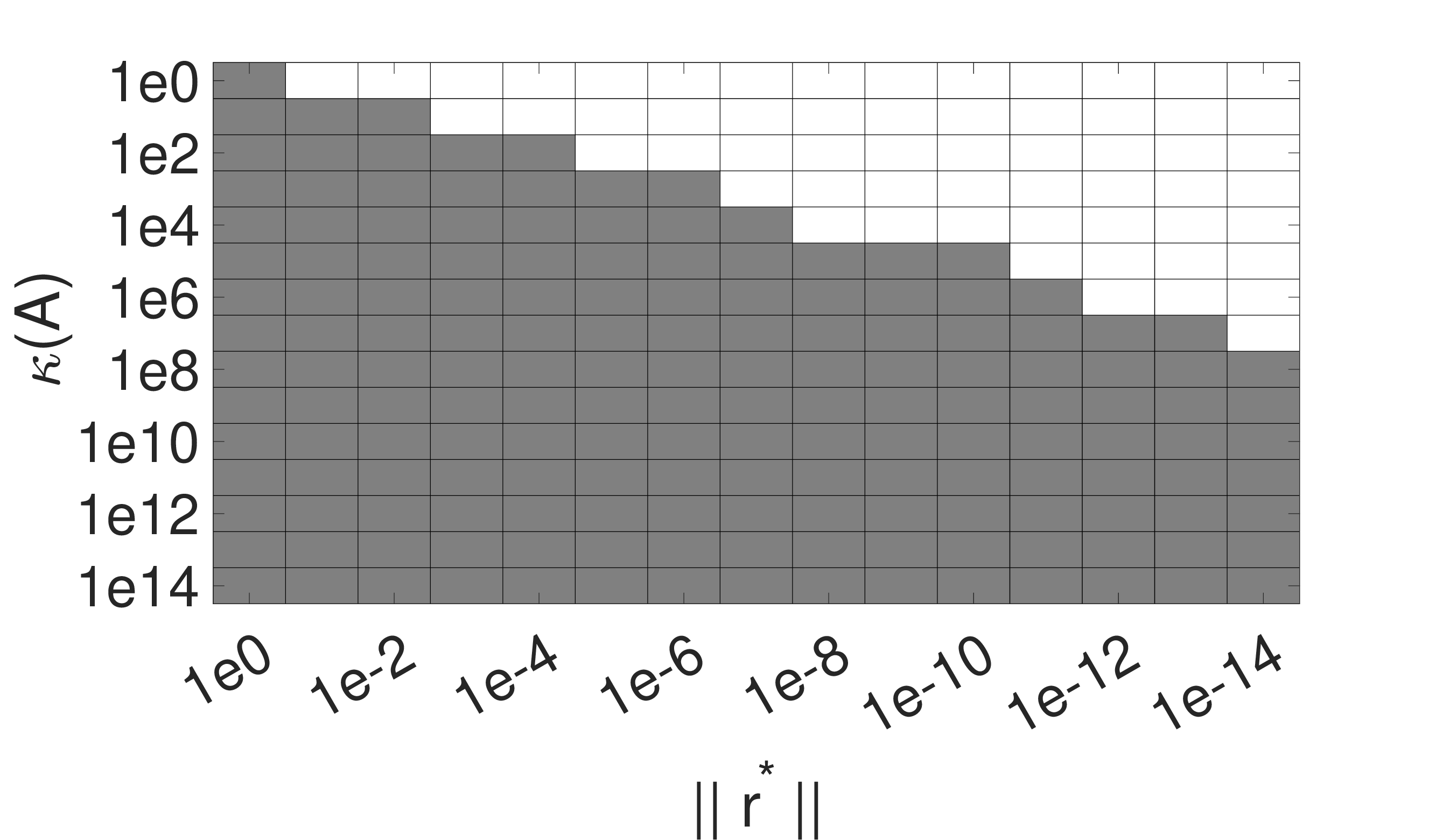}
    \caption{LS system approach. Criteria for the least-squares system approach to recognize a true solution if it is found \eqref{eq:gol_wilk_condition} for a set of LS problems with varying $\kappa(A)$ and $\Vert r^* \Vert$. The white cells indicate that \eqref{eq:gol_wilk_condition} is satisfied.}
    \label{fig:recog_true_sol_criteria_LS_approach}
\end{figure}

\begin{figure}
    \centering
\begin{subfigure}[t]{0.45\linewidth}
  \centering
 \includegraphics[width=\linewidth]{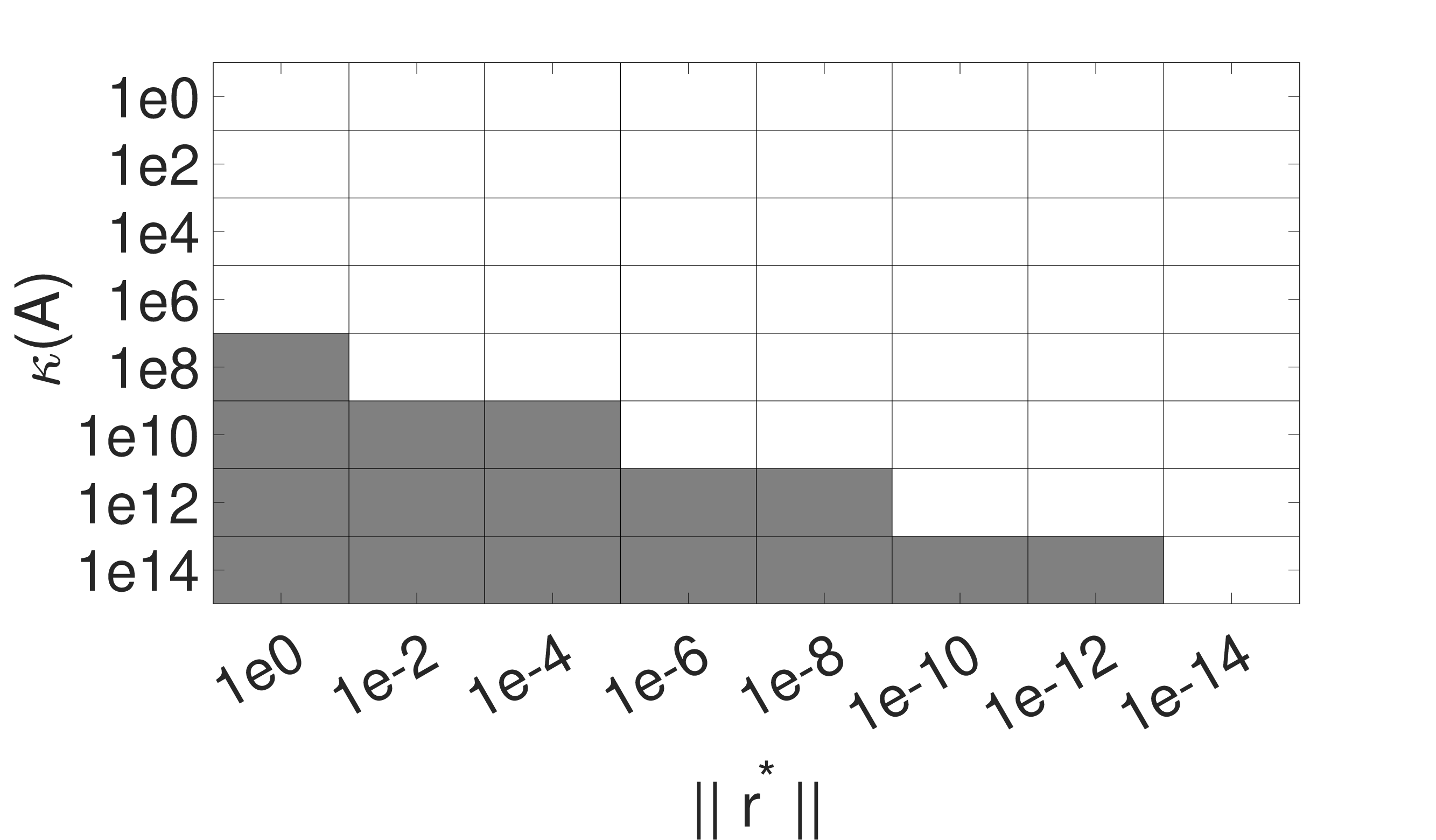}
  \caption{$u_r$ quad, $u$ double}
\end{subfigure}
\begin{subfigure}[t]{0.45\linewidth}
  \centering
 \includegraphics[width=\linewidth]{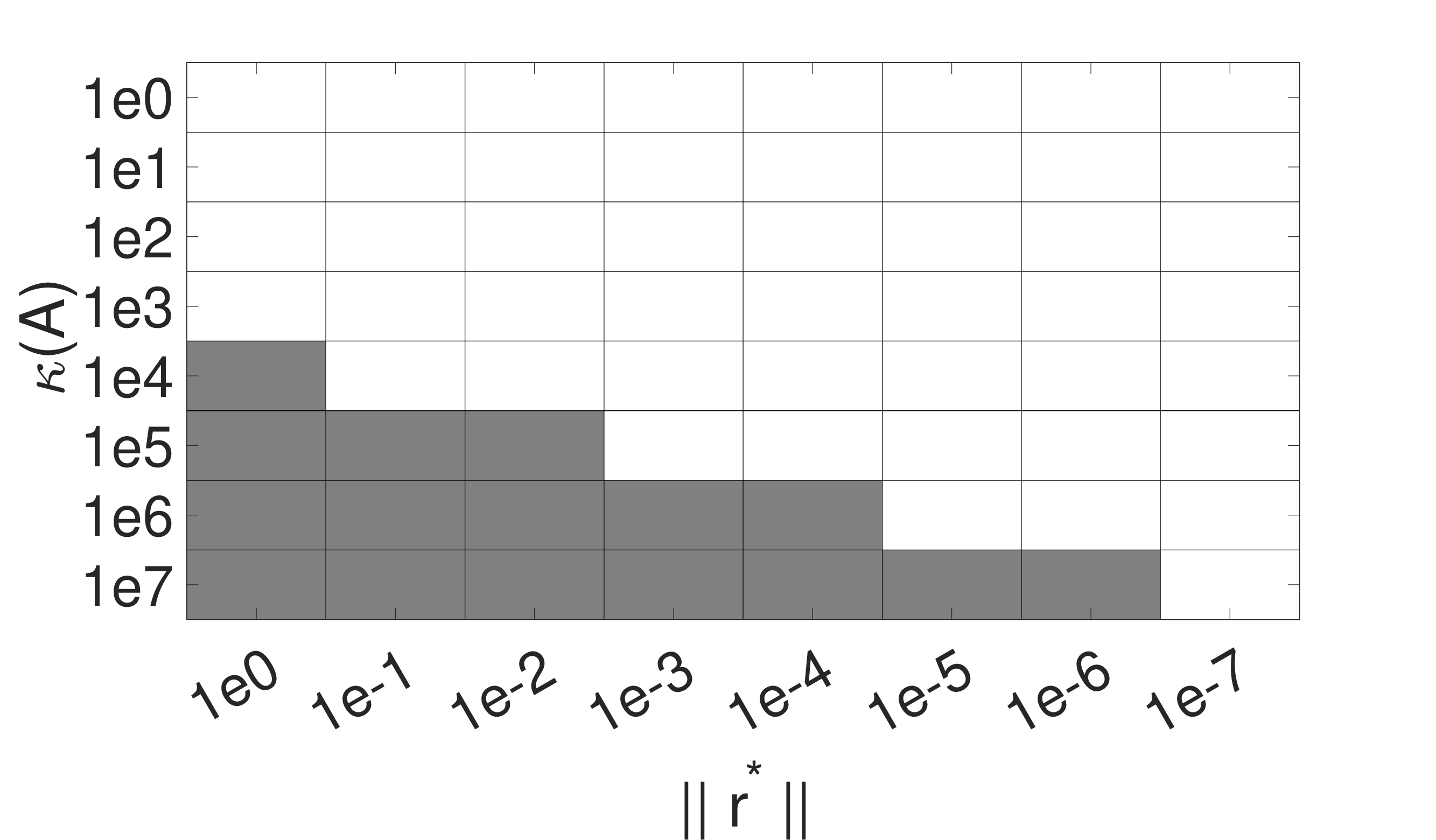}
  \caption{$u_r$ double, $u$ single}
\end{subfigure}
    \caption{Augmented system approach. LSIR convergence criteria \eqref{eq:condition_x_convergence_simplified} for a set of LS problems with varying $\kappa(A)$ and $\Vert r^* \Vert$ for two sets of precisions $u_r$ and $u$. The white cells indicate that \eqref{eq:condition_x_convergence_simplified} is satisfied.}
    \label{fig:x_convergence_condition_augmented}
\end{figure}

In Figures~\ref{fig:ir_iterations}, \ref{fig:x_error}, and \ref{fig:r_error}, we show the iteration count, and relative errors  $\Vert x_i - x^* \Vert / \Vert x^* \Vert  $  and $\Vert r_i - r^* \Vert / \Vert r^* \Vert $, respectively, for the LS, semi-normal equations, and augmented system approaches.
The augmented system approach converges in all the cases where the criteria \eqref{eq:condition_x_convergence_simplified} is satisfied (as shown in Figure~\ref{fig:x_convergence_condition_augmented}) and even for more ill-conditioned systems with larger residuals. As suggested by the theoretical analysis, the semi-normal equations approach converges for all problems when $\kappa(A)<u^{-1/2}$ (and even for some worse conditioned ones) and is not sensitive to the size of the residual. 

Comparing all three approaches, the augmented system approach is able to refine the solution $x$ for both more ill-conditioned problems and problems with larger true residual. The semi-normal equations approach outperforms the LS approach but not the augmented system approach. Note that in most cases very few LSIR iterations are needed; in general a larger number of iterations is needed when $\kappa(A)$ grows so large that the theoretical results cannot guarantee convergence. Convergence plots in Figures~\ref{fig:x_conv_qr} and \ref{fig:r_conv_qr} show that for large $\Vert r^* \Vert_2$ the semi-normal approach performs similar to the augmented system approach and for small $\Vert r^* \Vert_2$ the semi-normal approach gives results similar to the LS approach.

Our stopping criteria for LSIR is based on the error in $x$ only and we thus may stop the algorithm too early to refine $r$ to the attainable level. However, in the cases where the augmented system approach is unable 
to refine the solution $x$, the error in the relative residual reaches the $\mathcal{O}(u)$ level. The LS system approach mostly fails to refine the residual to such accuracy even when $\Vert x_i - x^* \Vert / \Vert x^* \Vert  $ reaches $\mathcal{O}(u)$. The semi-normal equations approach refines the residual for problems with small $\Vert r^* \Vert$. We can clearly observe that $r$ is harder to refine with all the methods when $\Vert r^* \Vert$ decreases. 

We have also performed numerical experiments where $n \ll m$ does \emph{not} hold, namely, we set $n=900$. The results are similar to the case reported above and agree with the theory, and we thus do not show them here. The augmented approach still outperforms what is expected from Figure~\ref{fig:x_convergence_condition_augmented}, though it appears to be slightly more sensitive to the residual size than when $n \ll m$. The latter point holds for the LS approach too. The semi-normal equations approach converges for $\kappa_2(A)<u^{-1/2}$ only.

\begin{figure}
    \centering
\begin{subfigure}[t]{0.45\linewidth}
  \centering
 \includegraphics[width=\linewidth]{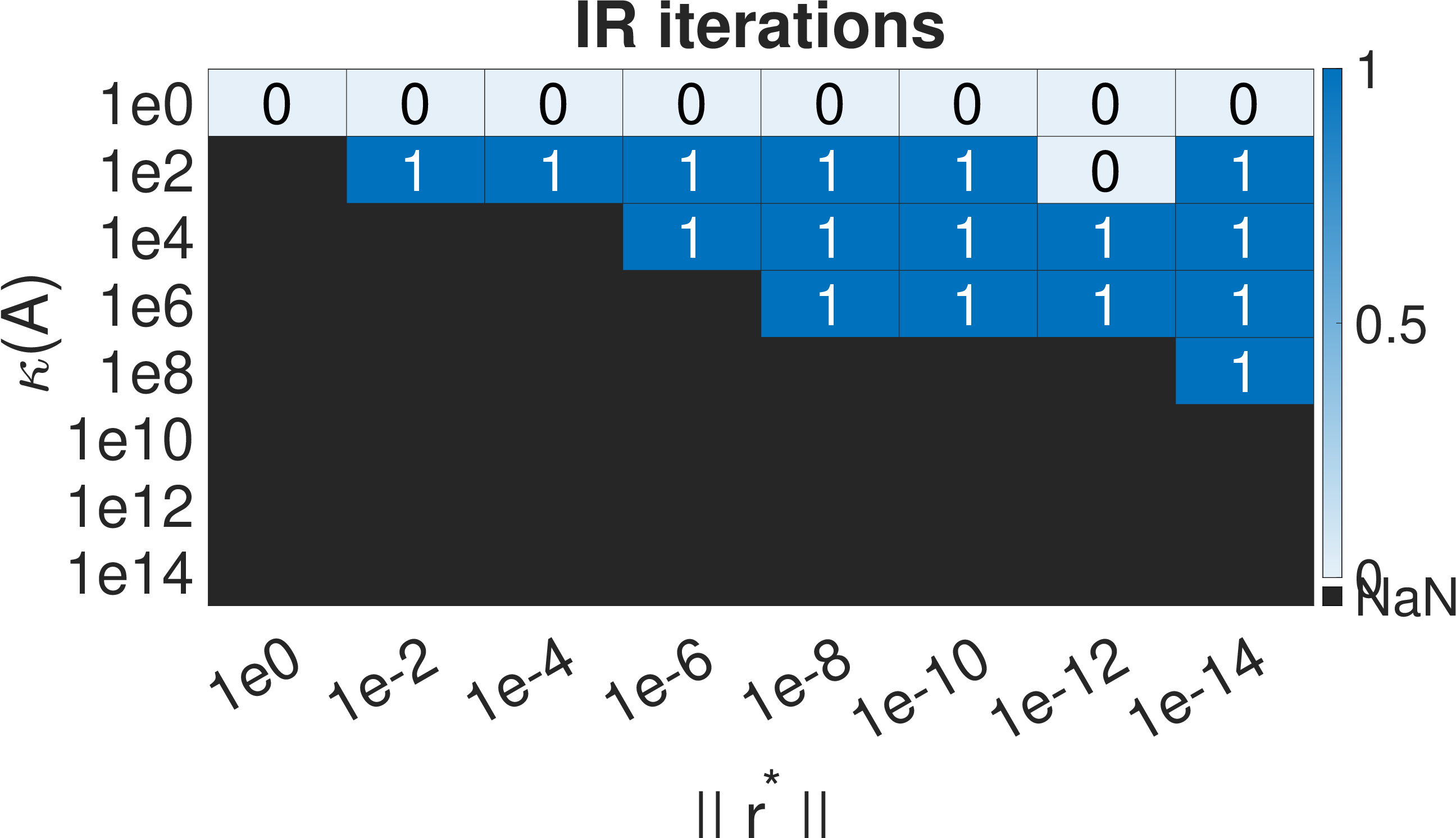}
  \caption{LS approach, (quad, double)}
\end{subfigure}
\begin{subfigure}[t]{0.45\linewidth}
  \centering
 \includegraphics[width=\linewidth]{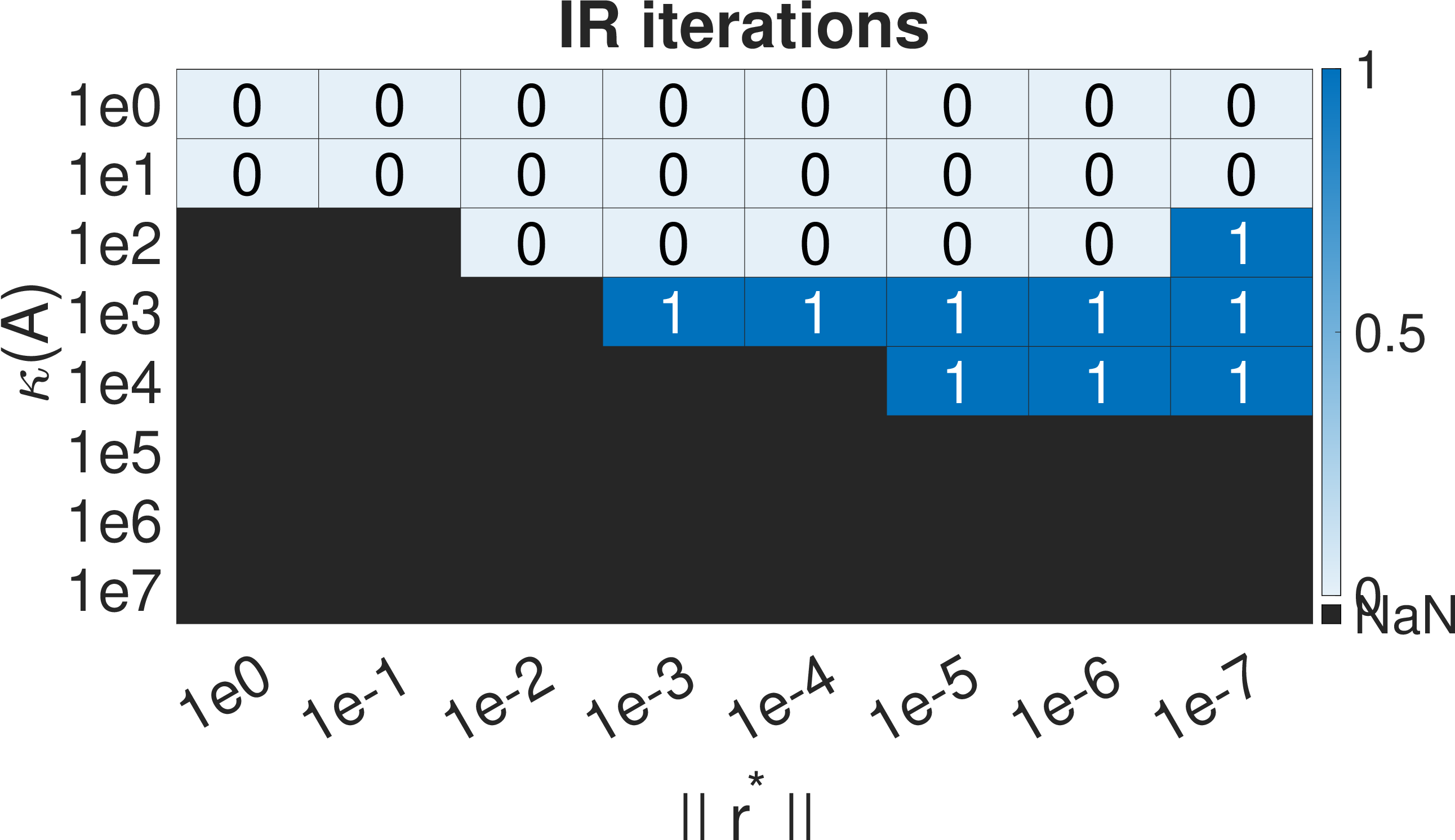}
  \caption{LS approach, (double, single)}
\end{subfigure}
\begin{subfigure}[t]{0.45\linewidth}
  \centering
 \includegraphics[width=\linewidth]{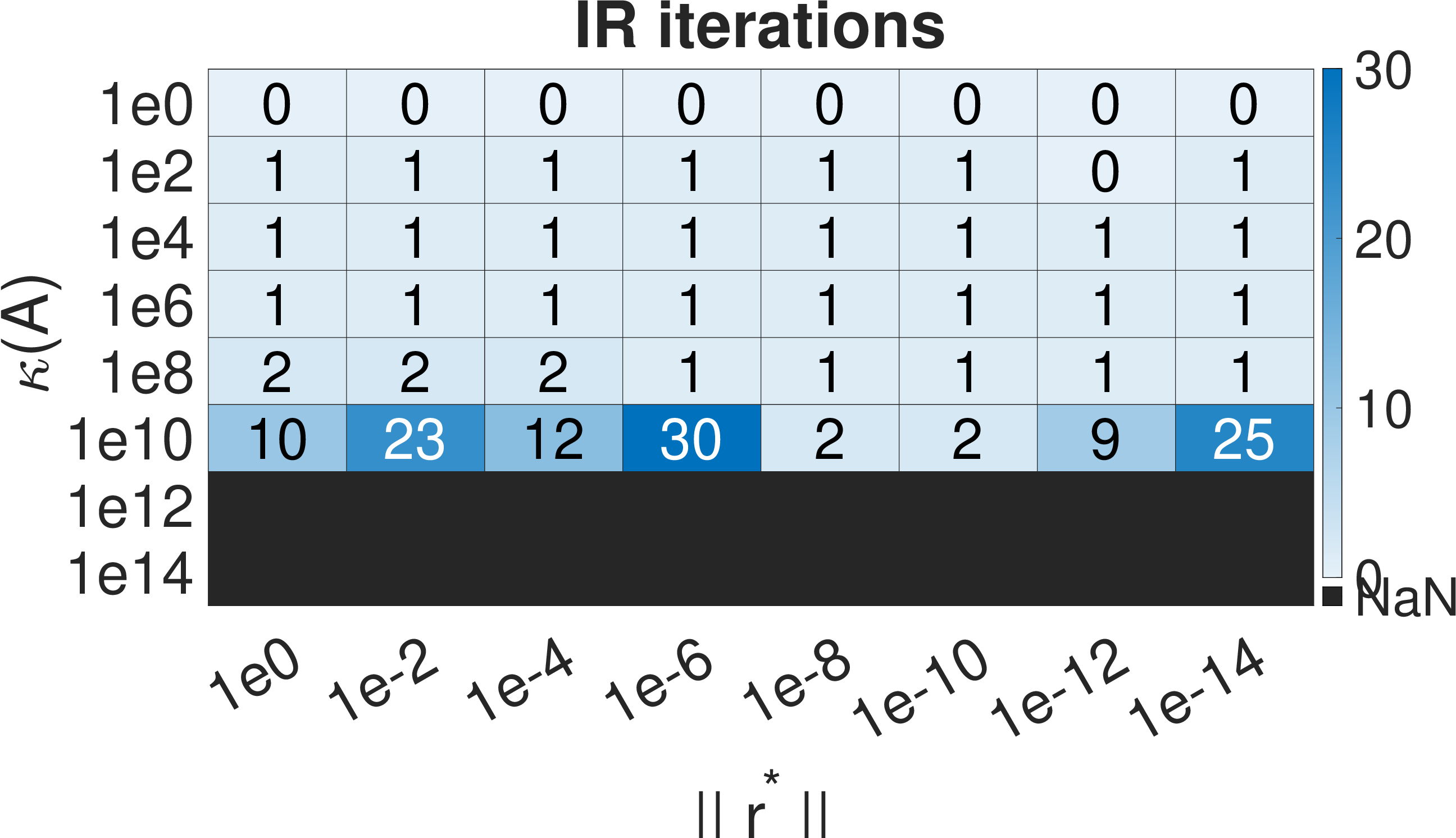}
  \caption{semi-normal, (quad, double)}
\end{subfigure}
\begin{subfigure}[t]{0.45\linewidth}
  \centering
 \includegraphics[width=\linewidth]{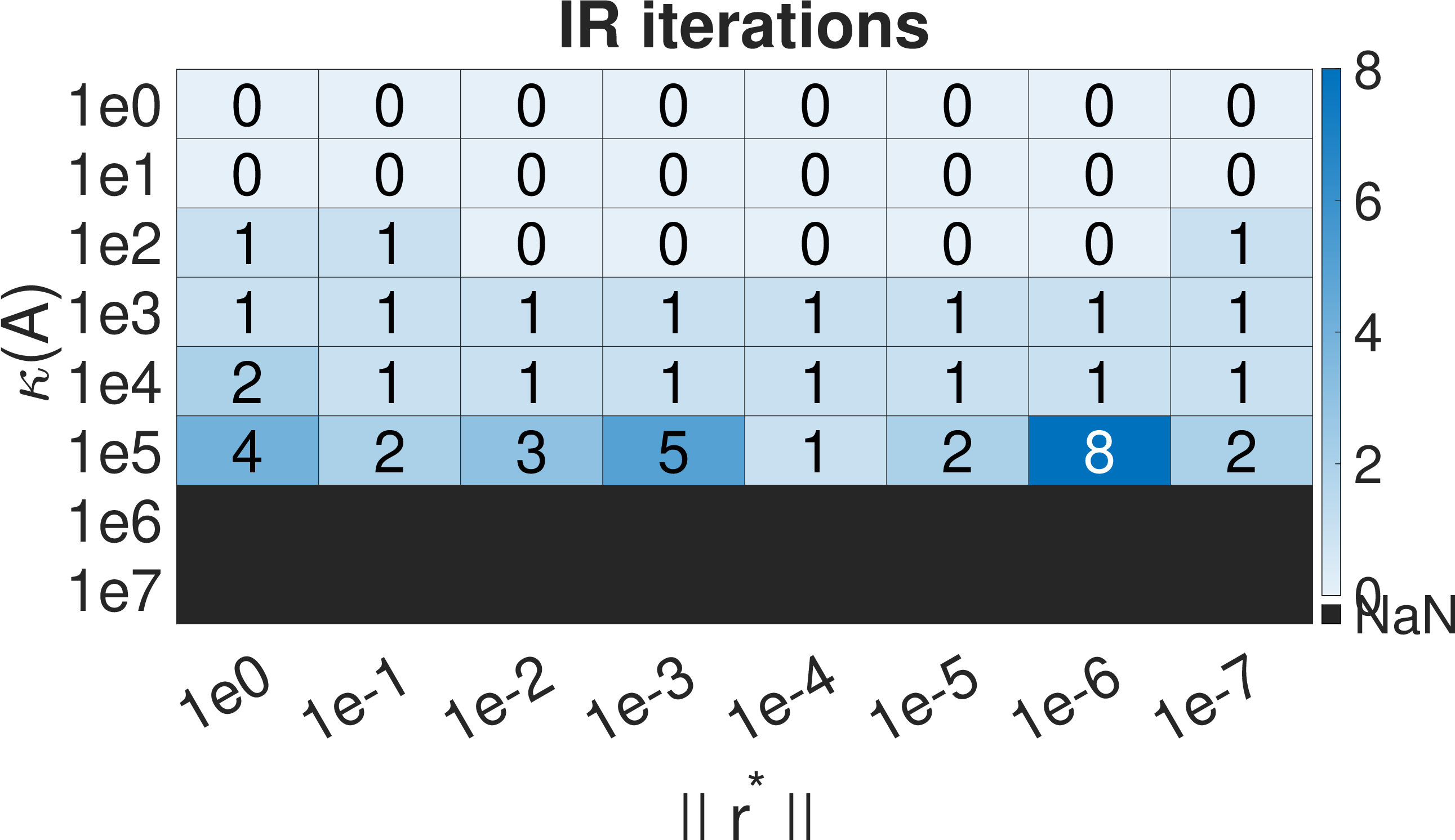}
  \caption{semi-normal, (double, single)}
\end{subfigure}
\begin{subfigure}[t]{0.45\linewidth}
  \centering
 \includegraphics[width=\linewidth]{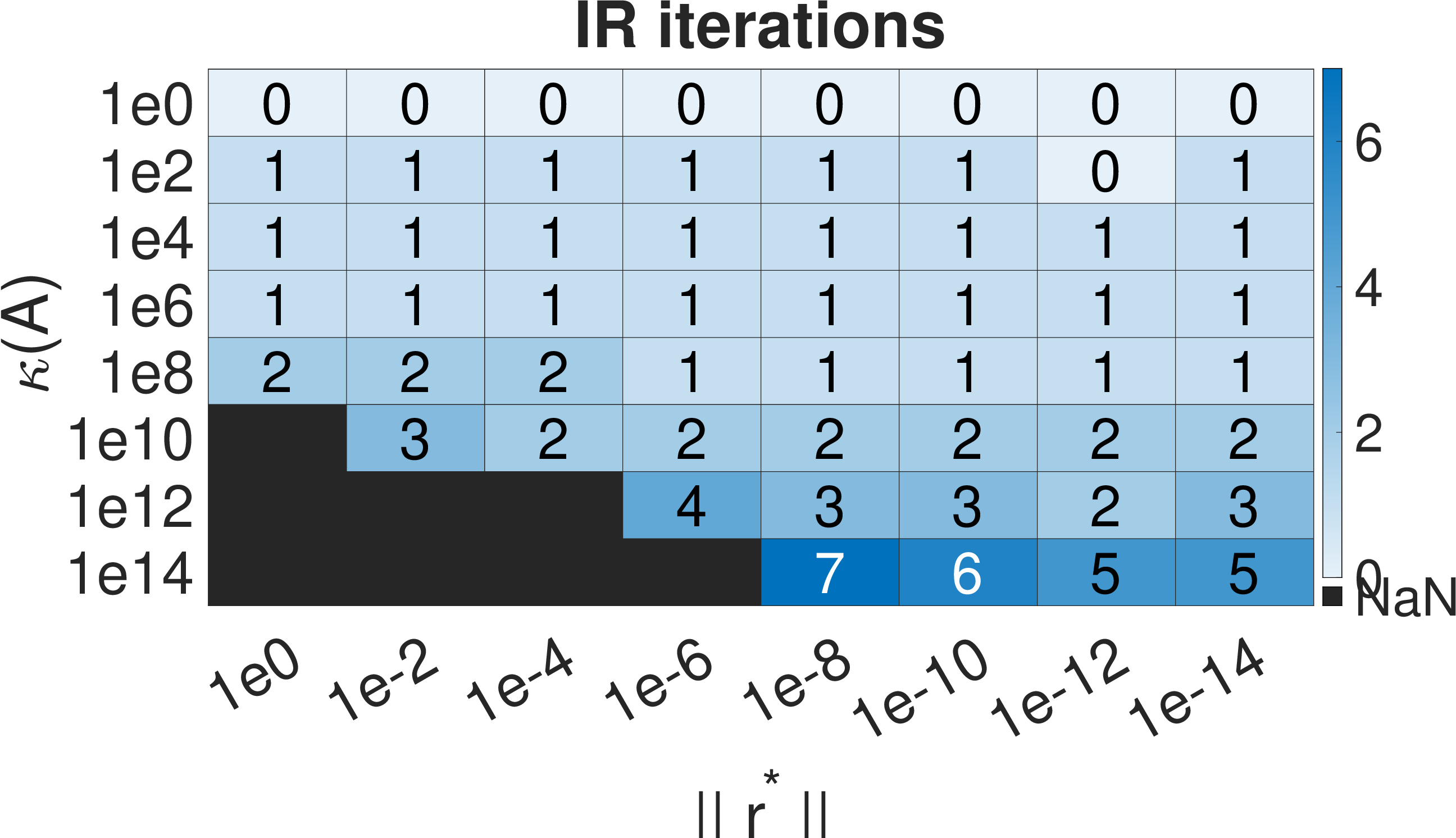}
  \caption{augmented approach, (quad, double)}
\end{subfigure}
\begin{subfigure}[t]{0.45\linewidth}
  \centering
 \includegraphics[width=\linewidth]{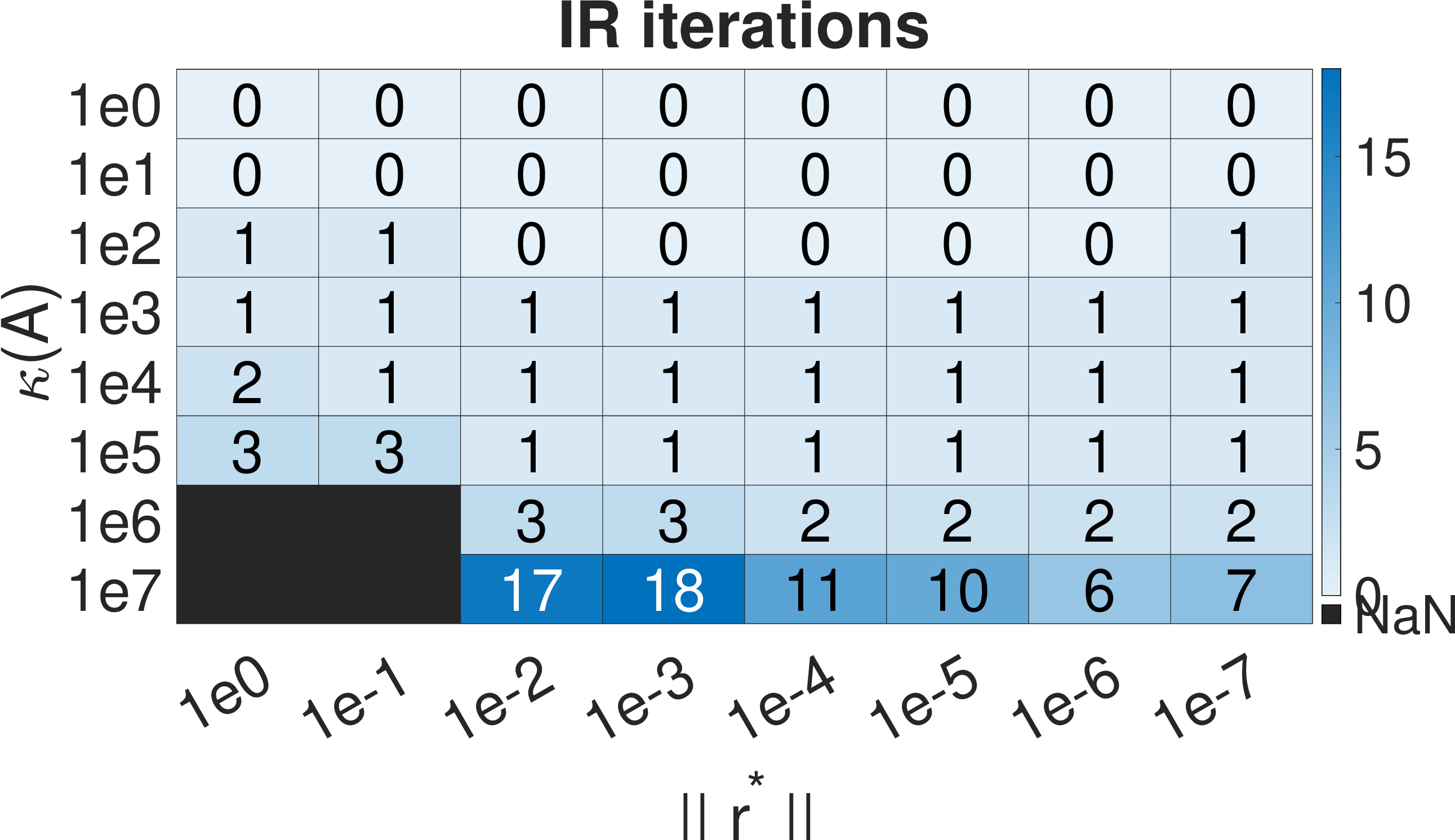}
  \caption{augmented approach, (double, single)}
\end{subfigure}
    \caption{Count of iterative refinement iterations for the LS, semi-normal equations, and augmented system approach when $u_r$ is set to quadruple and $u$ is set to double (left panels), and $u_r$ is set to double and $u$ is set to single (right). Black denotes when LSIR does not converge in 30 iterations.}
    \label{fig:ir_iterations}
\end{figure}

\begin{figure}
    \centering
\begin{subfigure}[t]{0.45\linewidth}
  \centering
 \includegraphics[width=\linewidth]{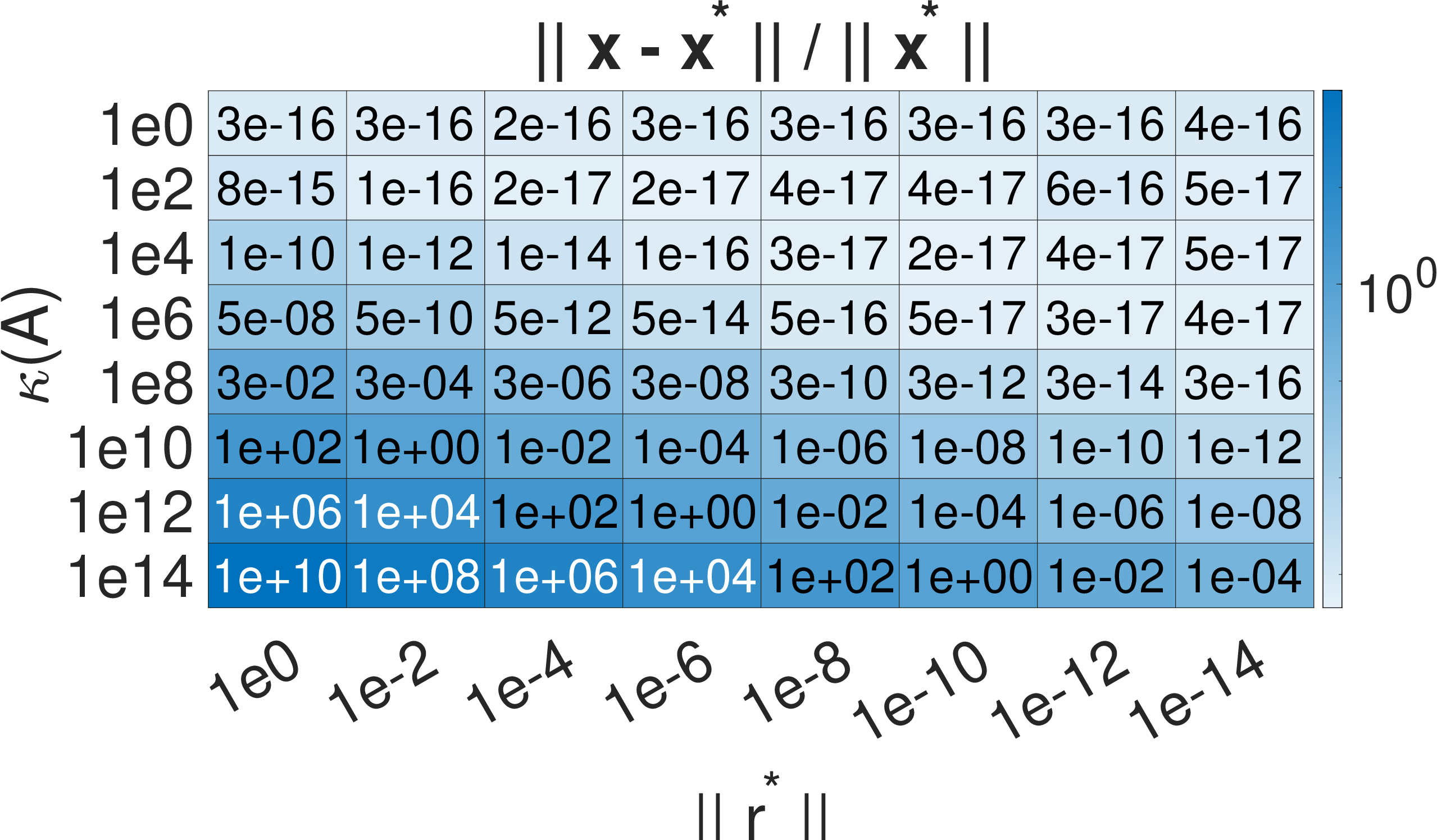}
  \caption{LS approach, (quad, double)}
\end{subfigure}
\begin{subfigure}[t]{0.45\linewidth}
  \centering
 \includegraphics[width=\linewidth]{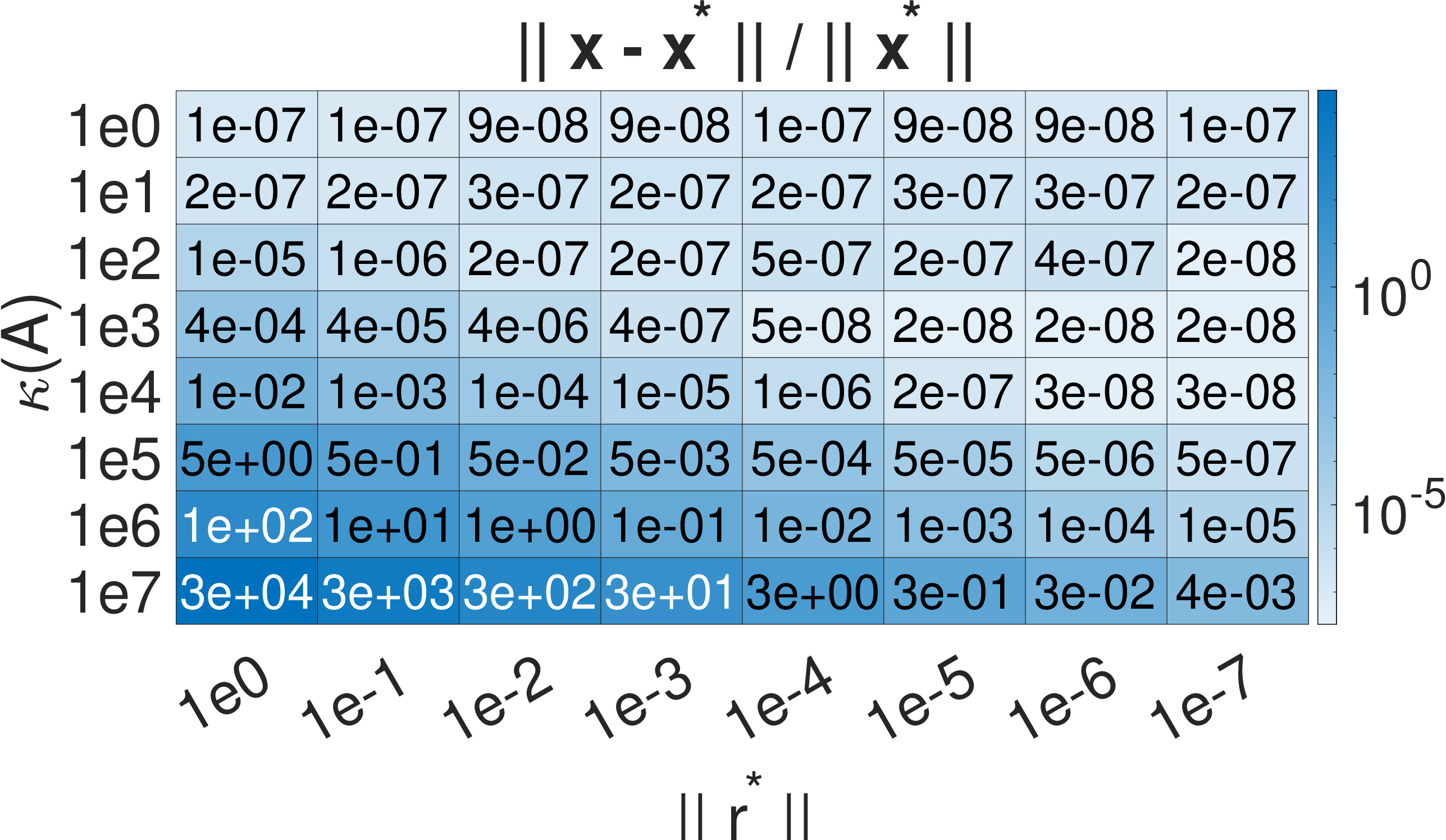}
  \caption{LS approach, (double, single)}
\end{subfigure}
\begin{subfigure}[t]{0.45\linewidth}
  \centering
 \includegraphics[width=\linewidth]{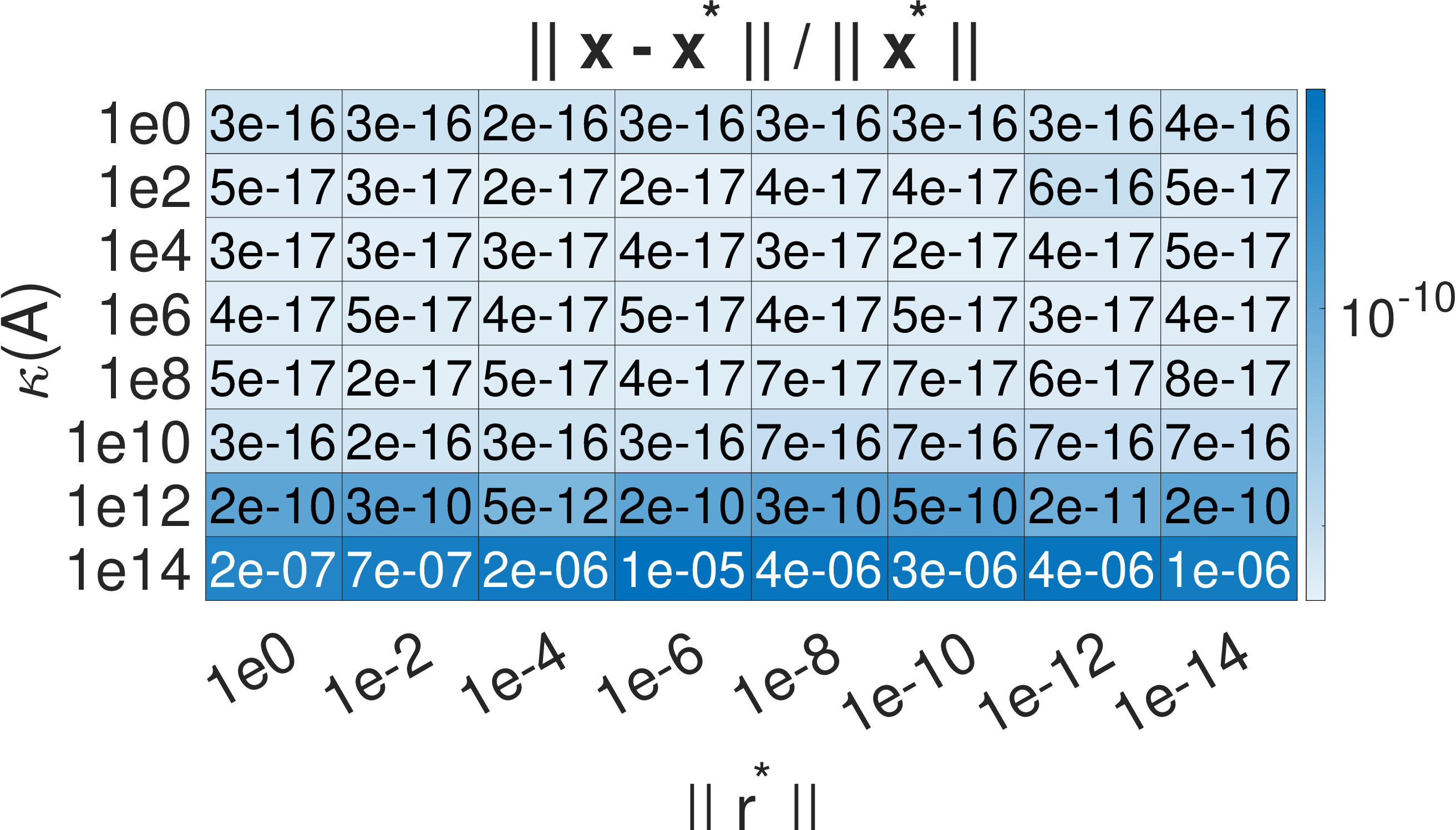}
  \caption{semi-normal, (quad, double)}
\end{subfigure}
\begin{subfigure}[t]{0.45\linewidth}
  \centering
 \includegraphics[width=\linewidth]{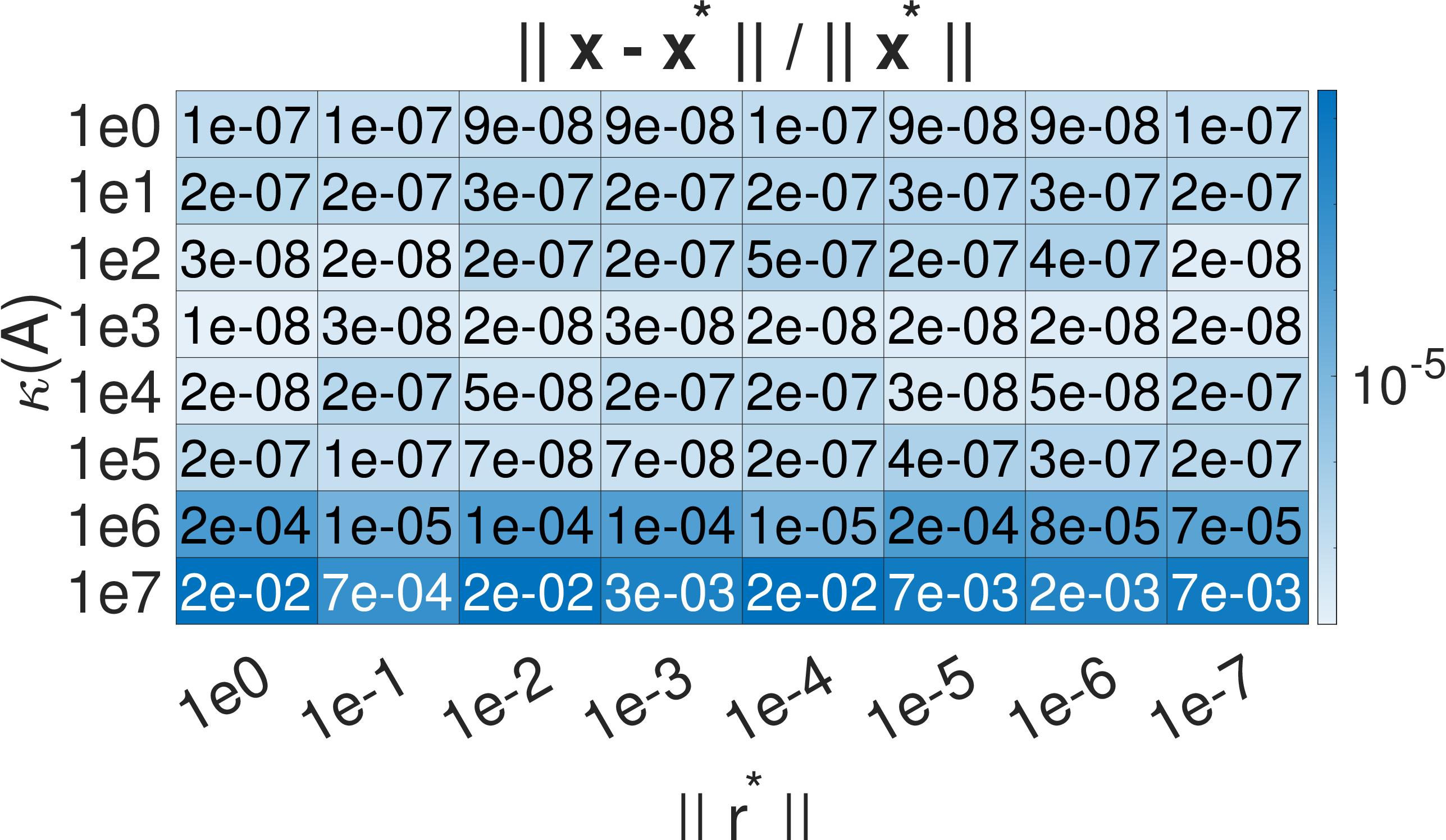}
  \caption{semi-normal, (double, single)}
\end{subfigure}
\begin{subfigure}[t]{0.45\linewidth}
  \centering
 \includegraphics[width=\linewidth]{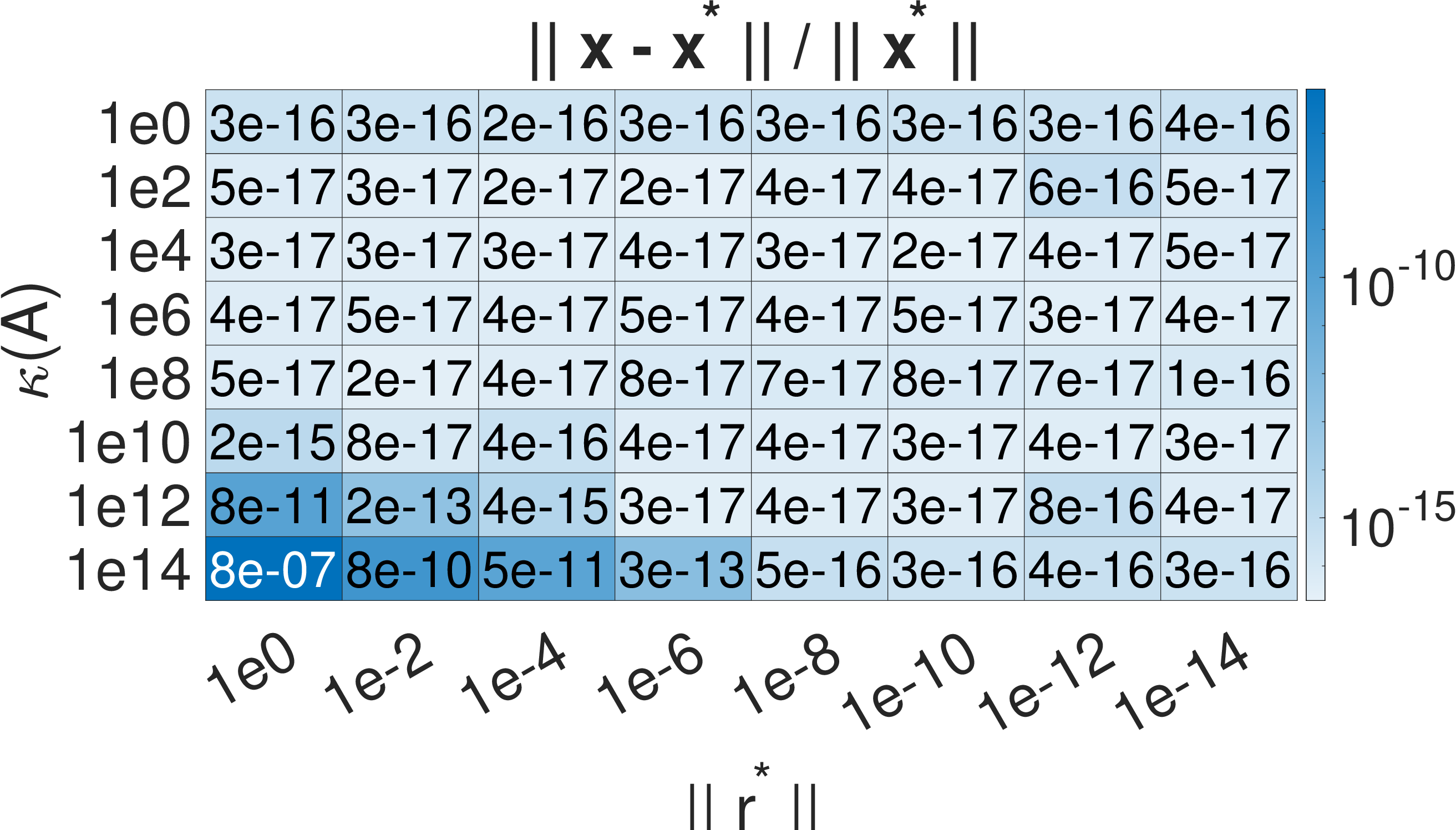}
  \caption{augmented approach, (quad, double)}
\end{subfigure}
\begin{subfigure}[t]{0.45\linewidth}
  \centering
 \includegraphics[width=\linewidth]{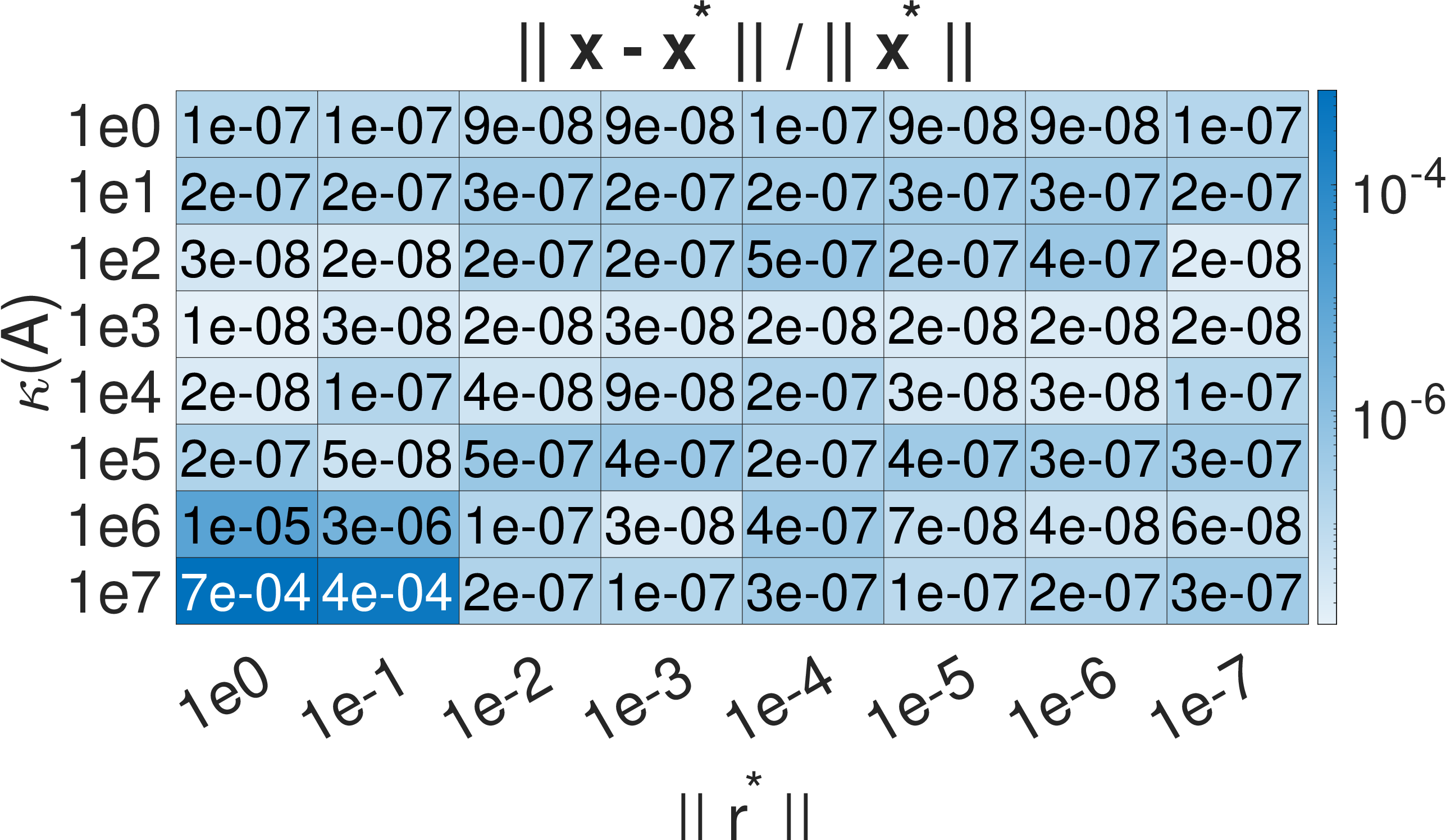}
  \caption{augmented approach, (double, single)}
\end{subfigure}
    \caption{As in Figure~\ref{fig:ir_iterations}, but for the relative error in $x$.}
    \label{fig:x_error}
\end{figure}

\begin{figure}
    \centering
\begin{subfigure}[t]{0.45\linewidth}
  \centering
 \includegraphics[width=\linewidth]{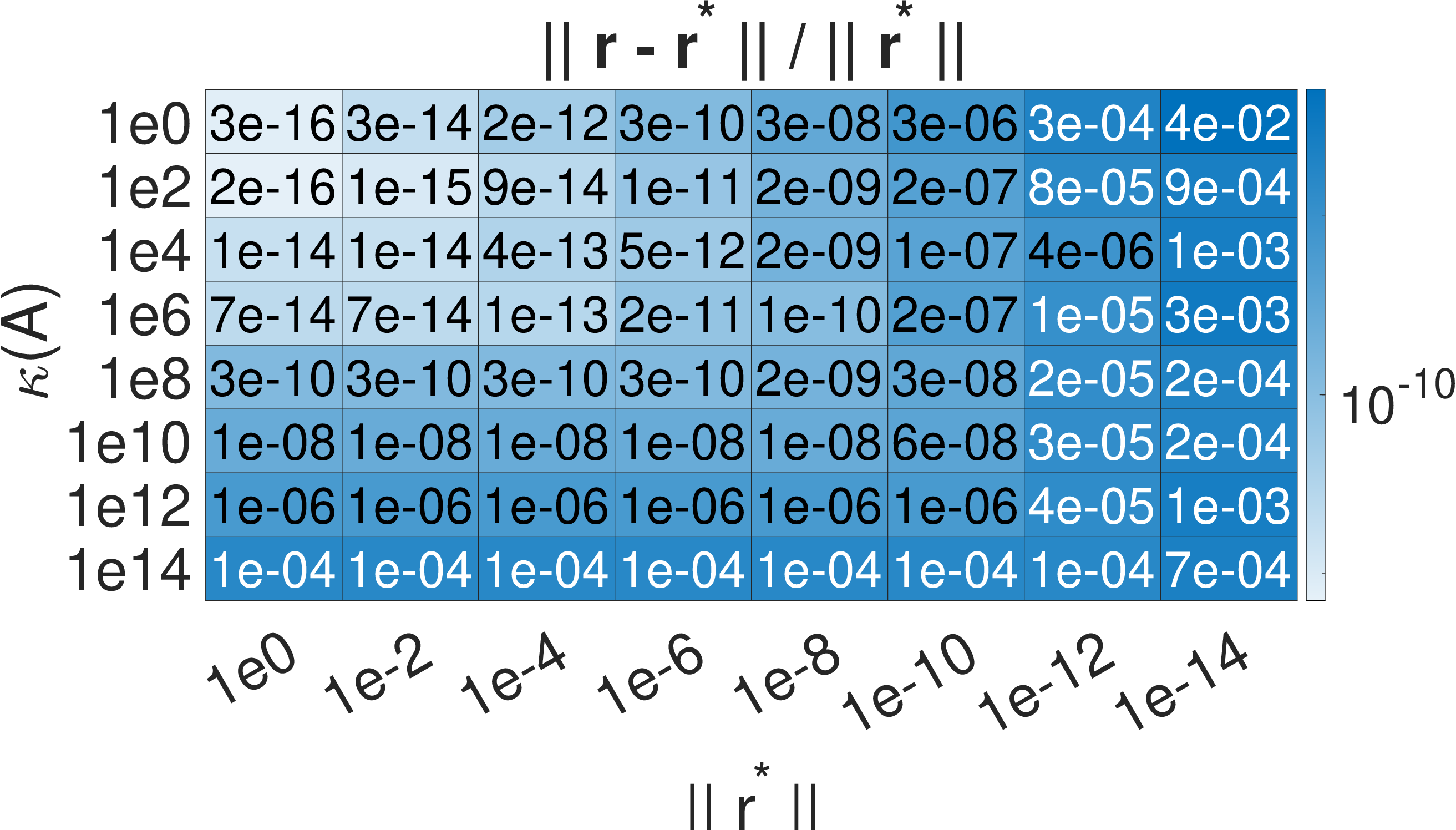}
  \caption{LS approach, (quad, double)}
\end{subfigure}
\begin{subfigure}[t]{0.45\linewidth}
  \centering
 \includegraphics[width=\linewidth]{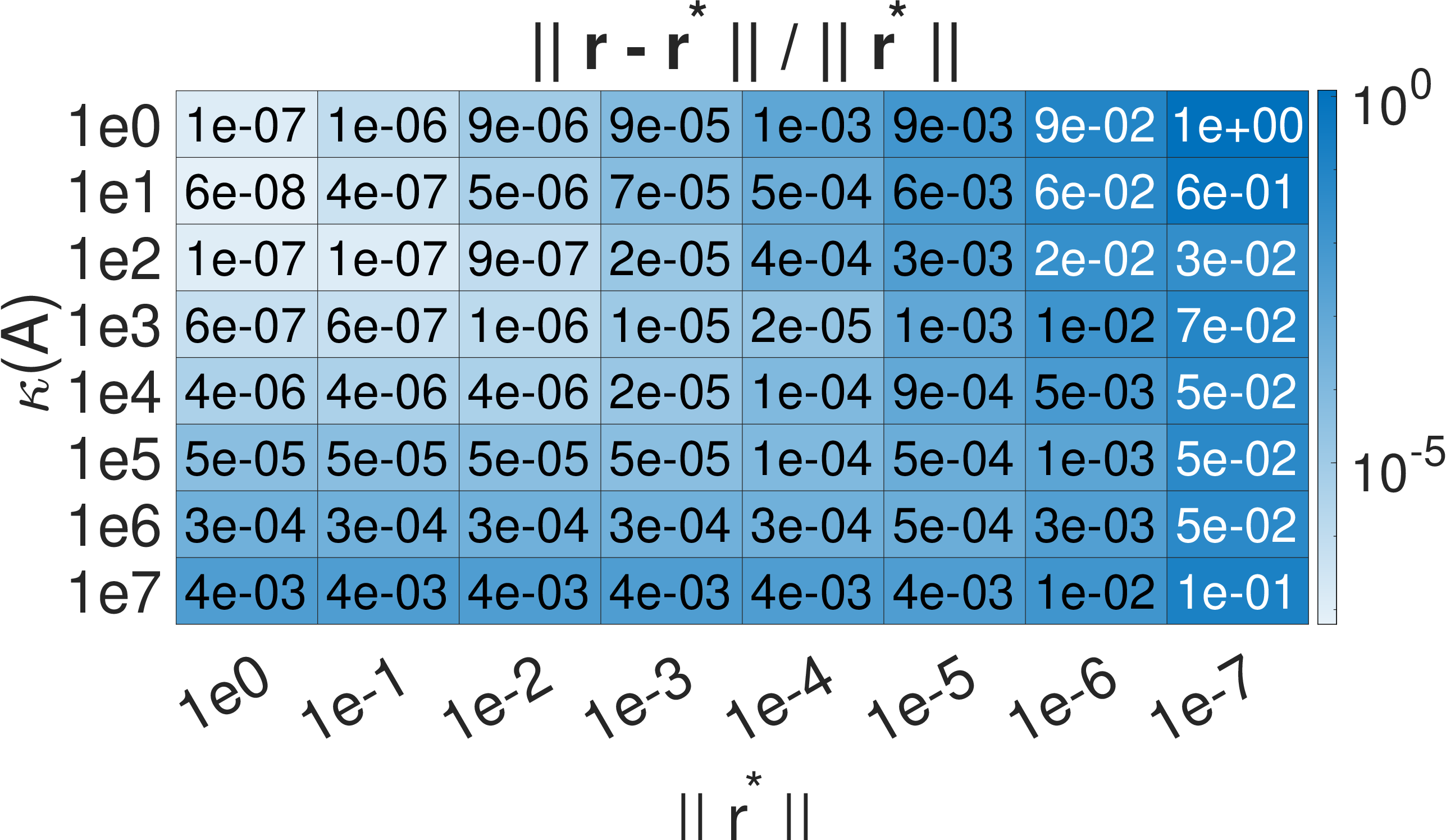}
  \caption{LS approach, (double, single)}
\end{subfigure}
\begin{subfigure}[t]{0.45\linewidth}
  \centering
 \includegraphics[width=\linewidth]{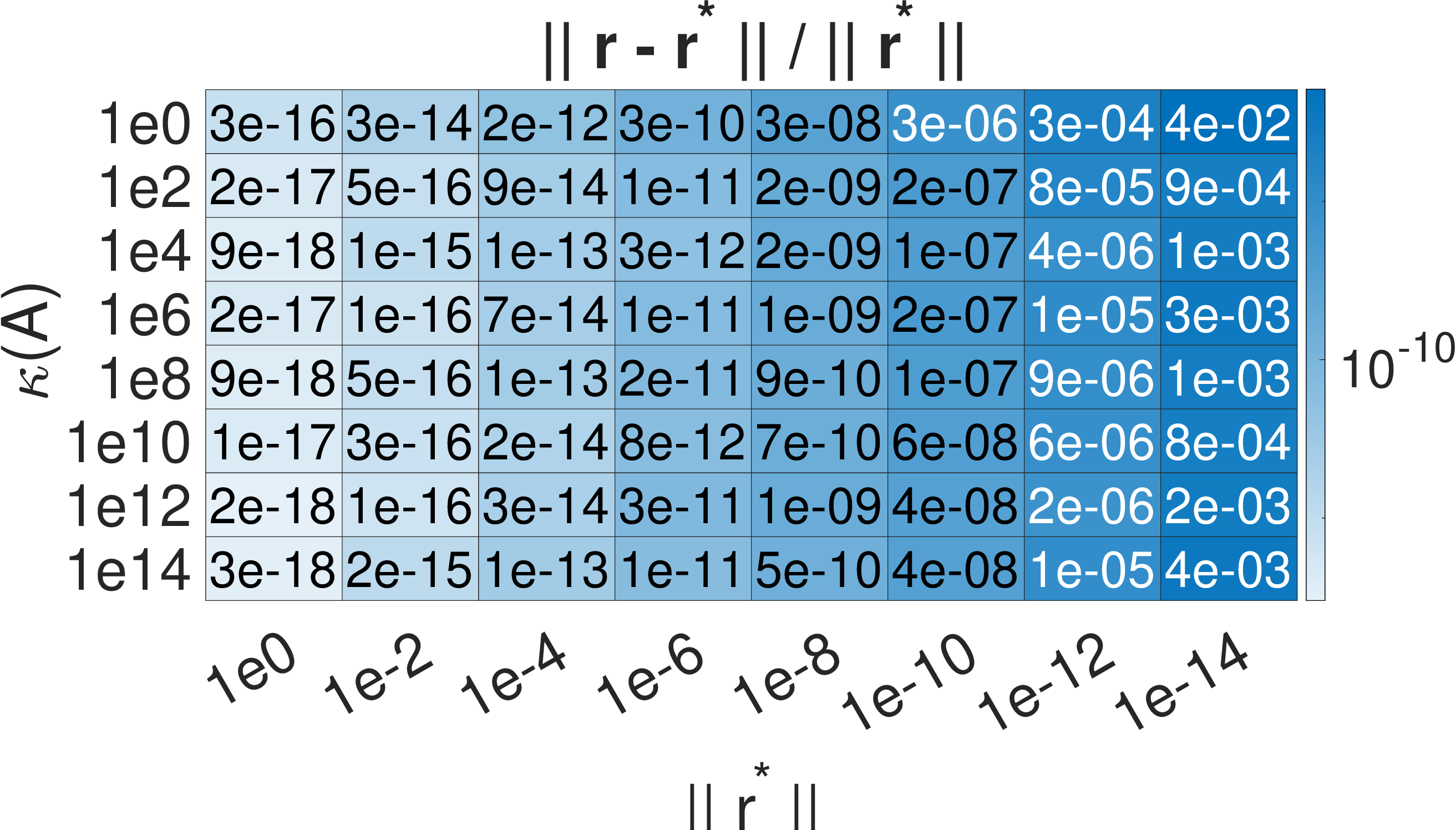}
  \caption{semi-normal, (quad, double)}
\end{subfigure}
\begin{subfigure}[t]{0.45\linewidth}
  \centering
 \includegraphics[width=\linewidth]{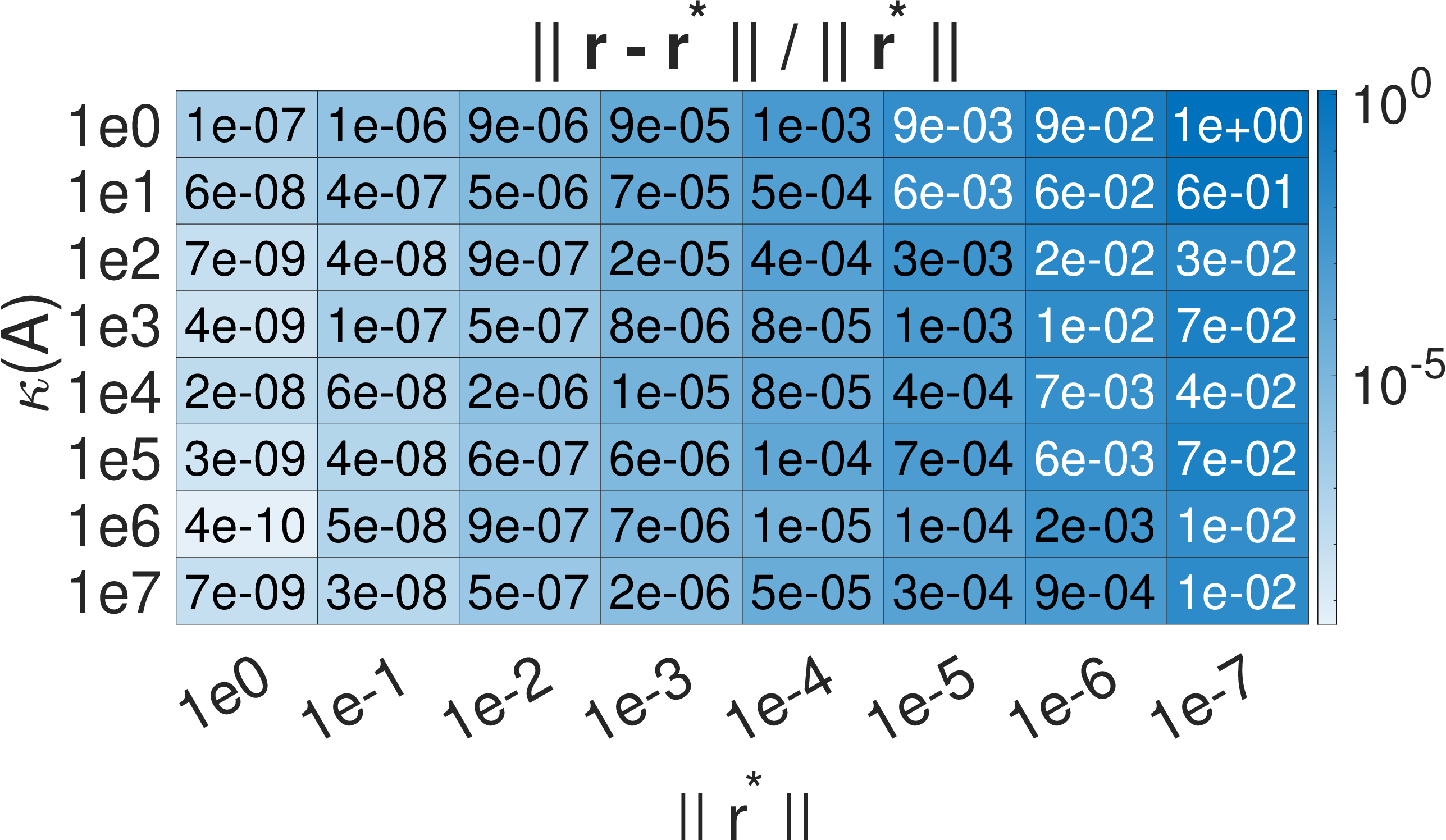}
  \caption{semi-normal, (double, single)}
\end{subfigure}
\begin{subfigure}[t]{0.45\linewidth}
  \centering
 \includegraphics[width=\linewidth]{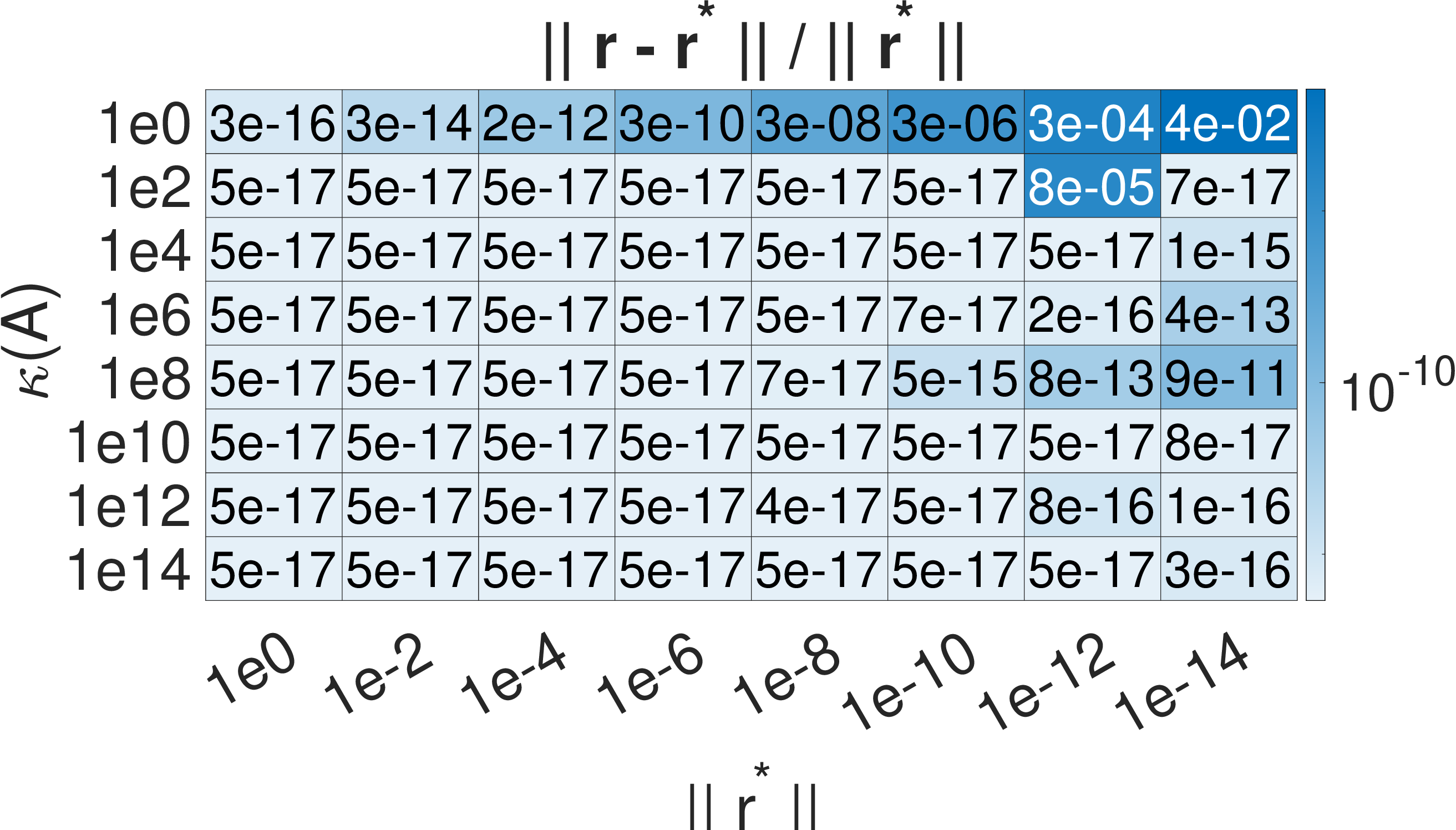}
  \caption{augmented approach, (quad, double)}
\end{subfigure}
\begin{subfigure}[t]{0.45\linewidth}
  \centering
 \includegraphics[width=\linewidth]{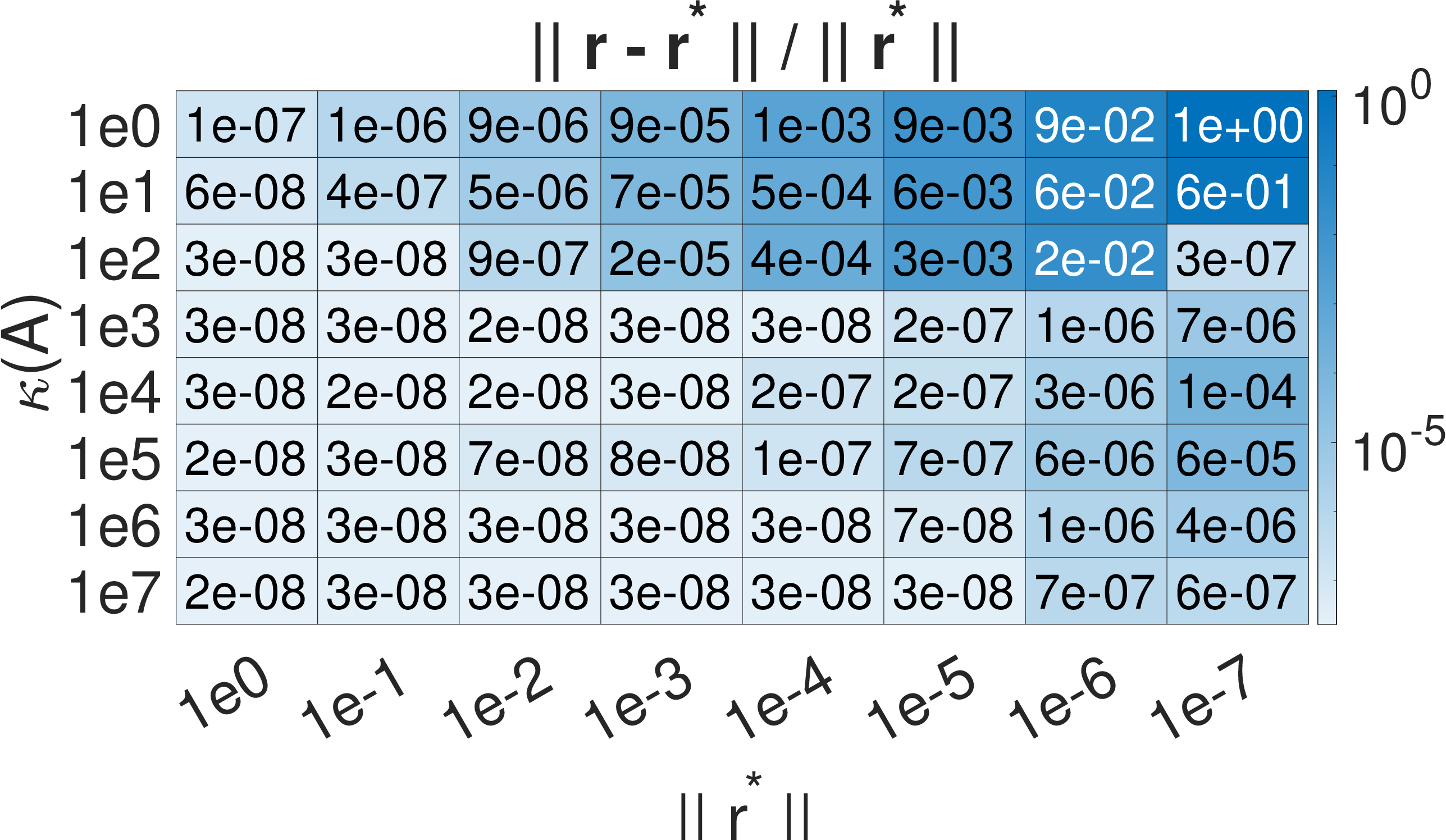}
  \caption{augmented approach, (double, single)}
\end{subfigure}
    \caption{As in Figure~\ref{fig:ir_iterations}, but for the relative error in $r$.}
    \label{fig:r_error}
\end{figure}

\begin{figure}
    \centering
\begin{subfigure}[t]{0.45\linewidth}
  \centering
 \includegraphics[width=\linewidth]{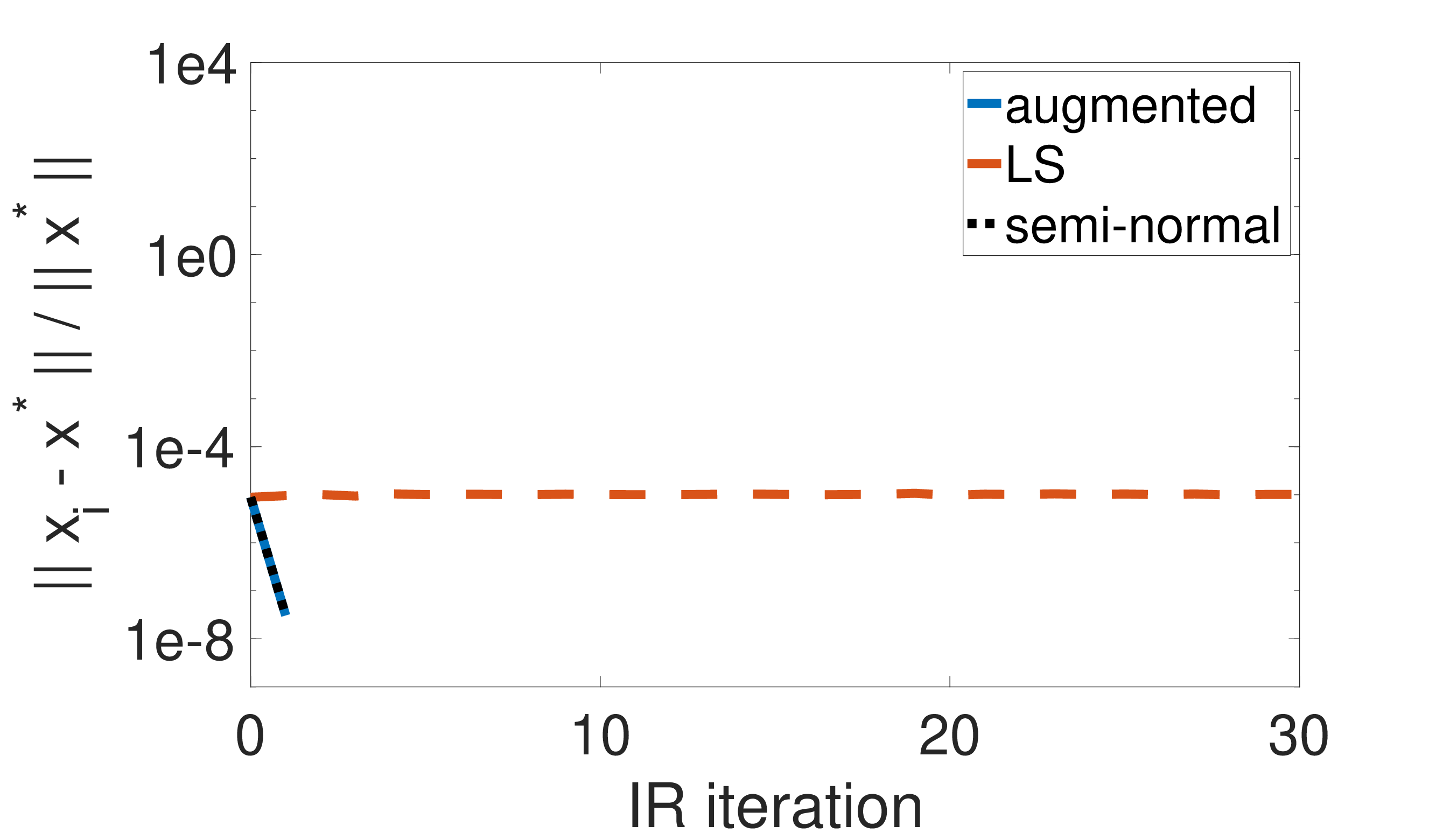}
  \caption{$\kappa_2(A)=$1e2, $\Vert r^* \Vert_2=$1e0 }
\end{subfigure}
\begin{subfigure}[t]{0.45\linewidth}
  \centering
 \includegraphics[width=\linewidth]{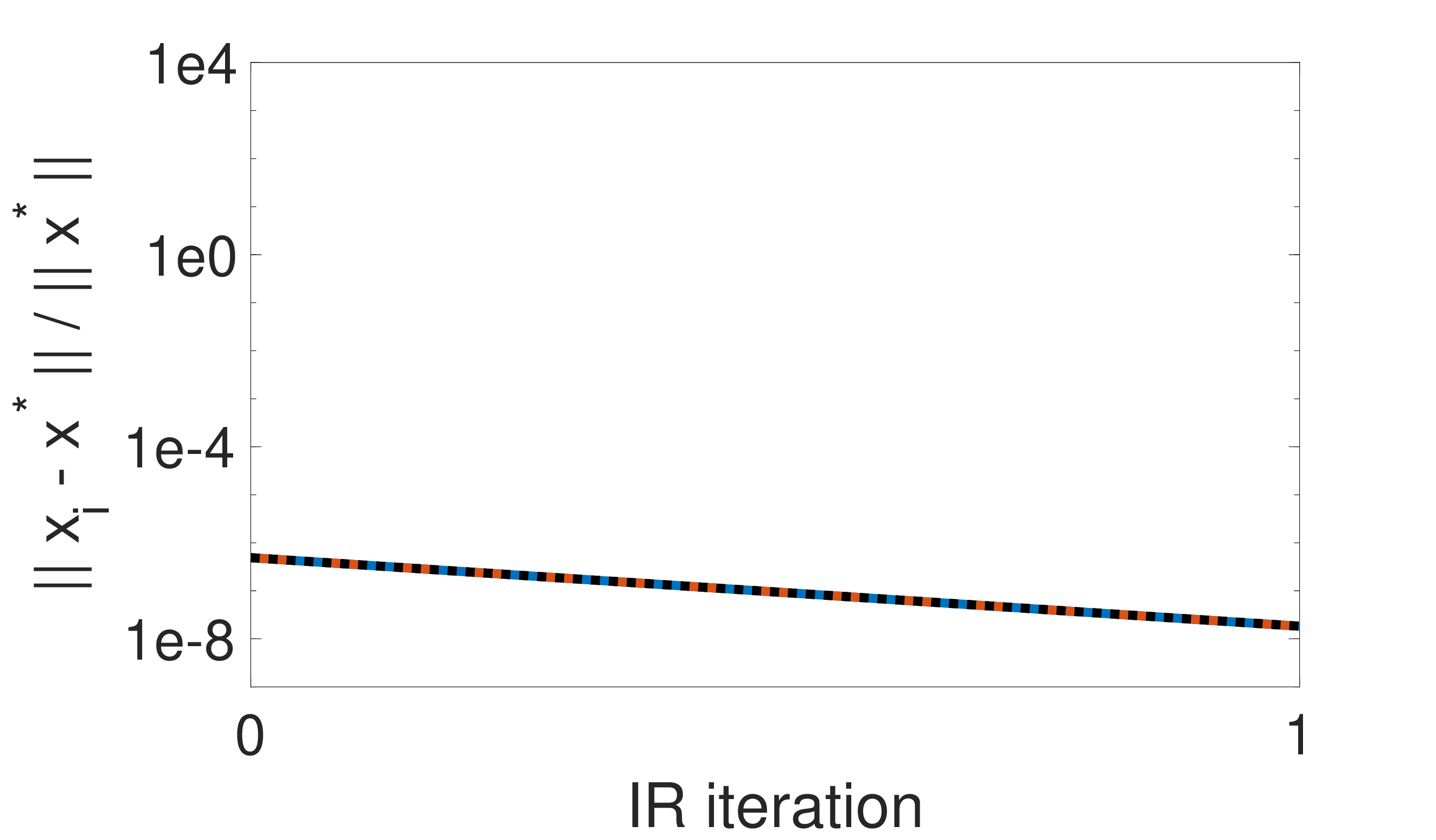}
  \caption{$\kappa_2(A)=$1e2, $\Vert r^* \Vert_2=$1e-7}
\end{subfigure}
\begin{subfigure}[t]{0.45\linewidth}
  \centering
 \includegraphics[width=\linewidth]{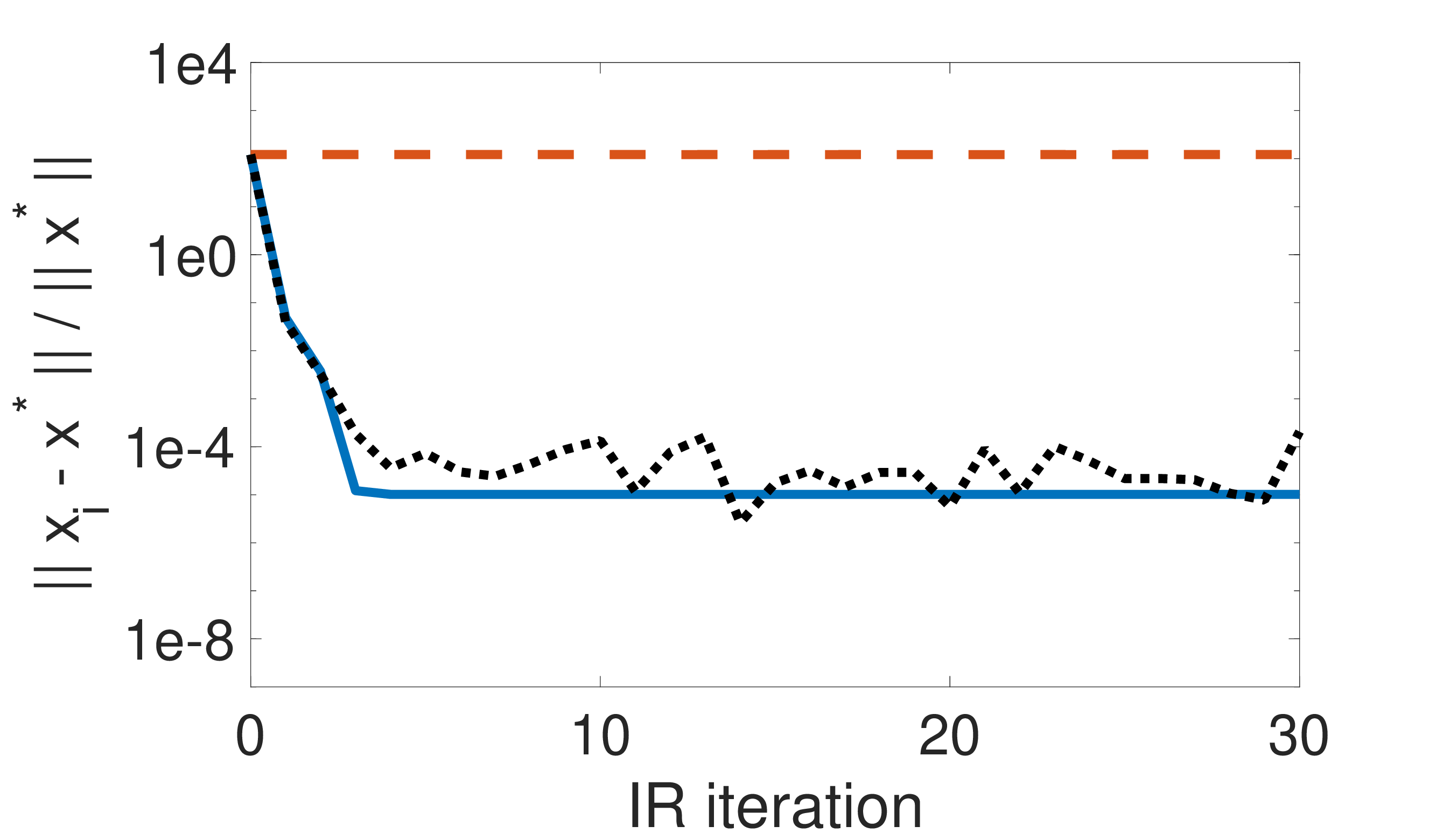}
  \caption{$\kappa_2(A)=$1e6, $\Vert r^* \Vert_2=$1e0 }
\end{subfigure}
\begin{subfigure}[t]{0.45\linewidth}
  \centering
 \includegraphics[width=\linewidth]{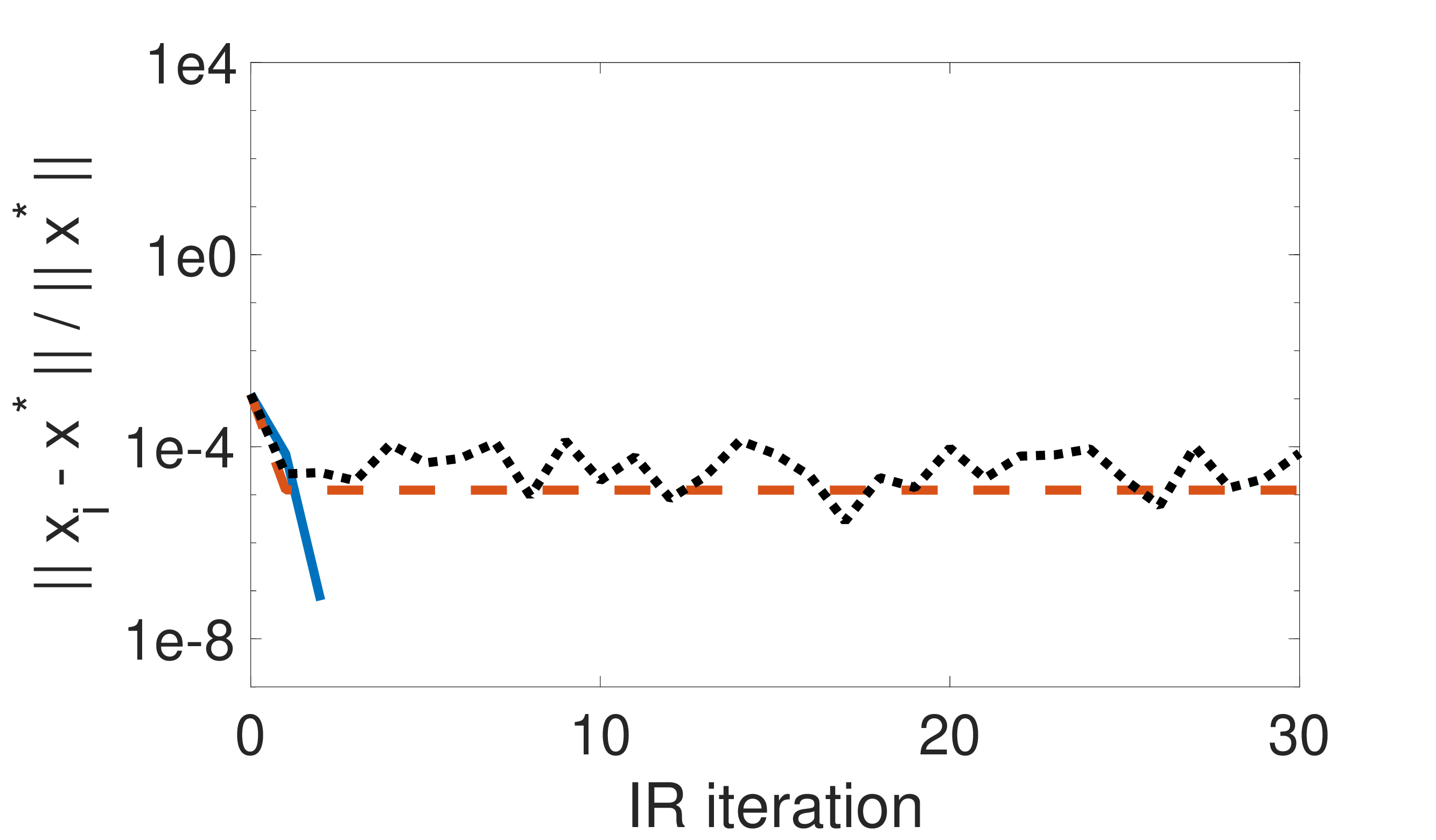}
  \caption{$\kappa_2(A)=$1e6, $\Vert r^* \Vert_2=$1e-7}
\end{subfigure}
    \caption{Relative error in $x$ at every IR iteration when solved via QR decomposition for combinations of large and small $\kappa_2(A)$ and $\Vert r^* \Vert_2$; $u_r$ is set to double and $u$ is set to single.}
    \label{fig:x_conv_qr}
\end{figure}

\begin{figure}
    \centering
\begin{subfigure}[t]{0.45\linewidth}
  \centering
 \includegraphics[width=\linewidth]{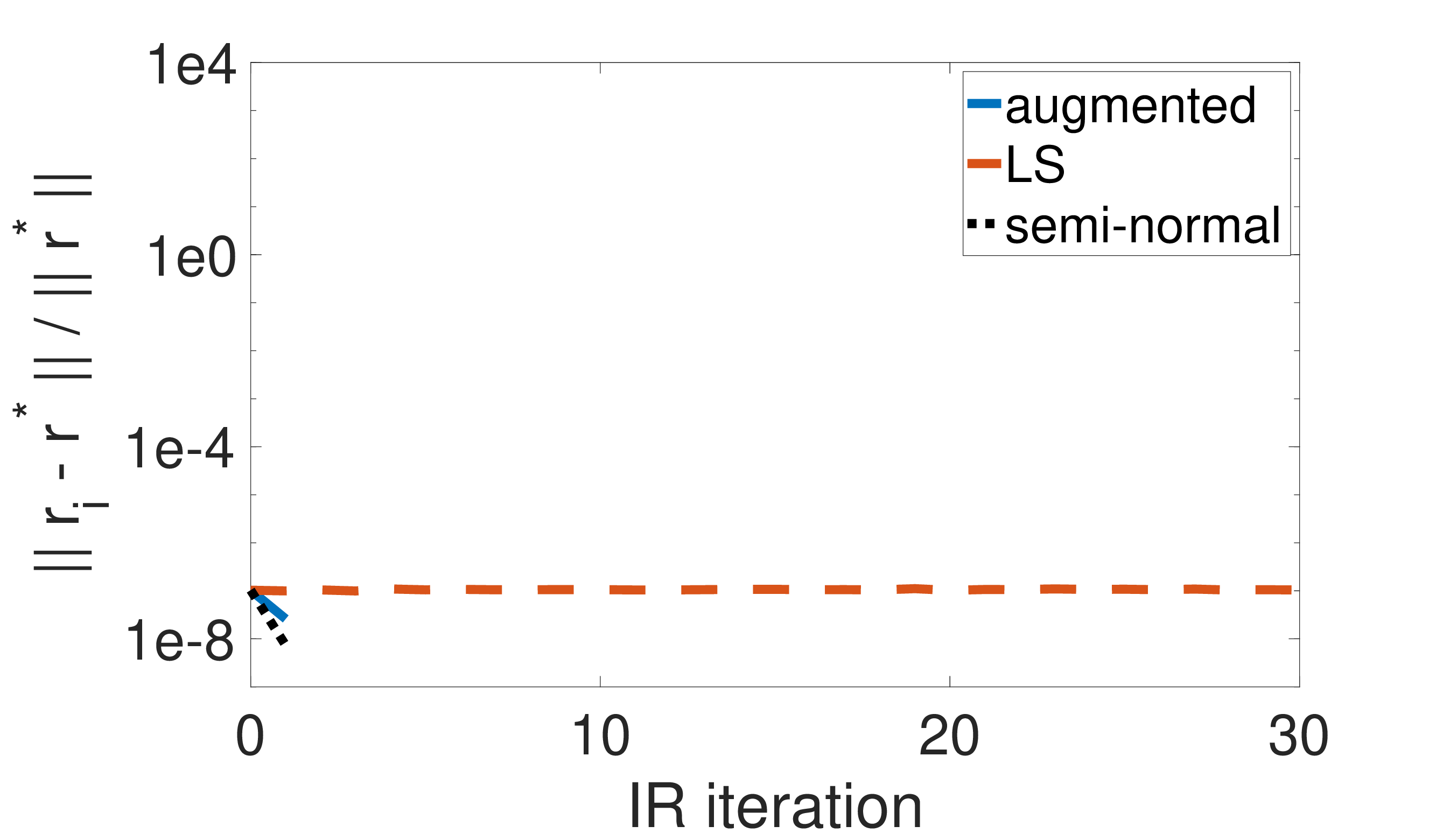}
  \caption{$\kappa_2(A)=$1e2, $\Vert r^* \Vert_2=$1e0 }
\end{subfigure}
\begin{subfigure}[t]{0.45\linewidth}
  \centering
 \includegraphics[width=\linewidth]{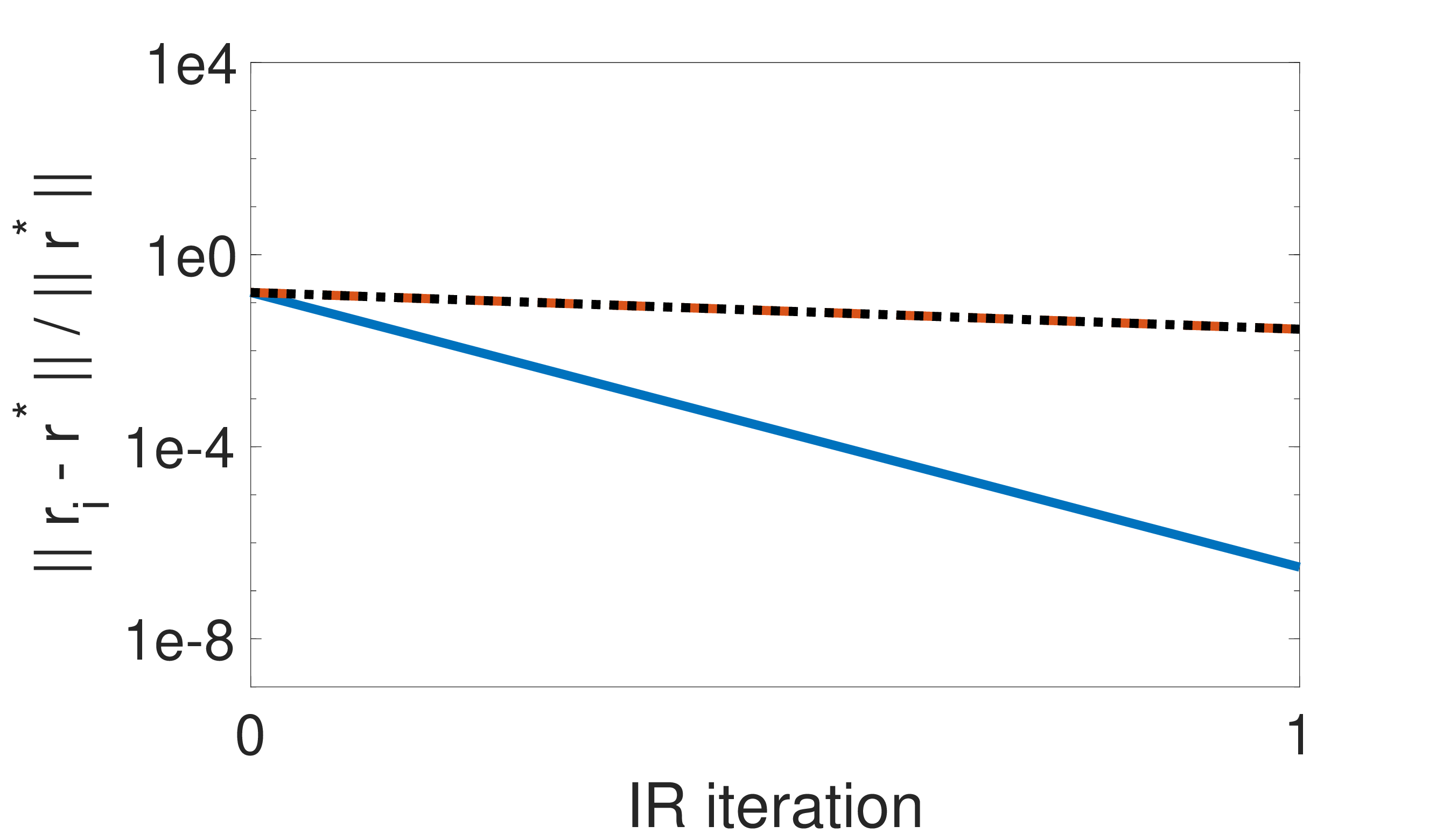}
  \caption{$\kappa_2(A)=$1e2, $\Vert r^* \Vert_2=$1e-7}
\end{subfigure}
\begin{subfigure}[t]{0.45\linewidth}
  \centering
 \includegraphics[width=\linewidth]{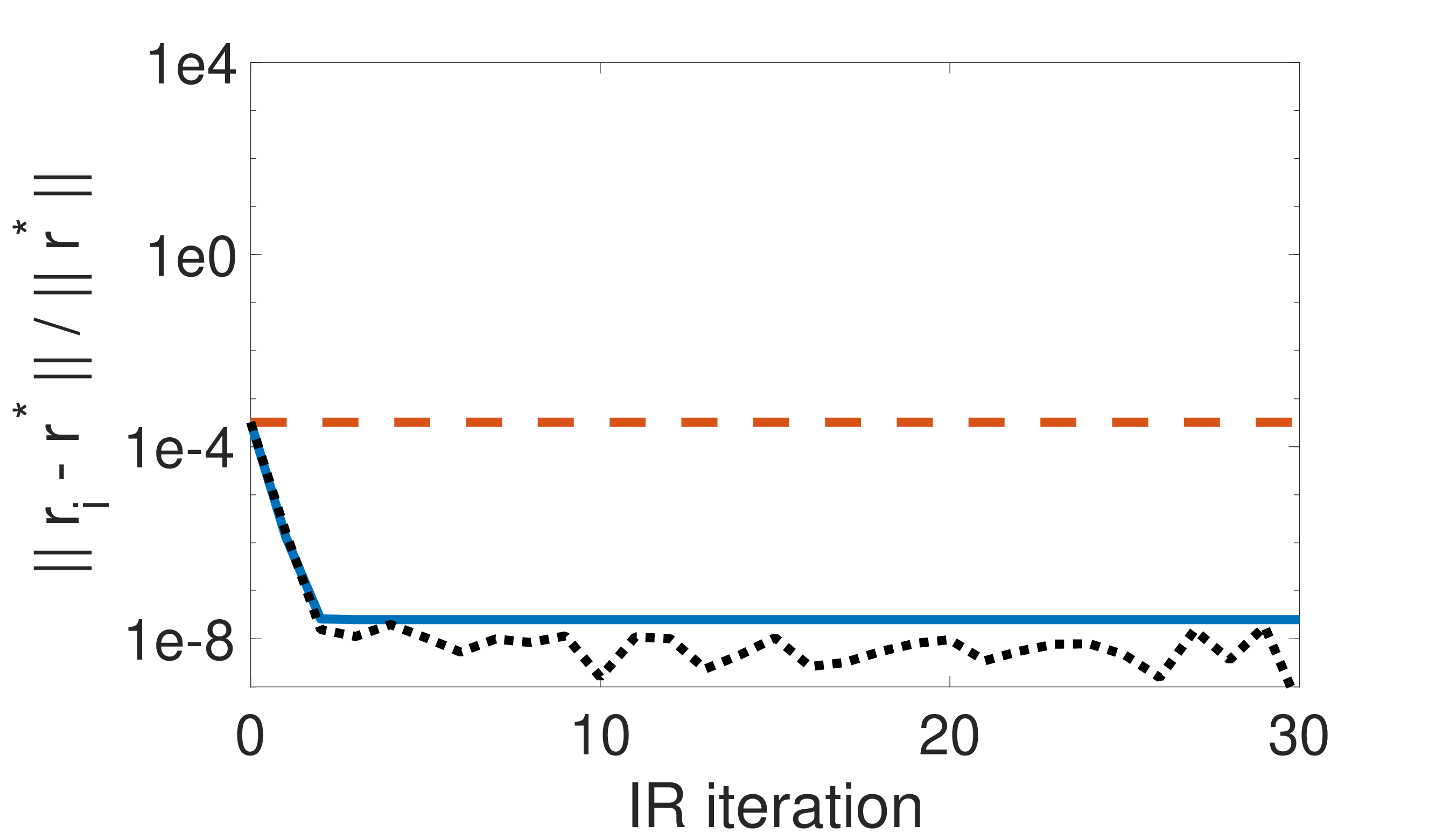}
  \caption{$\kappa_2(A)=$1e6, $\Vert r^* \Vert_2=$1e0 }
\end{subfigure}
\begin{subfigure}[t]{0.45\linewidth}
  \centering
 \includegraphics[width=\linewidth]{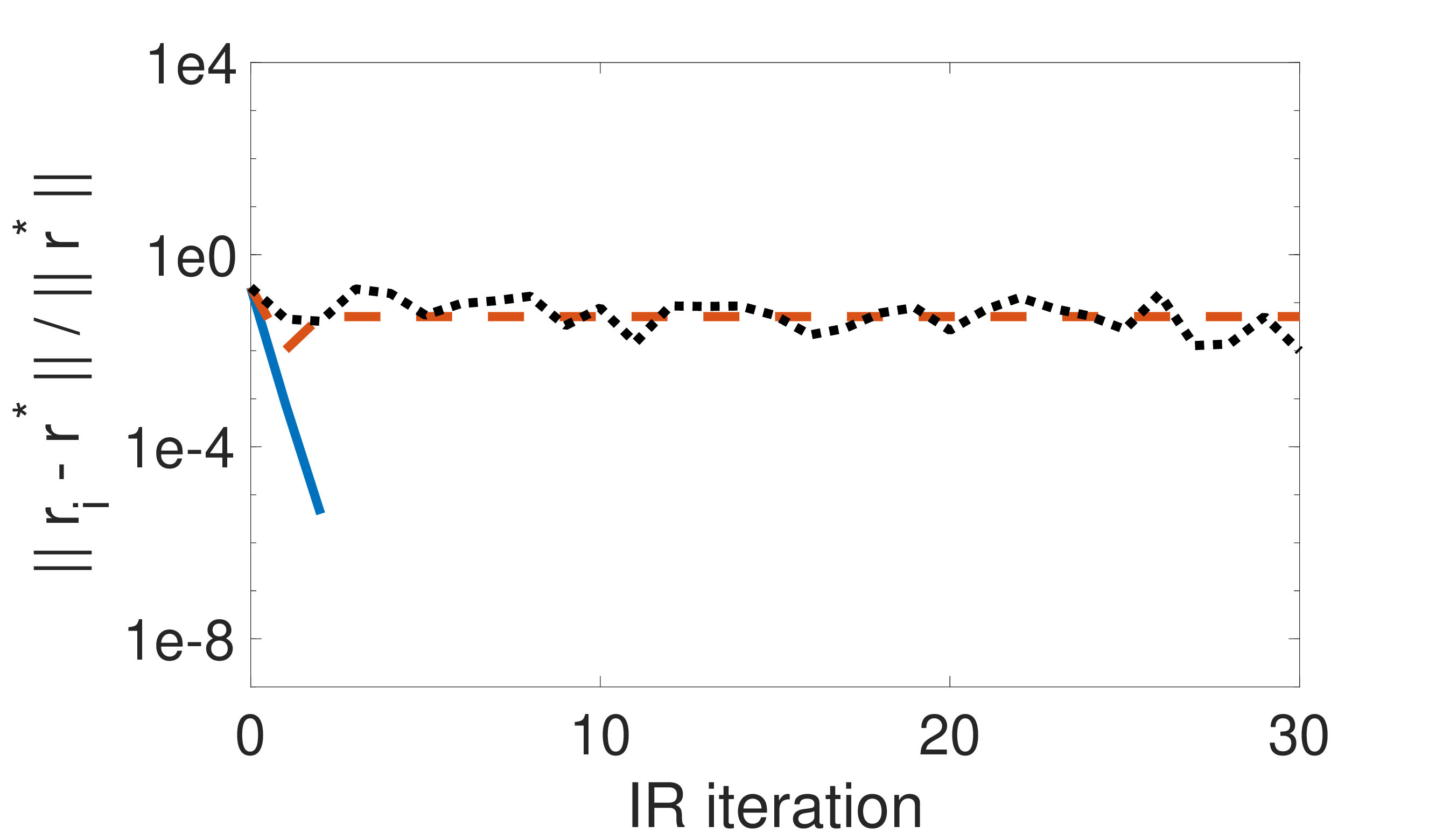}
  \caption{$\kappa_2(A)=$1e6, $\Vert r^* \Vert_2=$1e-7}
\end{subfigure}
    \caption{As in Figure~\ref{fig:x_conv_qr}, but for the relative error in $r$.}
    \label{fig:r_conv_qr}
\end{figure}

\subsection{Solving via iterative methods}
We perform some numerical experiments solving the correction equations via iterative methods, namely, LSQR for the least-squares system approach and GMRES for the augmented system approach. GMRES is used here because of the available backward error guarantees even for highly ill-conditioned problems. In practice, one may choose more appropriate solvers like MINRES; see, e.g., \cite{carson2020three} for results with using both GMRES and MINRES in LSIR. We also test the combined approach in Algorithm~\ref{alg:combined_approach} using LSQR for the least-squares problems and computing $P_N \fhat_i$ as $\fhat_i - A \delta \xhat_1^{(1)}$ as discussed in Section~\ref{sec:combining_ls_augmented}. 
Note that the success of these iterative methods depends on factors such as preconditioning, convergence criteria, the chosen scaling factor for the augmented system, precisions for computing matrix-vector products and applying the preconditioner, etc. The choice of these parameters is heavily problem dependent. We do not aim to provide an extensive study of these effects but rather to illustrate some of the behavior. We thus perform experiments using a randomized preconditioner, which has been shown to be effective for overdetermined LS problems with $n \ll m$ \cite{rokhlin2008fast,avron2010blendenpik}. 

We use the built-in MATLAB implementation of LSQR and our own implementation of GMRES. The convergence tolerances are set to $10^{-14}$ and $10^{-7}$ for LSQR and $10^{-12}$ and $10^{-6}$ for GMRES when $u$ is set to double and single, respectively. Additional experiments show that larger tolerances lead to worse performance (results not shown). The maximum number of LSQR iterations is set to $n=10$, and $50$ for GMRES. The initial approximate solution $x_0$ is obtained via preconditioned LSQR solve in all cases and is then used to compute $r_0 = b - A x_0$. As was observed by Golub when using a direct solver \cite{golub1965numerical}, additional numerical experiments show that the LS system approach shows large sensitivity to the quality of $x_0$ (results not shown here).

To generate the preconditioner, we compute the economic QR decomposition of $\Omega A$, where $\Omega=(4n)^{-1/2}G$ and $G$ is a $4n \times m$ random matrix with entries drawn from a standard normal distribution. The $R$-factor $R$ is used as a right preconditioner in LSQR for all the least-squares problems, except step~\ref{step:combined_underdetermined_LS} in Algorithm~\ref{alg:combined_approach} where it is used on the left. The augmented system is preconditioned with a split-preconditioner as proposed in \cite{carson2024mixed}, that is, we solve 
\begin{equation*}
   \begin{bmatrix}
I & 0 \\ 0 & R^{-T} 
\end{bmatrix} 
\begin{bmatrix}
 I & A\\ A^T & 0
\end{bmatrix}
\begin{bmatrix}
I & 0 \\ 0 & R^{-1} 
\end{bmatrix}
\begin{bmatrix}
\delta r_i \\  \delta \widetilde{x}_i
\end{bmatrix}
=
 \begin{bmatrix}
I & 0 \\ 0 & R^{-T} 
\end{bmatrix} 
\begin{bmatrix}
 f_i \\ g_i
\end{bmatrix},
\end{equation*}
where $\delta \widetilde{x}_i = R \delta x_i$. See \cite{carson2024mixed} for a theoretical analysis of FGMRES-based IR with such randomized preconditioning.

Note that $x_0$ is different than when solving via QR decomposition and we thus need to run IR for some problems that did not require any refinement in the previous section; see Figures~\ref{fig:ir_iterations_iterative}, \ref{fig:ir_inner_iterations_iterative}, \ref{fig:x_error_iterative}, and \ref{fig:r_error_iterative}. The augmented approach converges for the widest range of problems, although for highly ill-conditioned problems it performs worse than the QR approach. This may be tackled by tuning the parameters and precisions used inside the iterative solver. The LS approach converges for similar sets of problems as when solving via QR decomposition. The combined approach performs similar to the LS approach, except the combined approach converges for problems with larger $\Vert r^* \Vert_2$ and gives a substantially better improvement for the residual at the cost of more LSQR iterations.  
Convergence curves in Figures~\ref{fig:x_conv_iterative} and \ref{fig:r_conv_iterative} confirm these observations, and show that the combined approach follows the the behaviour of the augmented approach for well-conditioned problems and is similar to the LS approach for ill-conditioned problems.

\begin{figure}
    \centering
\begin{subfigure}[t]{0.45\linewidth}
  \centering
 \includegraphics[width=\linewidth]{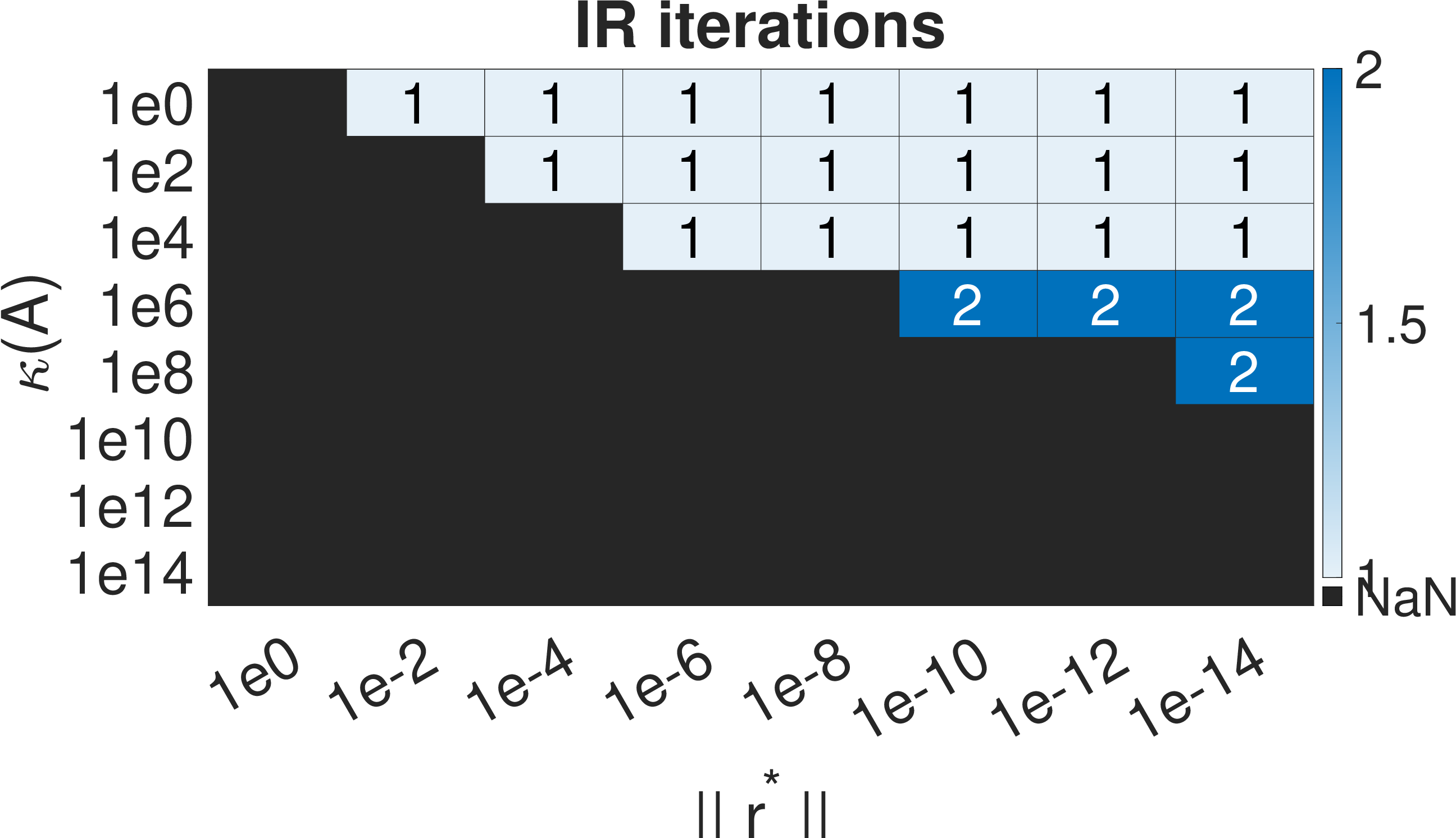}
  \caption{LS approach, (quad, double)}
\end{subfigure}
\begin{subfigure}[t]{0.45\linewidth}
  \centering
 \includegraphics[width=\linewidth]{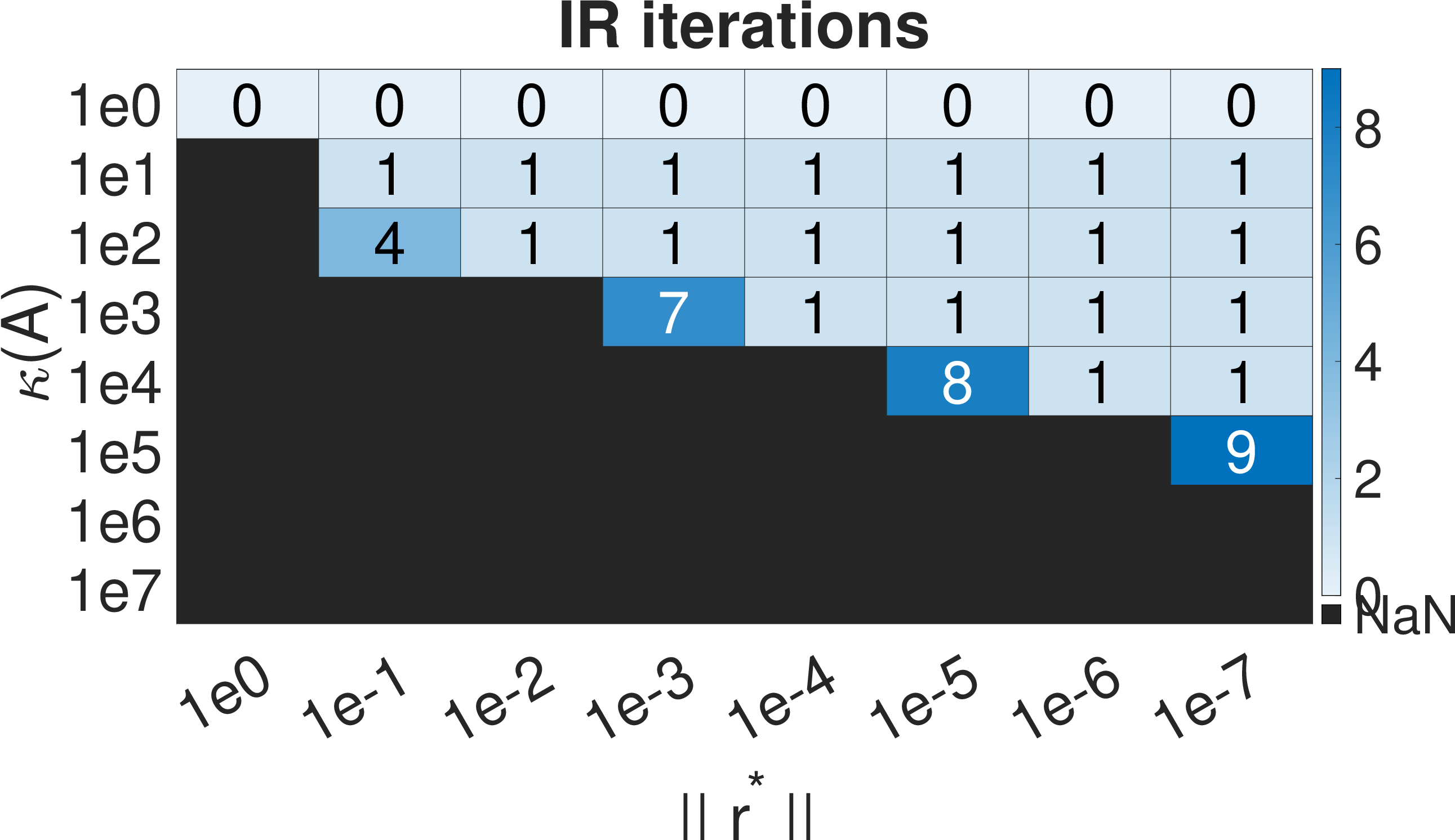}
  \caption{LS approach, (double, single)}
\end{subfigure}
\begin{subfigure}[t]{0.45\linewidth}
  \centering
 \includegraphics[width=\linewidth]{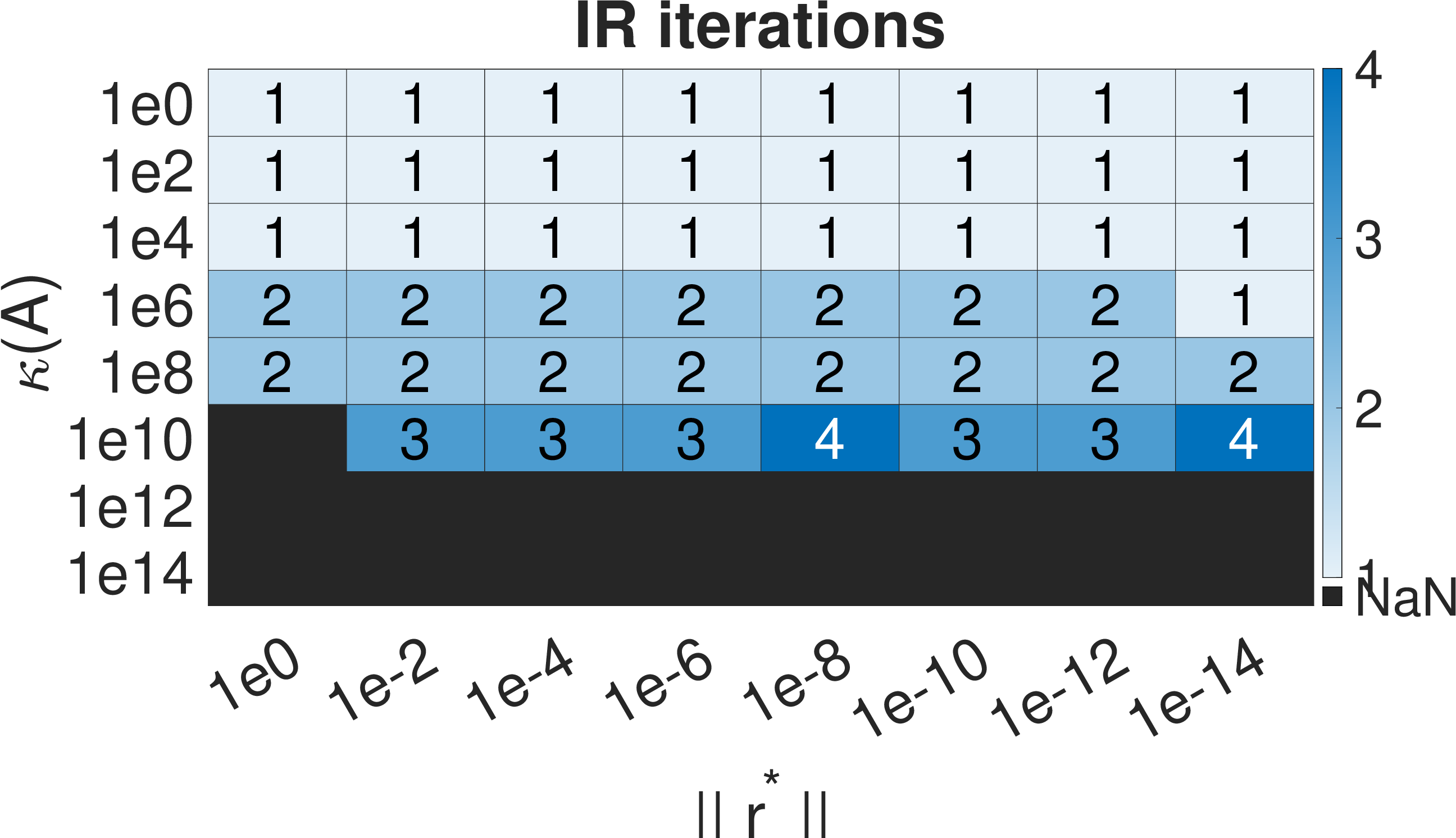}
  \caption{augmented approach, (quad, double)}
\end{subfigure}
\begin{subfigure}[t]{0.45\linewidth}
  \centering
 \includegraphics[width=\linewidth]{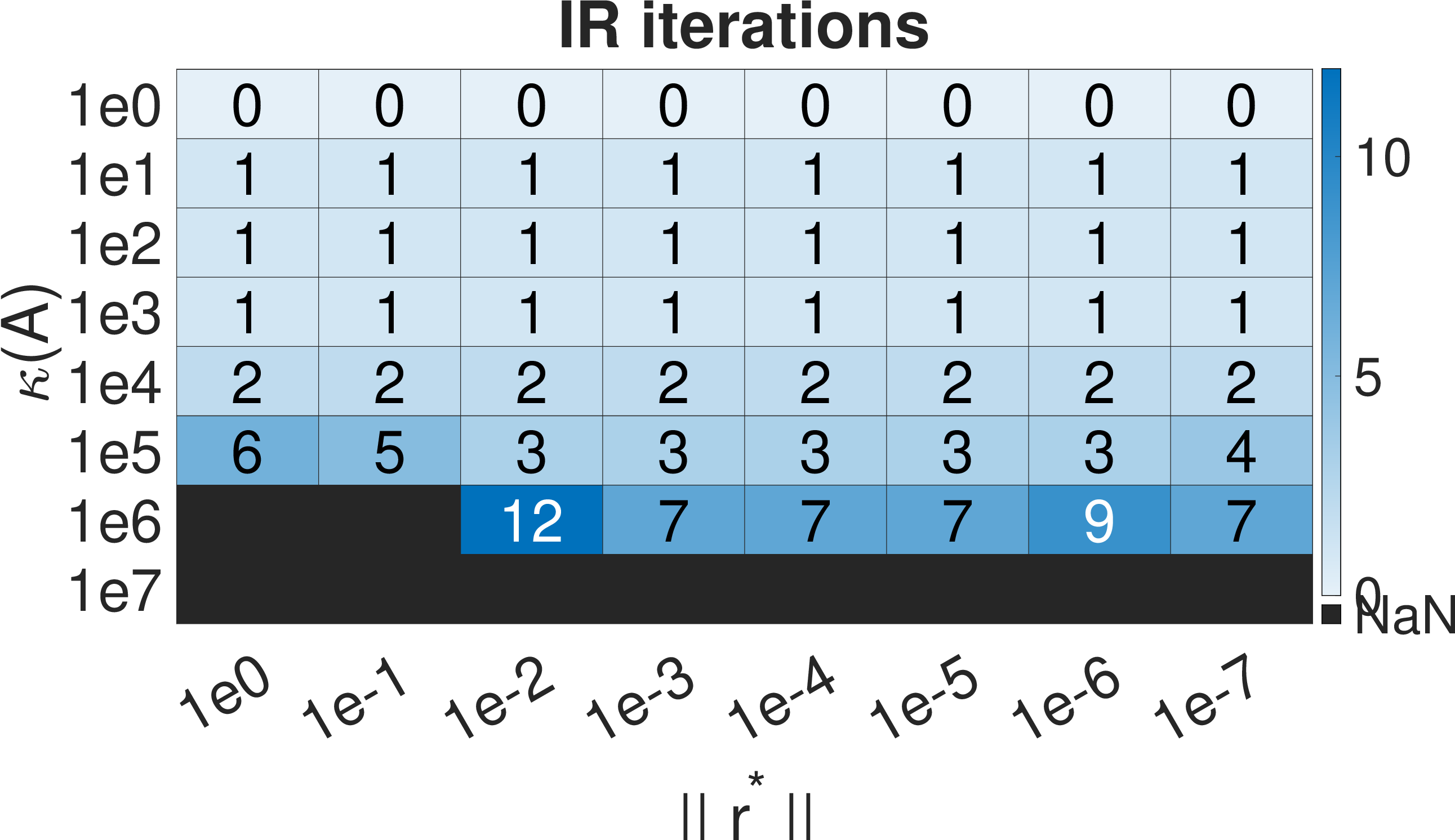}
  \caption{augmented approach, (double, single)}
\end{subfigure}
\begin{subfigure}[t]{0.45\linewidth}
  \centering
 \includegraphics[width=\linewidth]{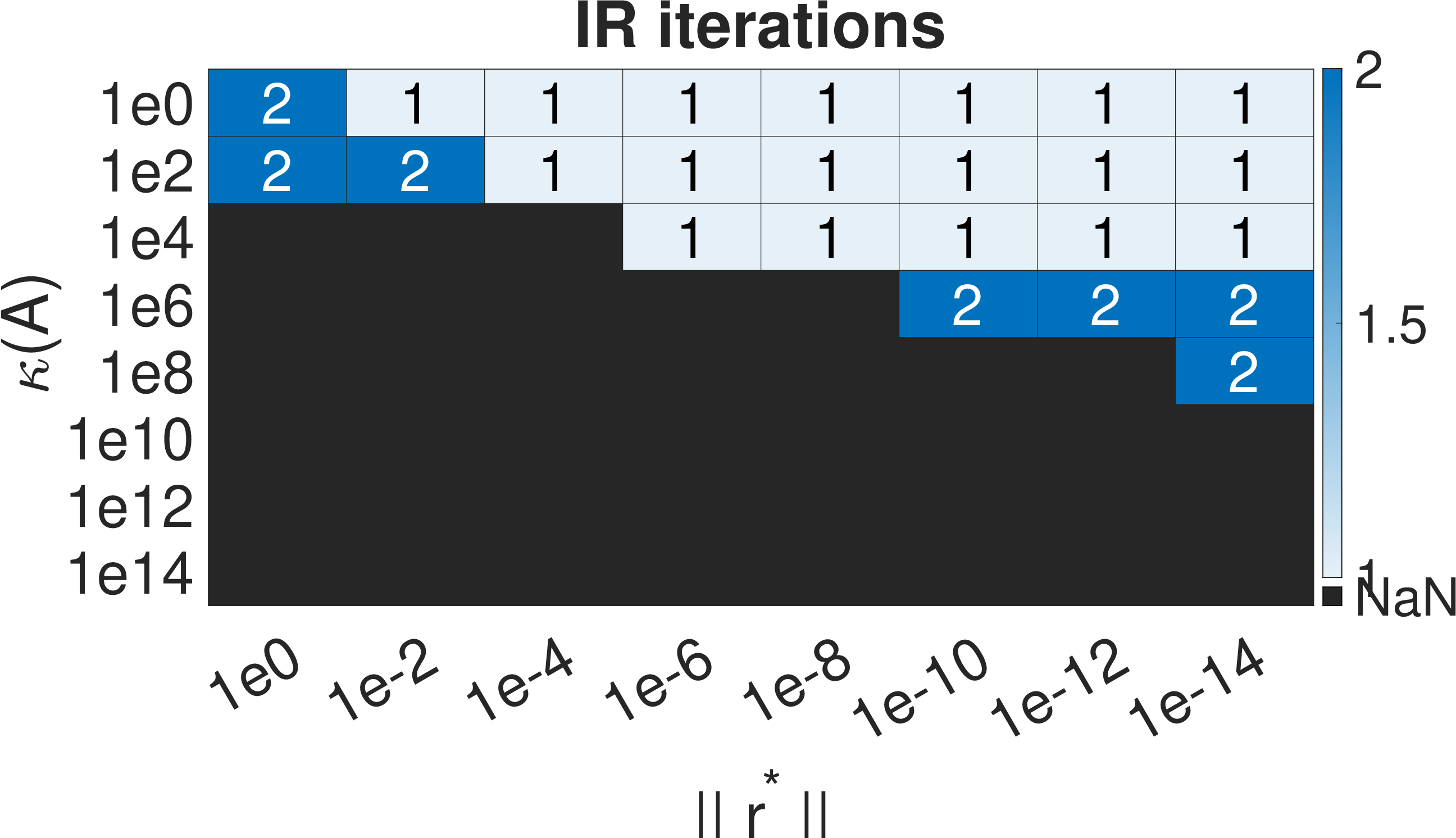}
  \caption{combined approach, (quad, double)}
\end{subfigure}
\begin{subfigure}[t]{0.45\linewidth}
  \centering
 \includegraphics[width=\linewidth]{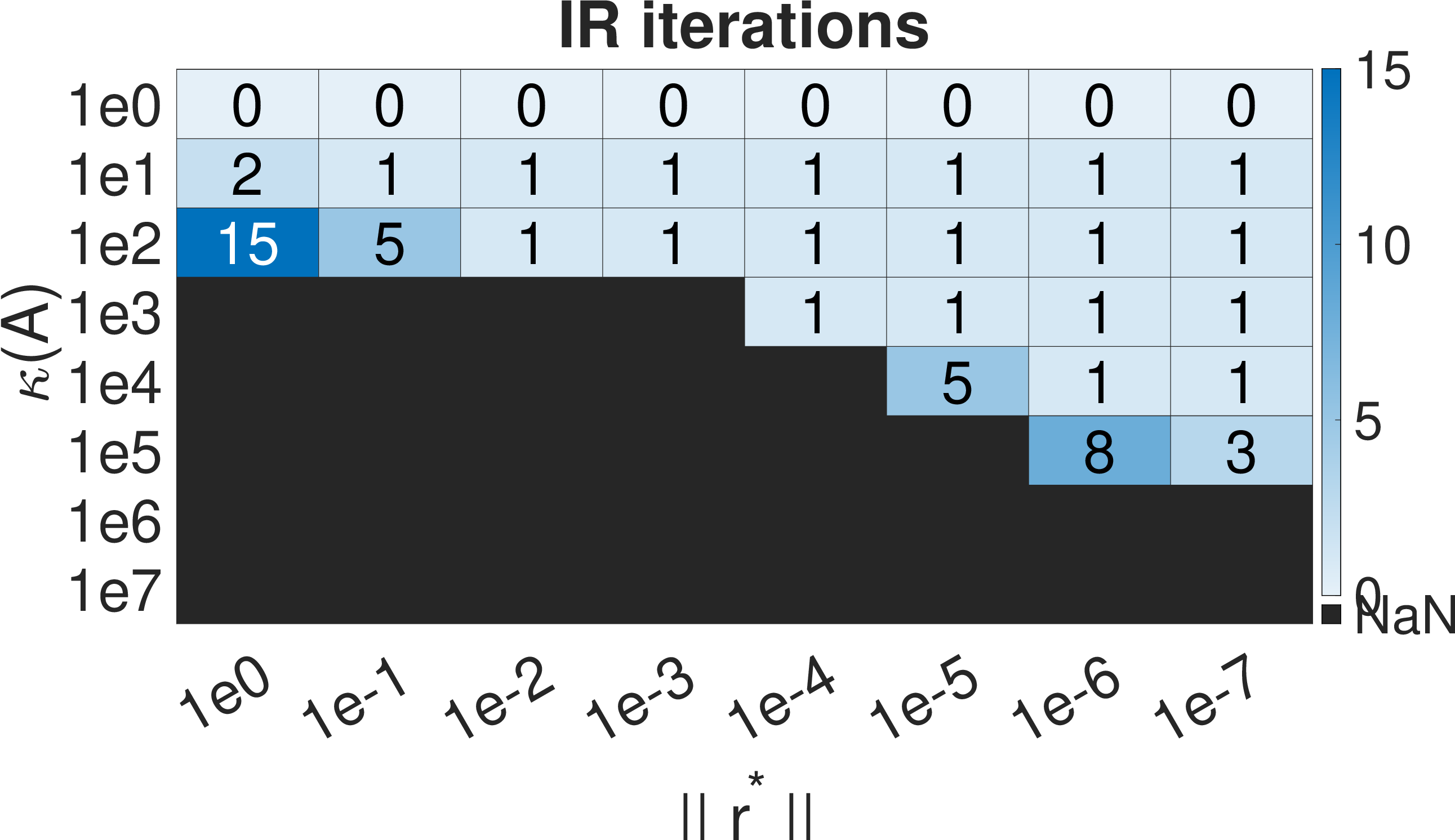}
  \caption{combined approach, (double, single)}
\end{subfigure}
    \caption{Count of iterative refinement iterations using preconditioned LSQR and GMRES for the LS, augmented system, and combined approaches when $u_r$ is set to quadruple and $u$ is set to double (left panels), and $u_r$ is set to double and $u$ is set to single (right). Black denotes when LSIR does not converge in 30 iterations.}
    \label{fig:ir_iterations_iterative}
\end{figure}

\begin{figure}
    \centering
\begin{subfigure}[t]{0.45\linewidth}
  \centering
 \includegraphics[width=\linewidth]{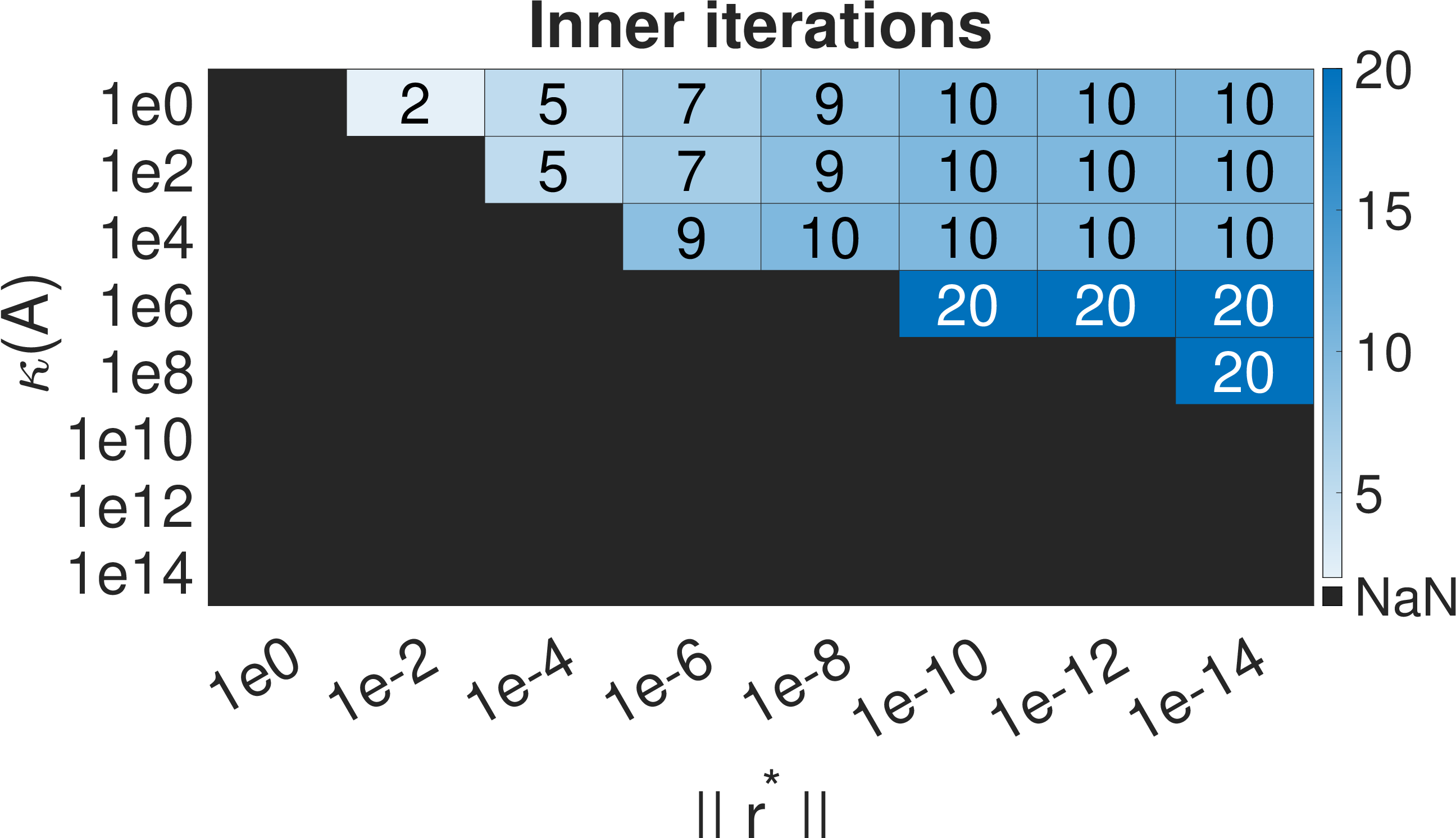}
  \caption{LS approach, (quad, double)}
\end{subfigure}
\begin{subfigure}[t]{0.45\linewidth}
  \centering
 \includegraphics[width=\linewidth]{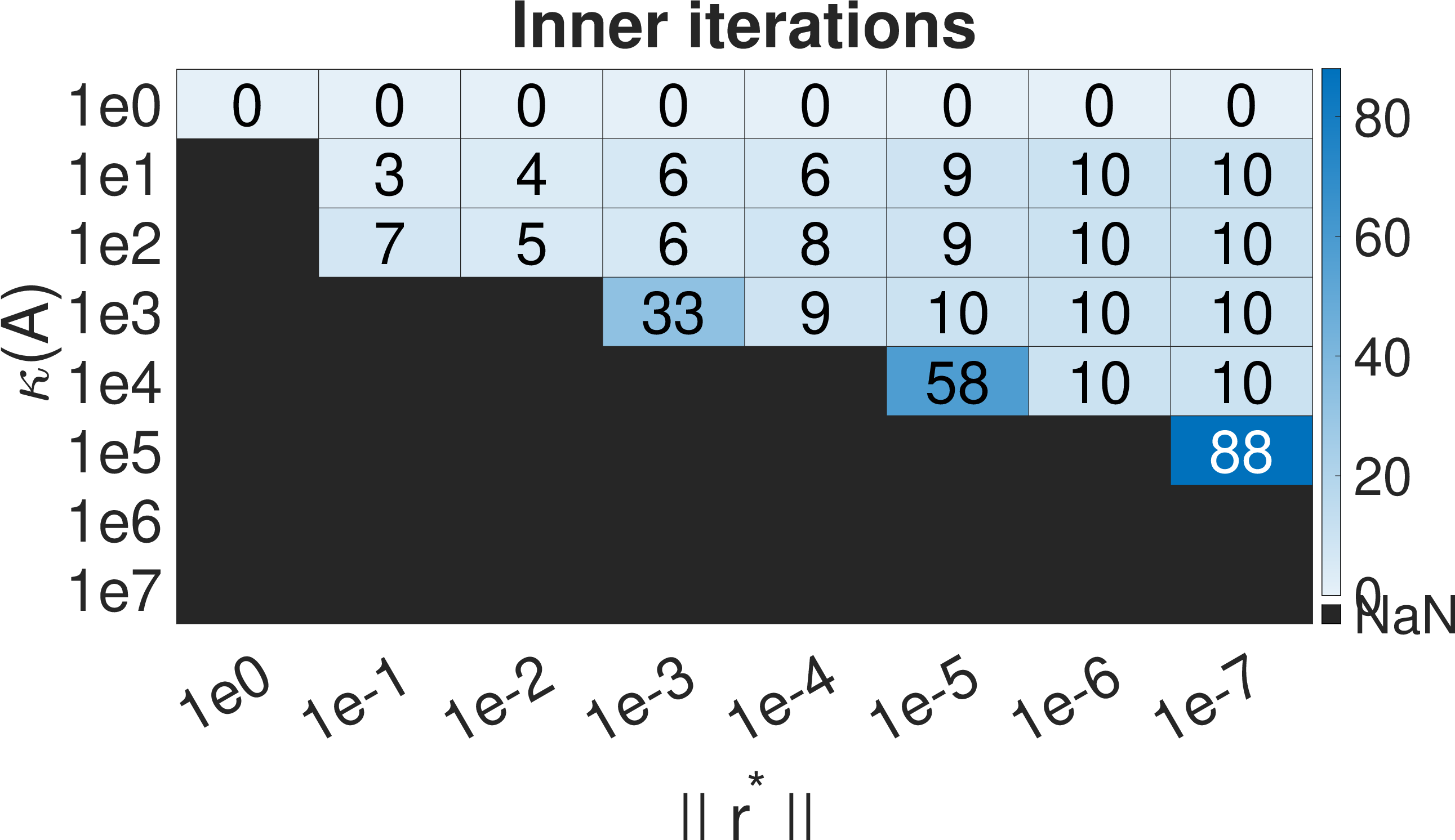}
  \caption{LS approach, (double, single)}
\end{subfigure}
\begin{subfigure}[t]{0.45\linewidth}
  \centering
 \includegraphics[width=\linewidth]{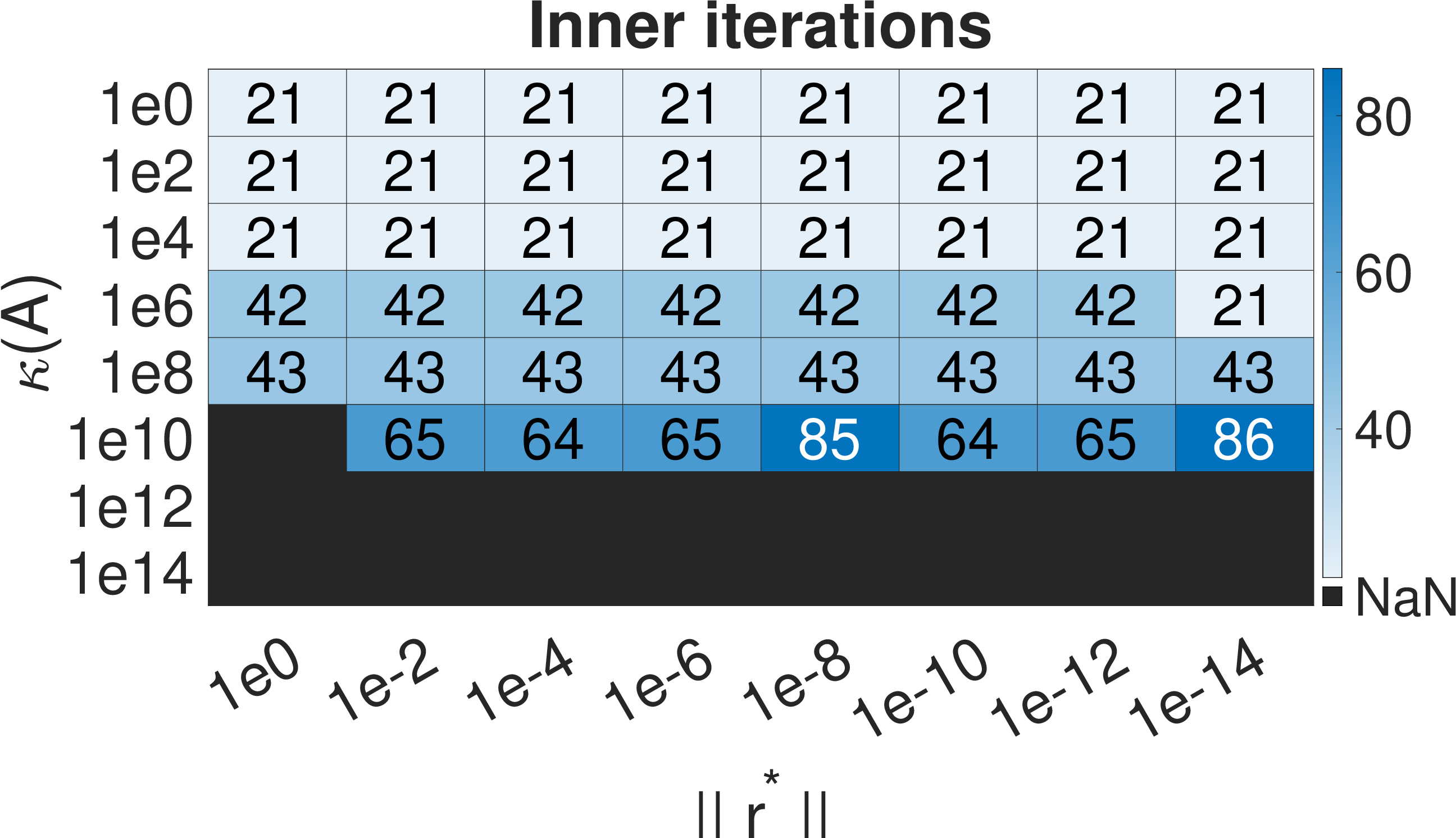}
  \caption{augmented approach, (quad, double)}
\end{subfigure}
\begin{subfigure}[t]{0.45\linewidth}
  \centering
 \includegraphics[width=\linewidth]{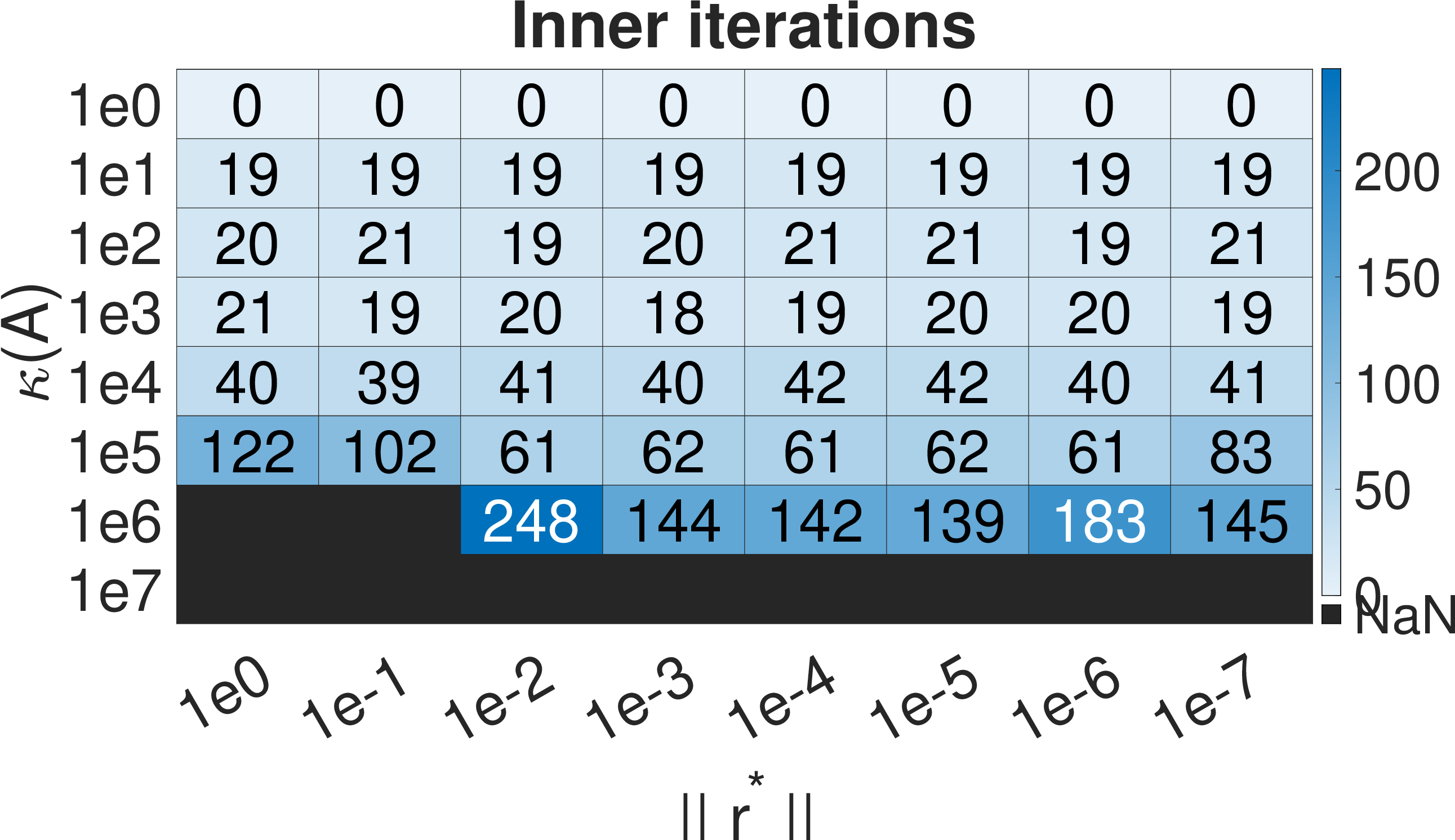}
  \caption{augmented approach, (double, single)}
\end{subfigure}
\begin{subfigure}[t]{0.45\linewidth}
  \centering
 \includegraphics[width=\linewidth]{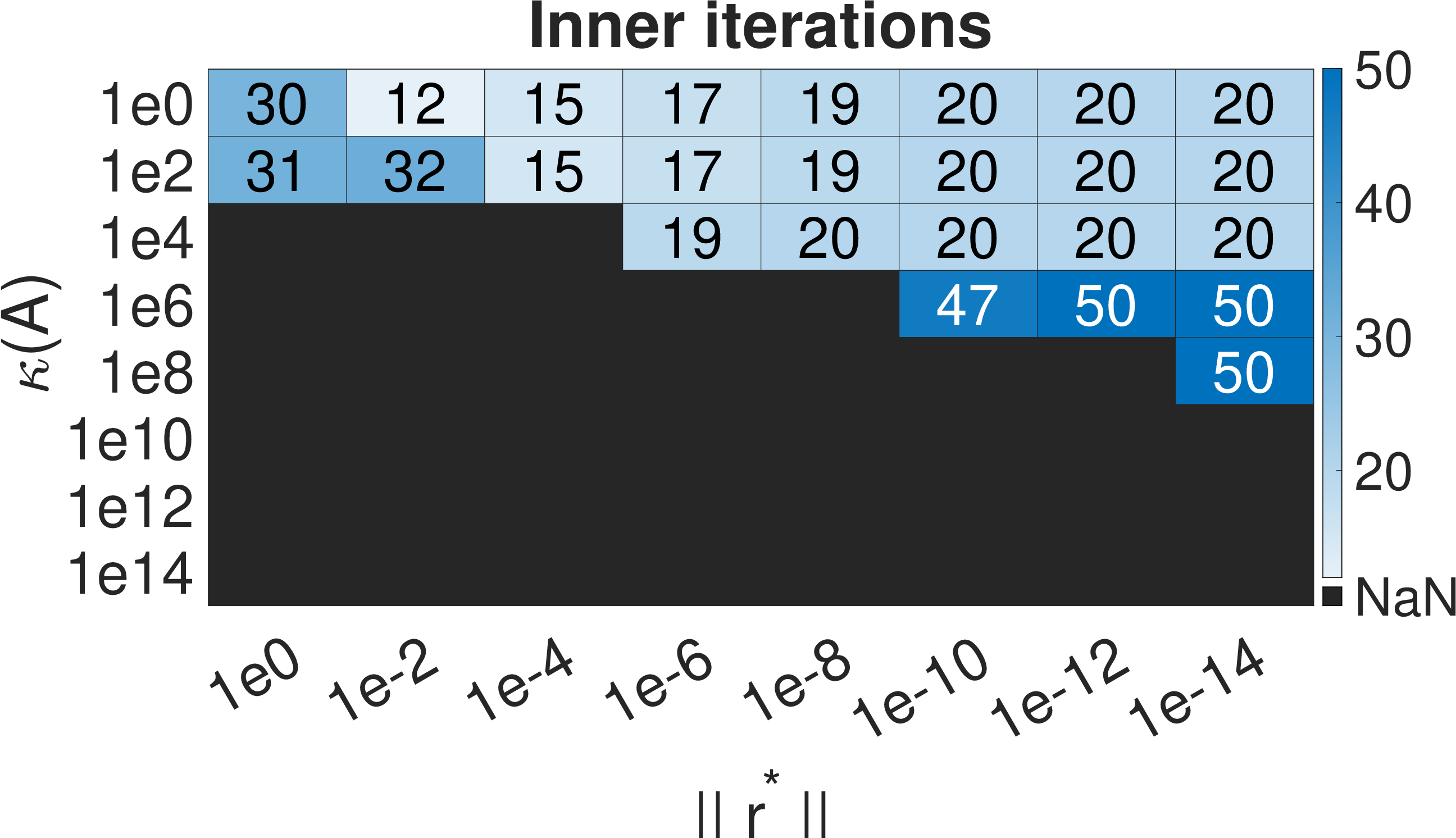}
  \caption{combined approach, (quad, double)}
\end{subfigure}
\begin{subfigure}[t]{0.45\linewidth}
  \centering
 \includegraphics[width=\linewidth]{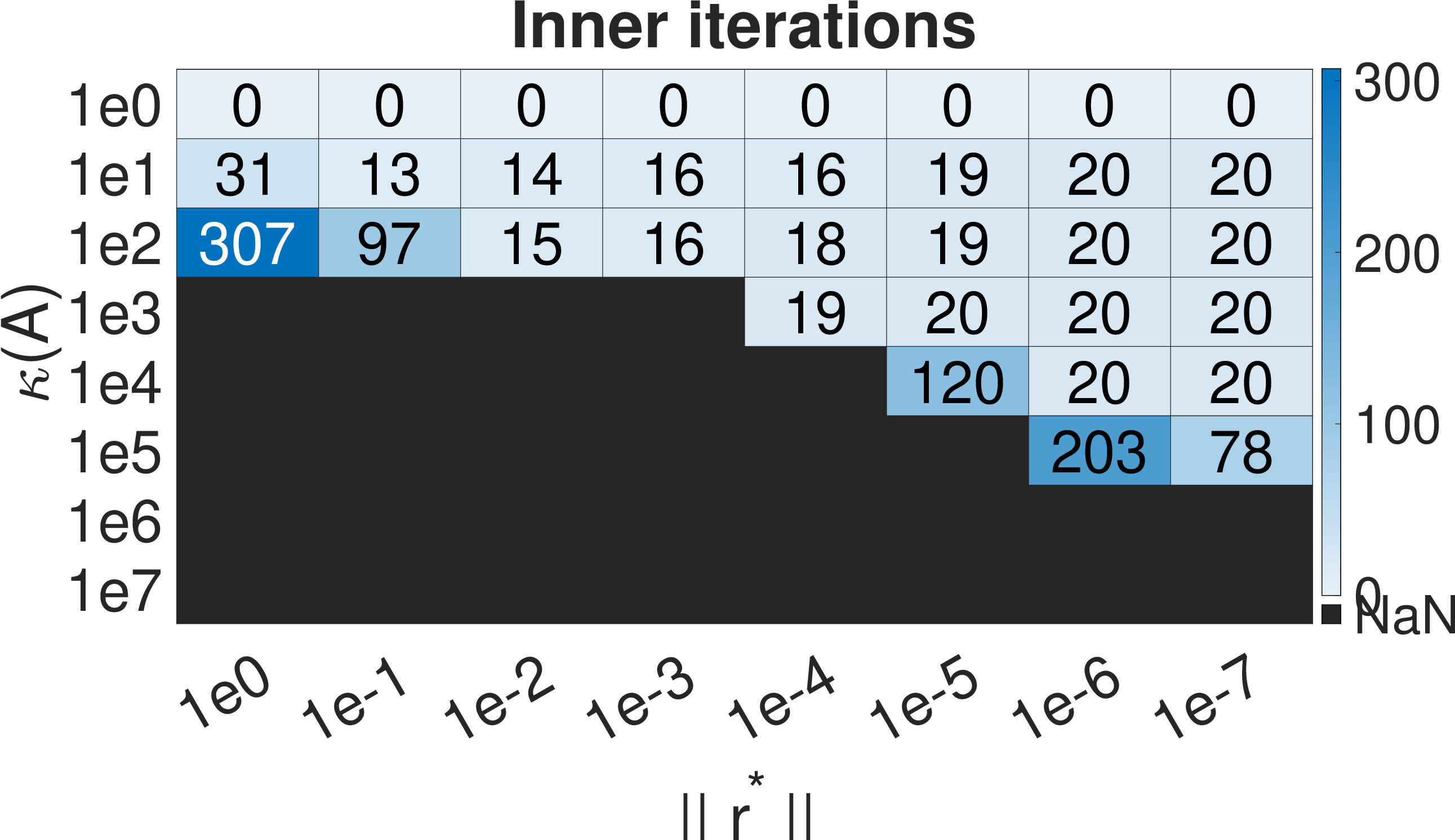}
  \caption{combined approach, (double, single)}
\end{subfigure}
    \caption{As in Figure~\ref{fig:ir_iterations_iterative}, but for the cumulative number of inner iterative refinement iterations, that is, LSQR iterations for the LS and combined approaches, and GMRES iterations for the augmented approach. Black denotes when LSIR does not converge in 30 iterations.}
    \label{fig:ir_inner_iterations_iterative}
\end{figure}

\begin{figure}
    \centering
\begin{subfigure}[t]{0.45\linewidth}
  \centering
 \includegraphics[width=\linewidth]{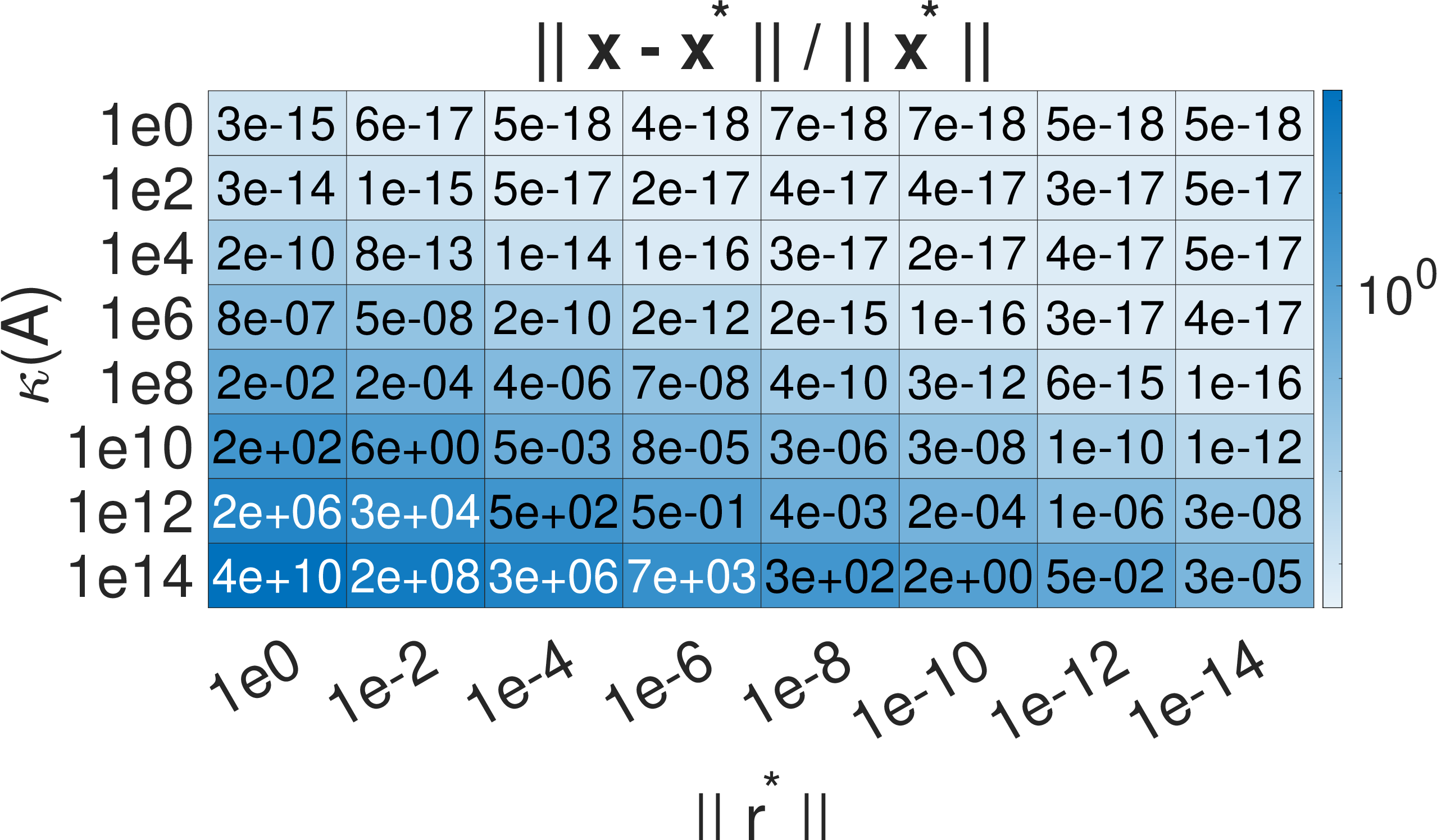}
  \caption{LS approach, (quad, double)}
\end{subfigure}
\begin{subfigure}[t]{0.45\linewidth}
  \centering
 \includegraphics[width=\linewidth]{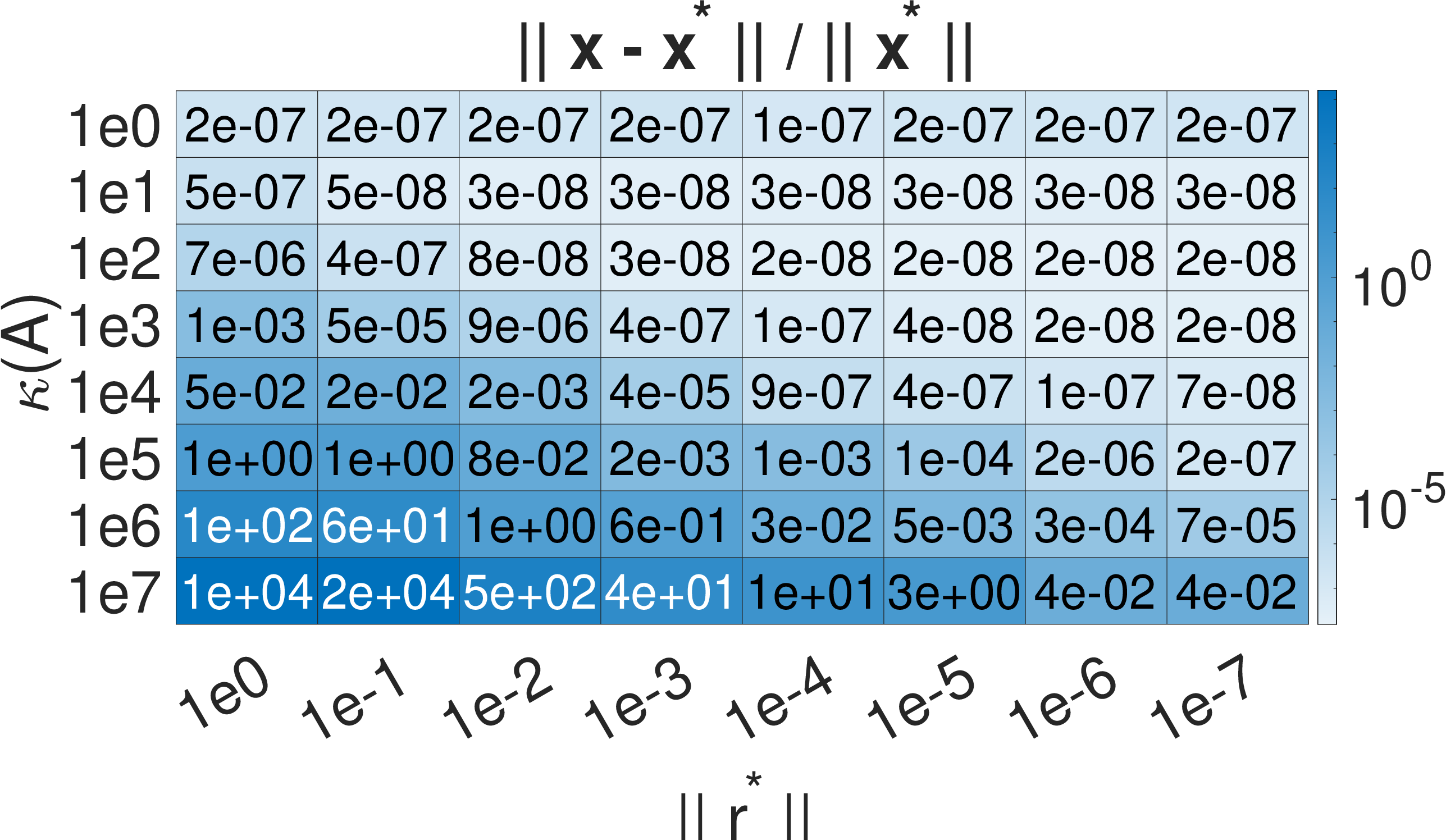}
  \caption{LS approach, (double, single)}
\end{subfigure}
\begin{subfigure}[t]{0.45\linewidth}
  \centering
 \includegraphics[width=\linewidth]{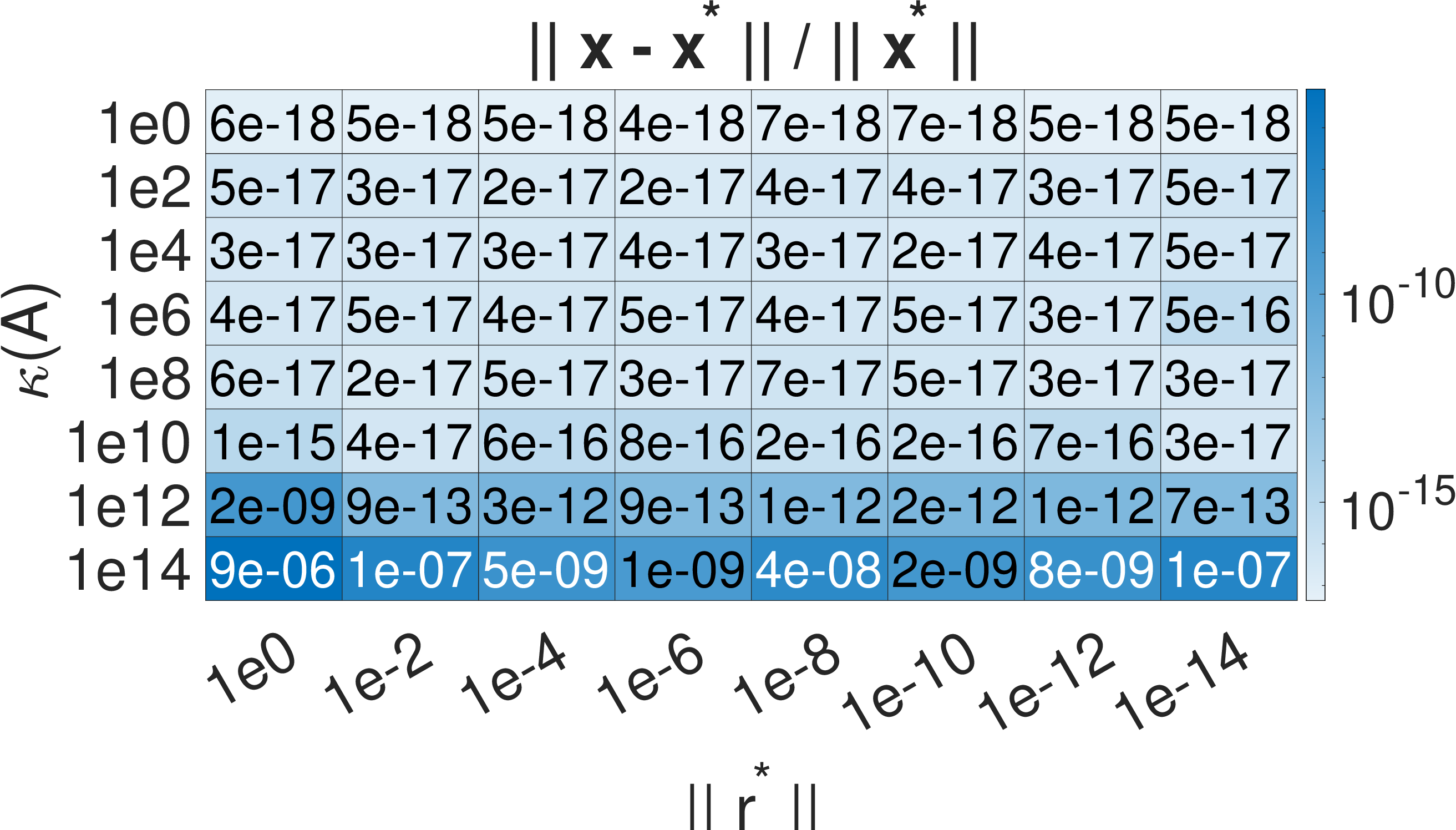}
  \caption{augmented approach, (quad, double)}
\end{subfigure}
\begin{subfigure}[t]{0.45\linewidth}
  \centering
 \includegraphics[width=\linewidth]{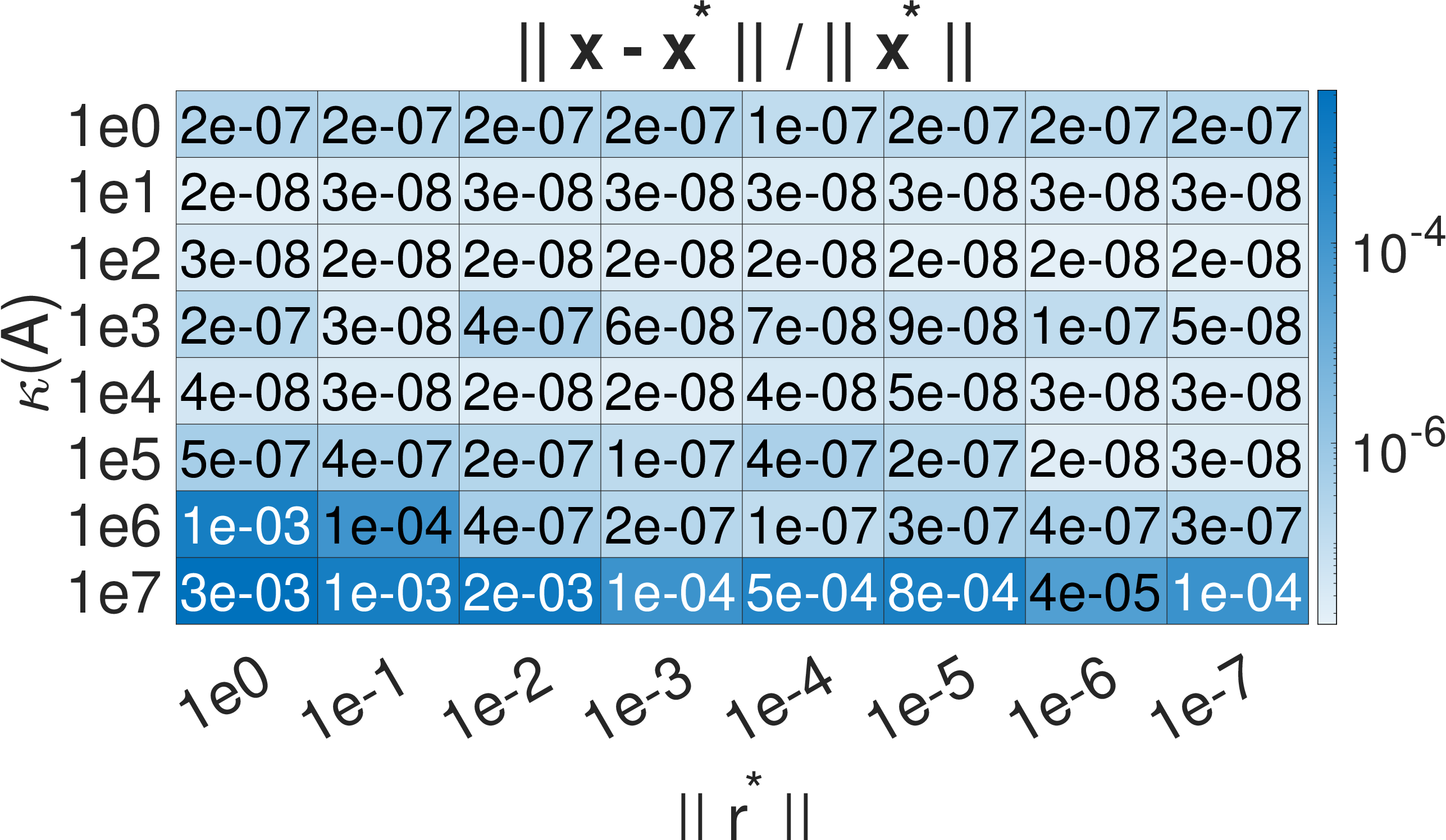}
  \caption{augmented approach, (double, single)}
\end{subfigure}
\begin{subfigure}[t]{0.45\linewidth}
  \centering
 \includegraphics[width=\linewidth]{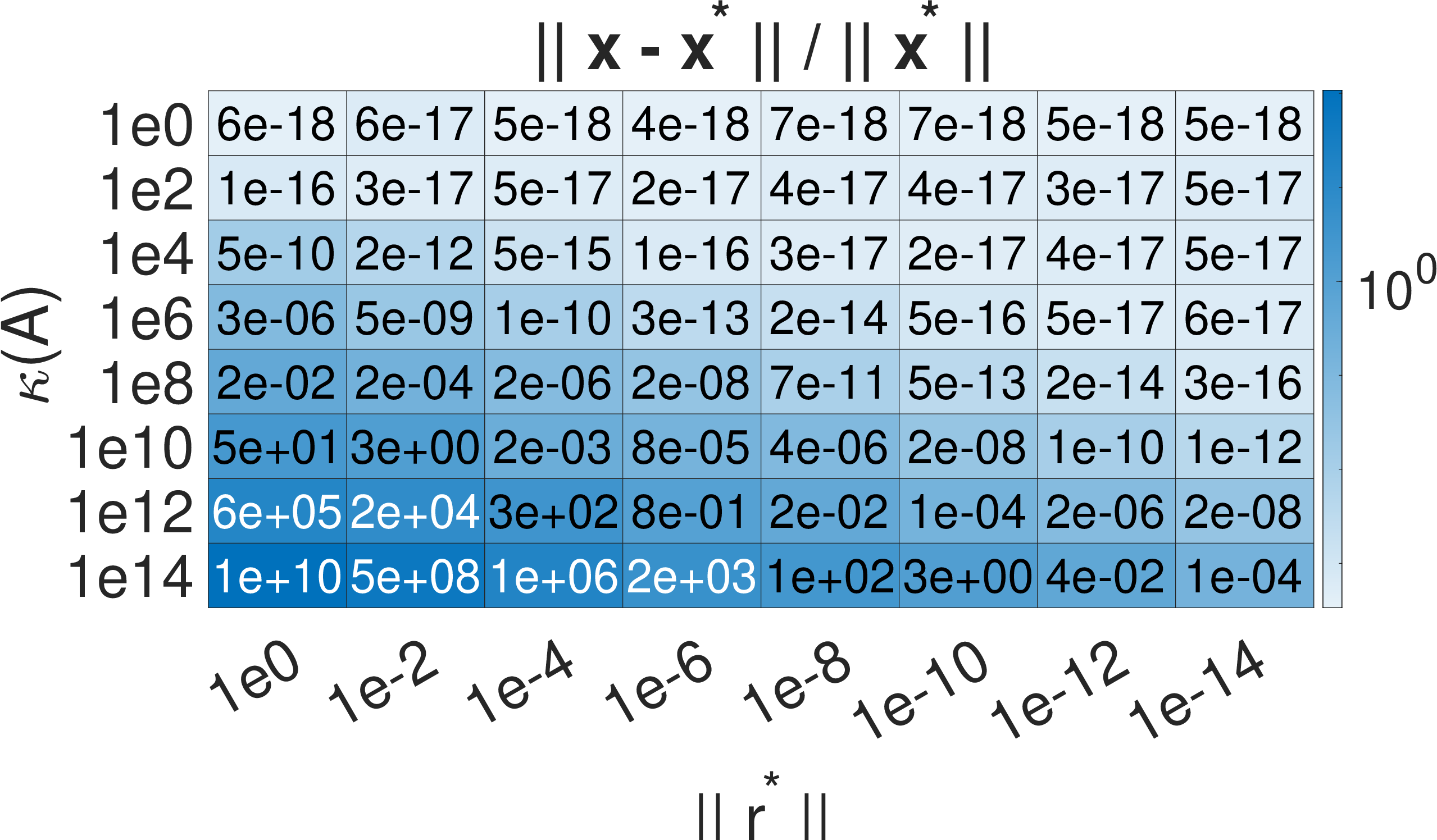}
  \caption{combined approach, (quad, double)}
\end{subfigure}
\begin{subfigure}[t]{0.45\linewidth}
  \centering
 \includegraphics[width=\linewidth]{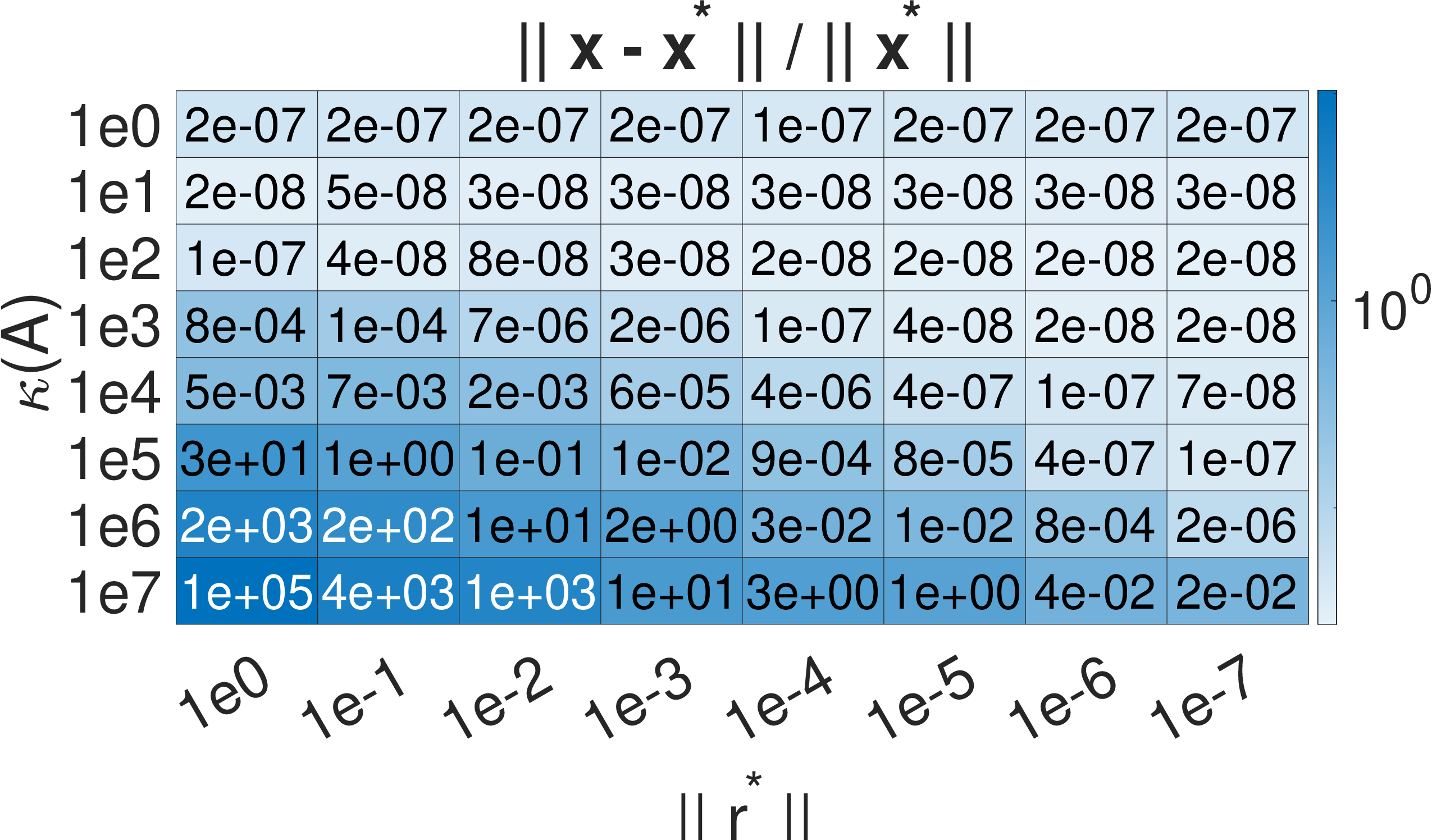}
  \caption{combined approach, (double, single)}
\end{subfigure}
    \caption{As in Figure~\ref{fig:ir_iterations_iterative}, but for the relative error in $x$.}
    \label{fig:x_error_iterative}
\end{figure}

\begin{figure}
    \centering
\begin{subfigure}[t]{0.45\linewidth}
  \centering
 \includegraphics[width=\linewidth]{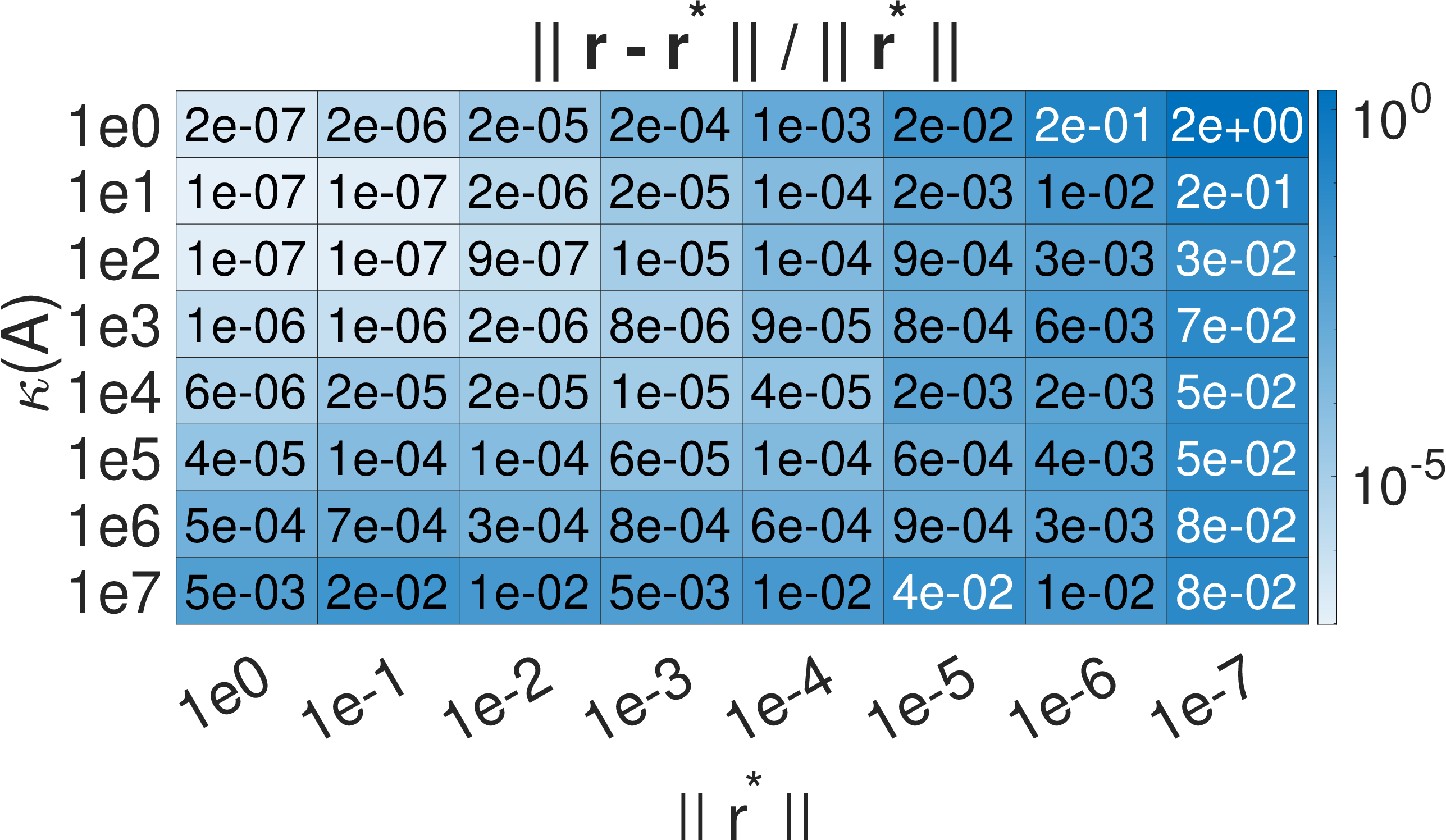}
  \caption{LS approach, (quad, double)}
\end{subfigure}
\begin{subfigure}[t]{0.45\linewidth}
  \centering
 \includegraphics[width=\linewidth]{plots_rand_precond/ls_us_r_error_solviter.eps}
  \caption{LS approach, (double, single)}
\end{subfigure}
\begin{subfigure}[t]{0.45\linewidth}
  \centering
 \includegraphics[width=\linewidth]{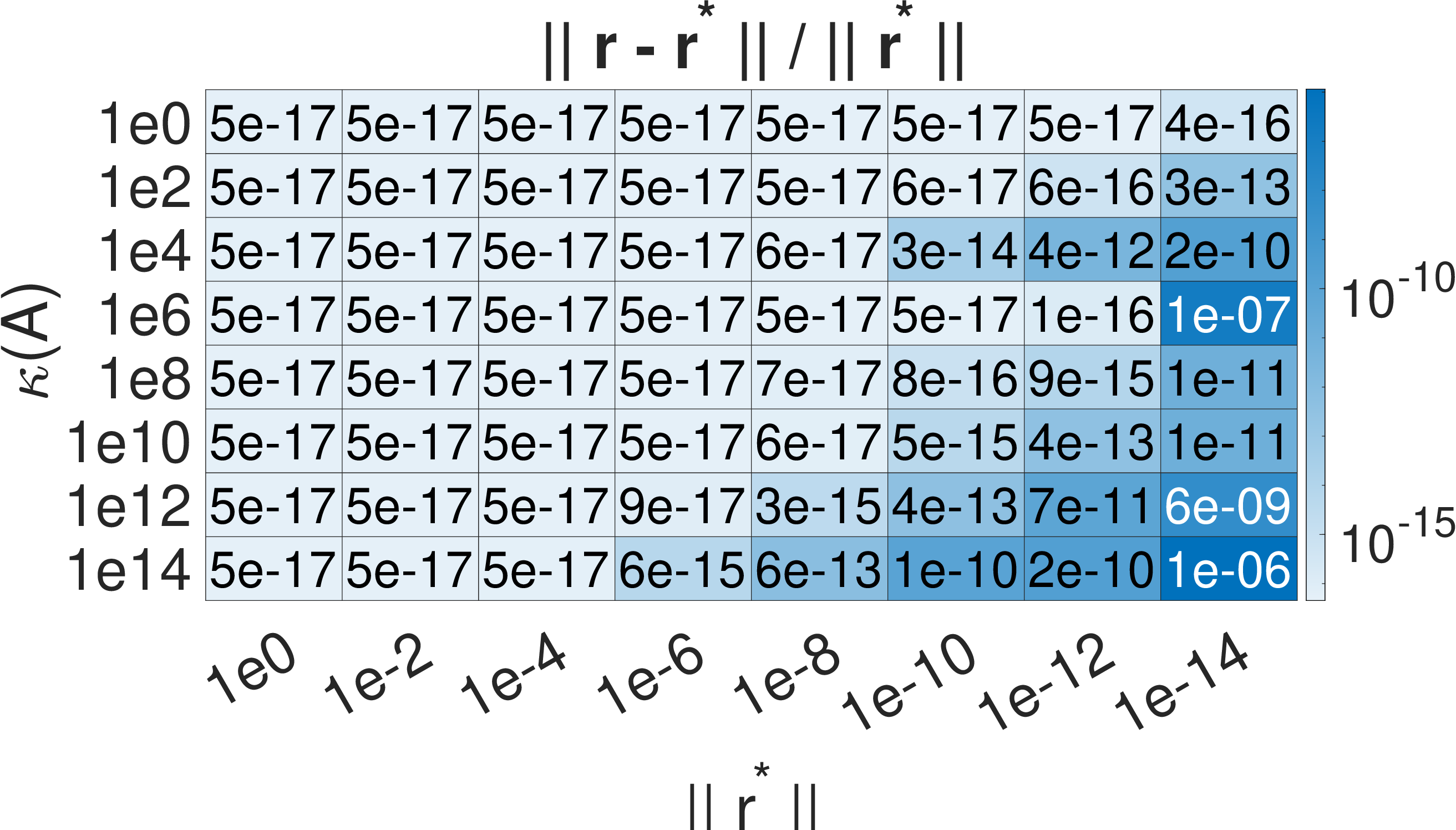}
  \caption{augmented approach, (quad, double)}
\end{subfigure}
\begin{subfigure}[t]{0.45\linewidth}
  \centering
 \includegraphics[width=\linewidth]{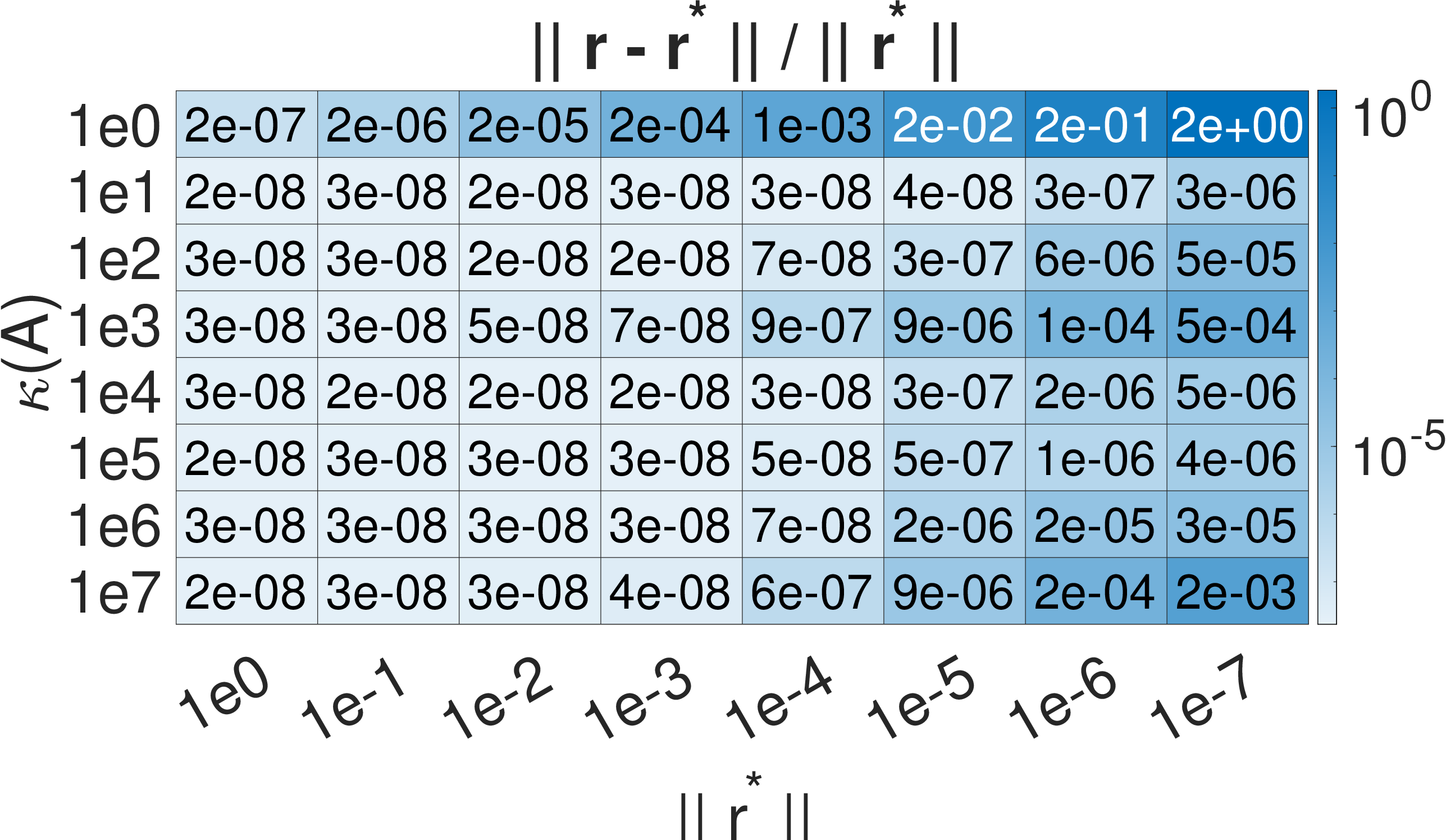}
  \caption{augmented approach, (double, single)}
\end{subfigure}
\begin{subfigure}[t]{0.45\linewidth}
  \centering
 \includegraphics[width=\linewidth]{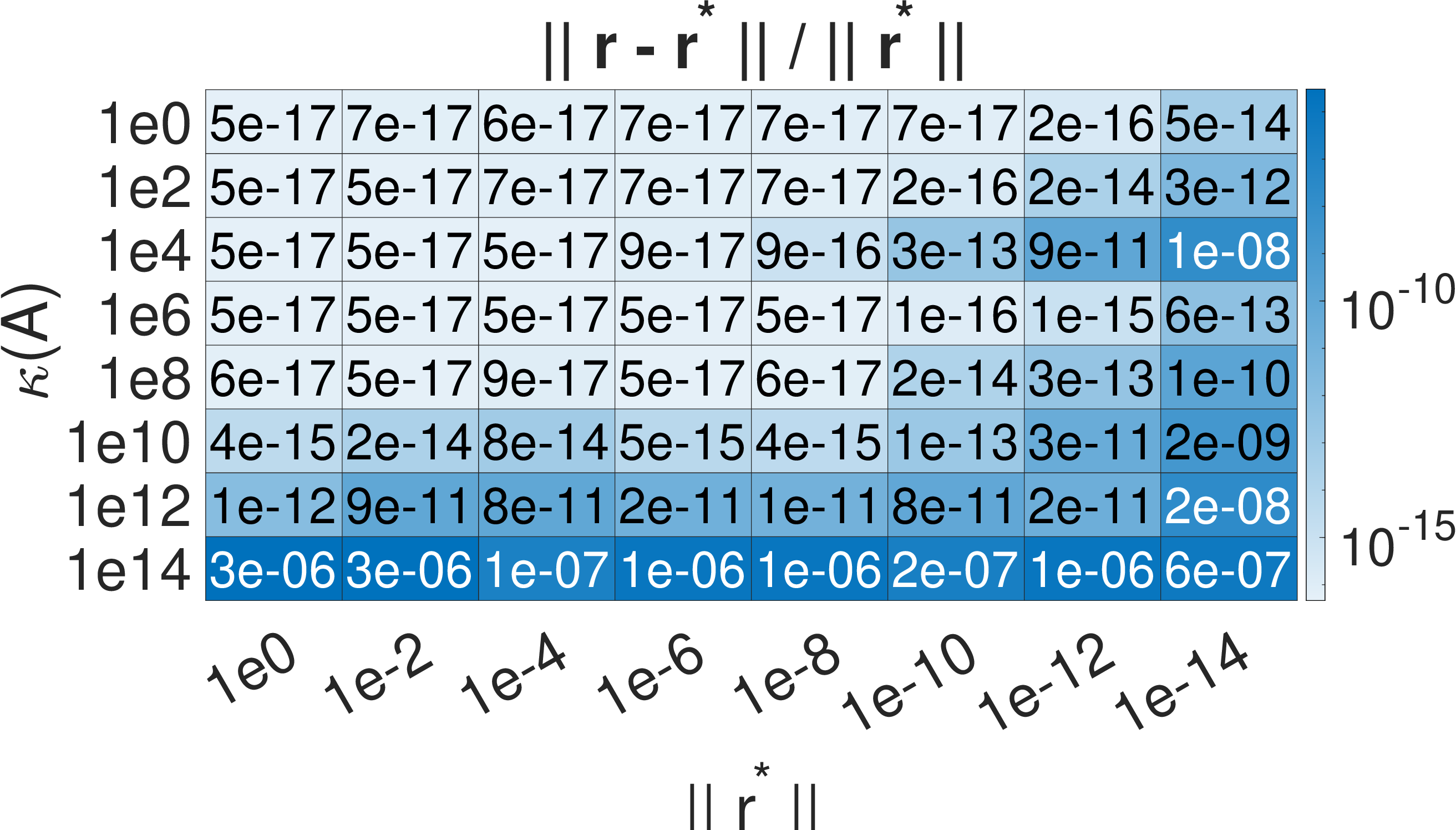}
  \caption{combined approach, (quad, double)}
\end{subfigure}
\begin{subfigure}[t]{0.45\linewidth}
  \centering
 \includegraphics[width=\linewidth]{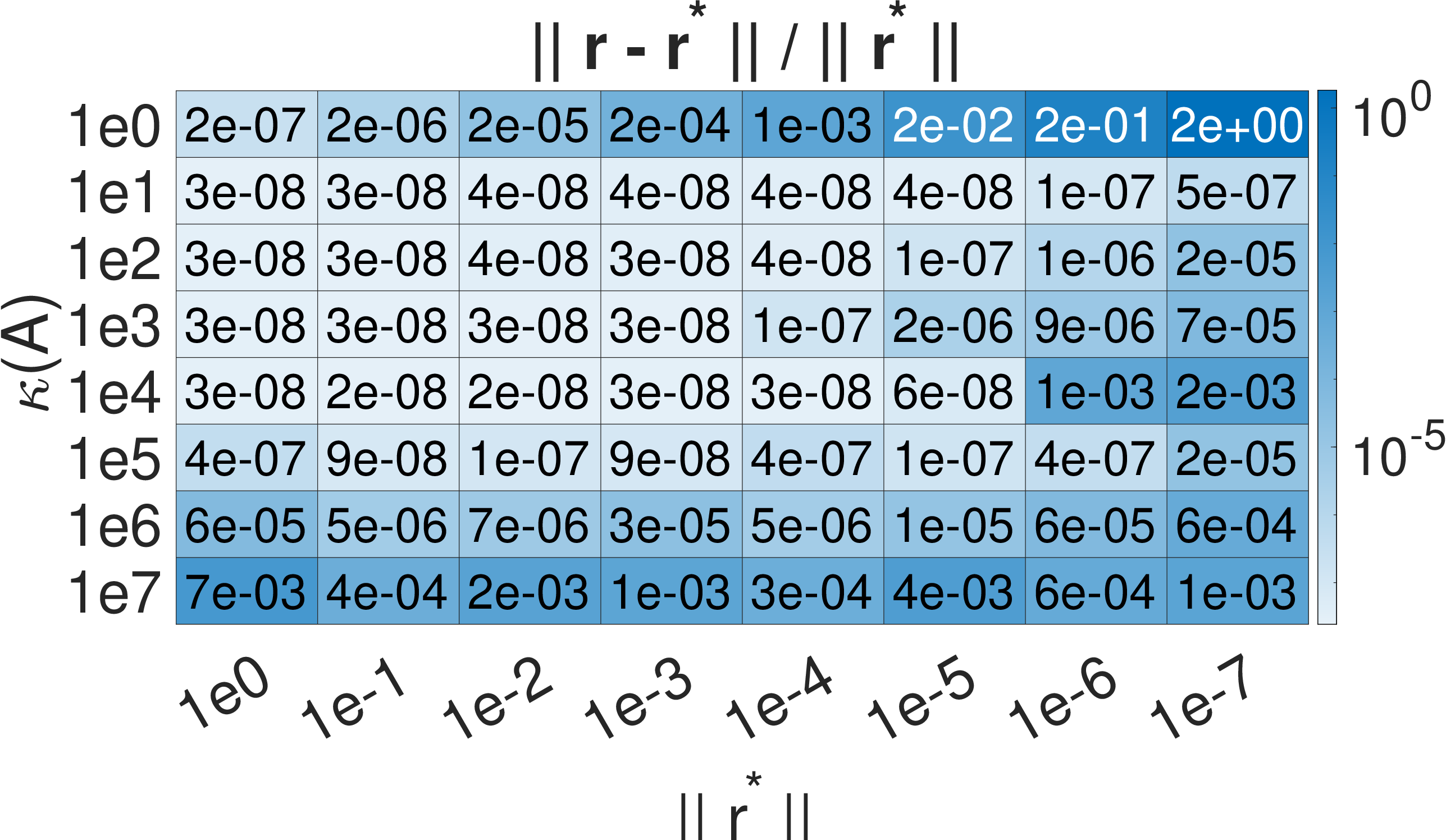}
  \caption{combined approach, (double, single)}
\end{subfigure}
    \caption{As in Figure~\ref{fig:ir_iterations_iterative}, but for the relative error in $r$.}
    \label{fig:r_error_iterative}
\end{figure}

\begin{figure}
    \centering
\begin{subfigure}[t]{0.45\linewidth}
  \centering
 \includegraphics[width=\linewidth]{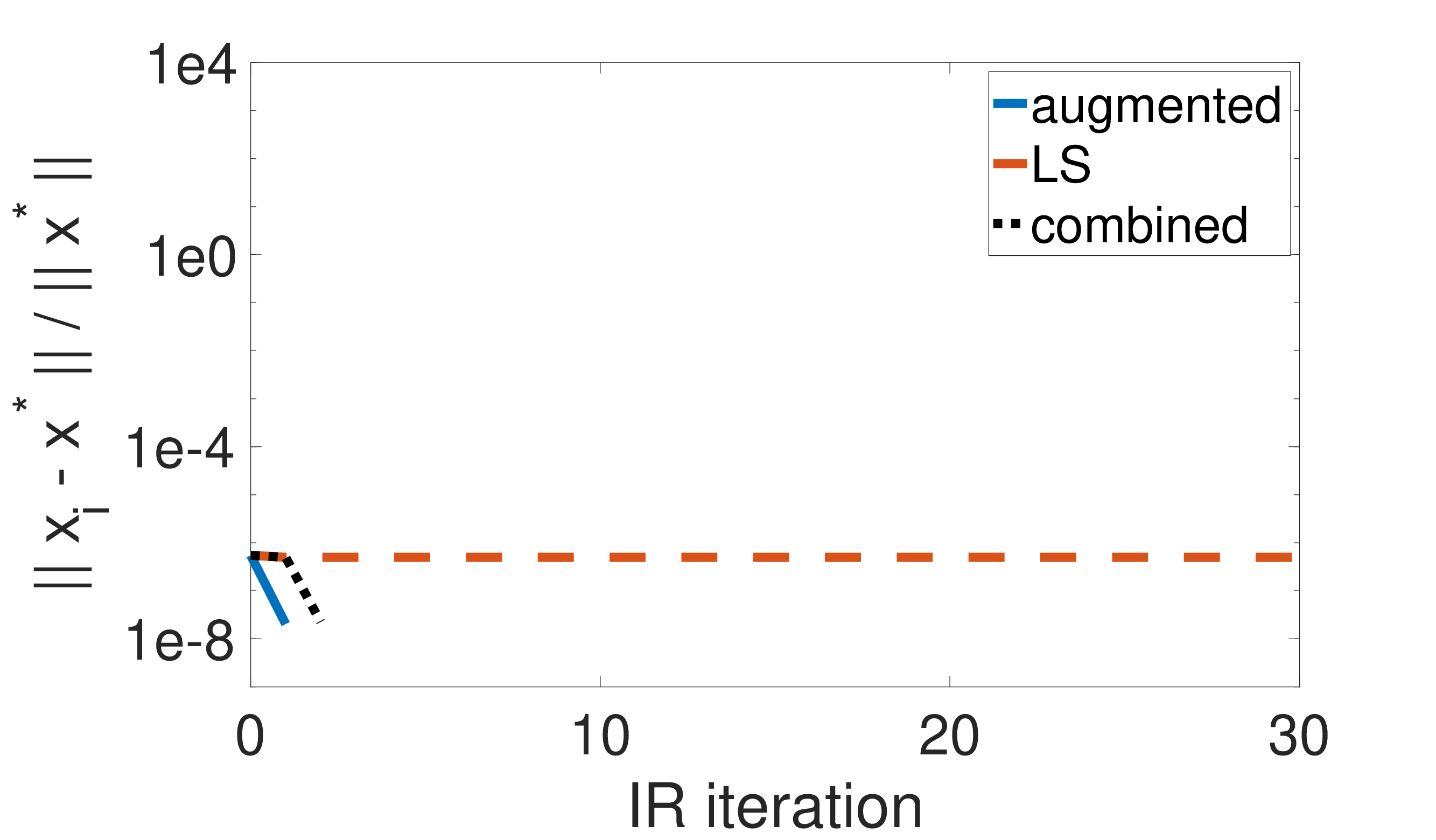}
  \caption{$\kappa_2(A)=$1e1, $\Vert r^* \Vert_2=$1e0 }
\end{subfigure}
\begin{subfigure}[t]{0.45\linewidth}
  \centering
 \includegraphics[width=\linewidth]{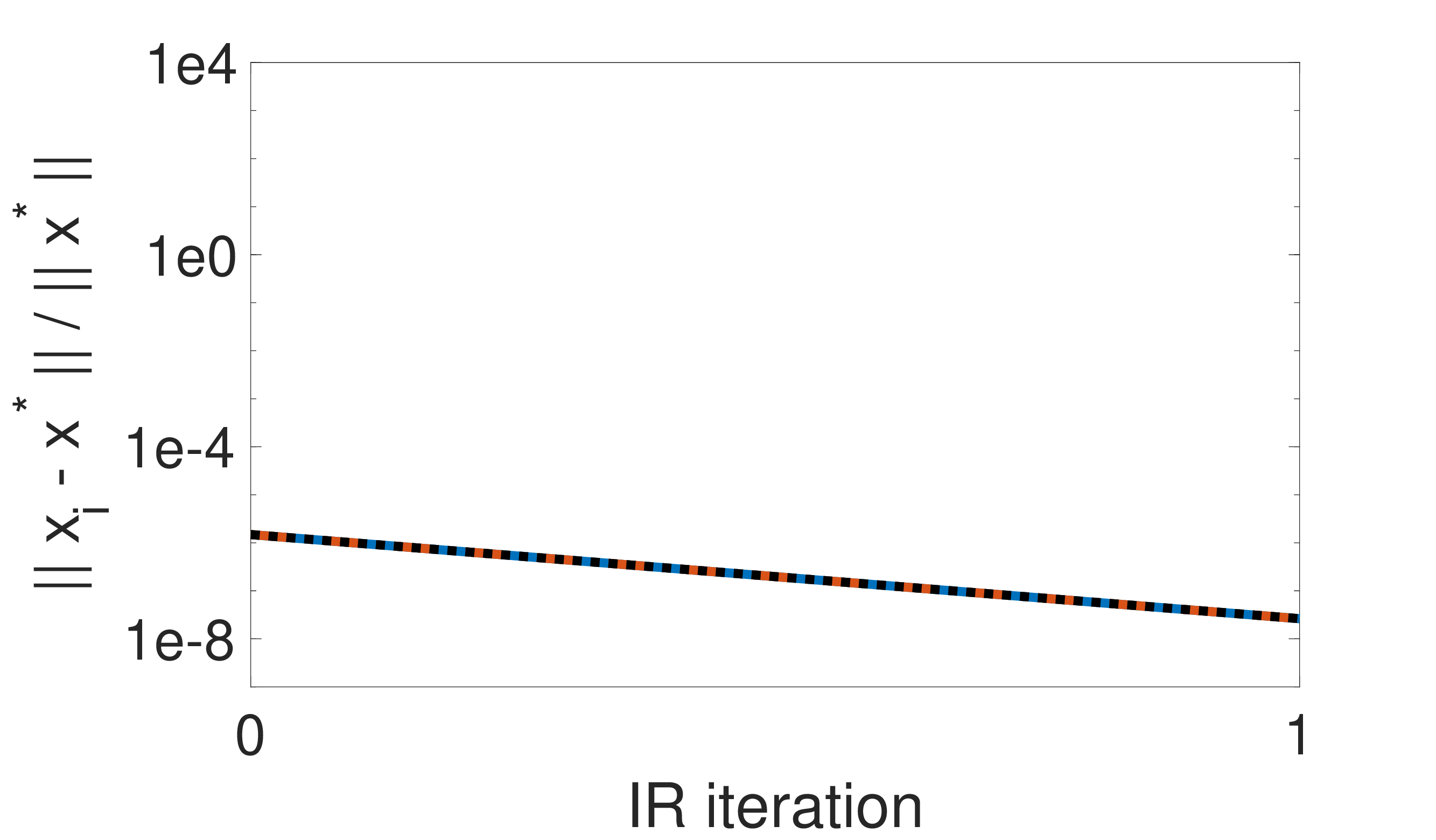}
  \caption{$\kappa_2(A)=$1e1, $\Vert r^* \Vert_2=$1e-6}
\end{subfigure}
\begin{subfigure}[t]{0.45\linewidth}
  \centering
 \includegraphics[width=\linewidth]{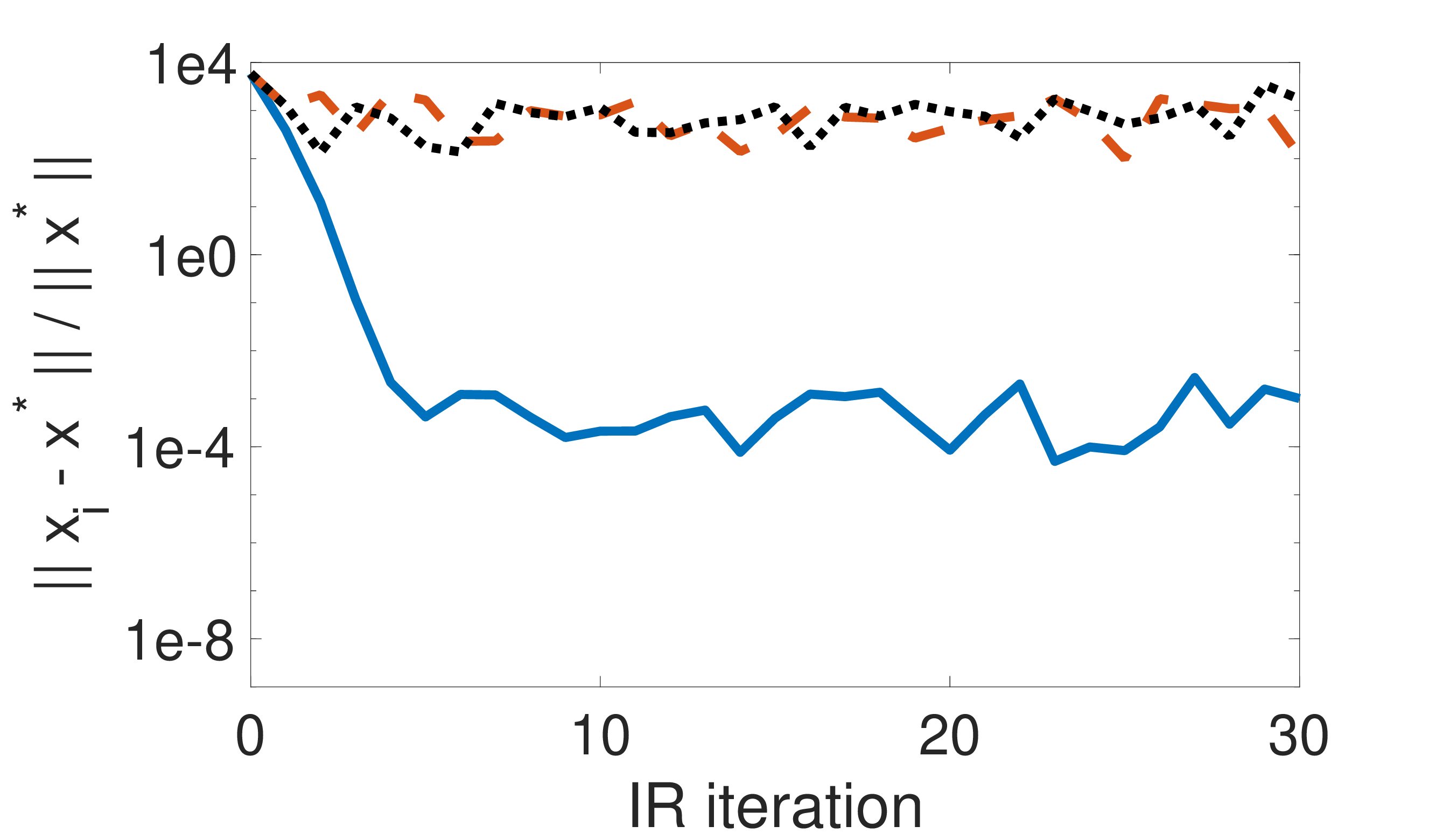}
  \caption{$\kappa_2(A)=$1e6, $\Vert r^* \Vert_2=$1e0 }
\end{subfigure}
\begin{subfigure}[t]{0.45\linewidth}
  \centering
 \includegraphics[width=\linewidth]{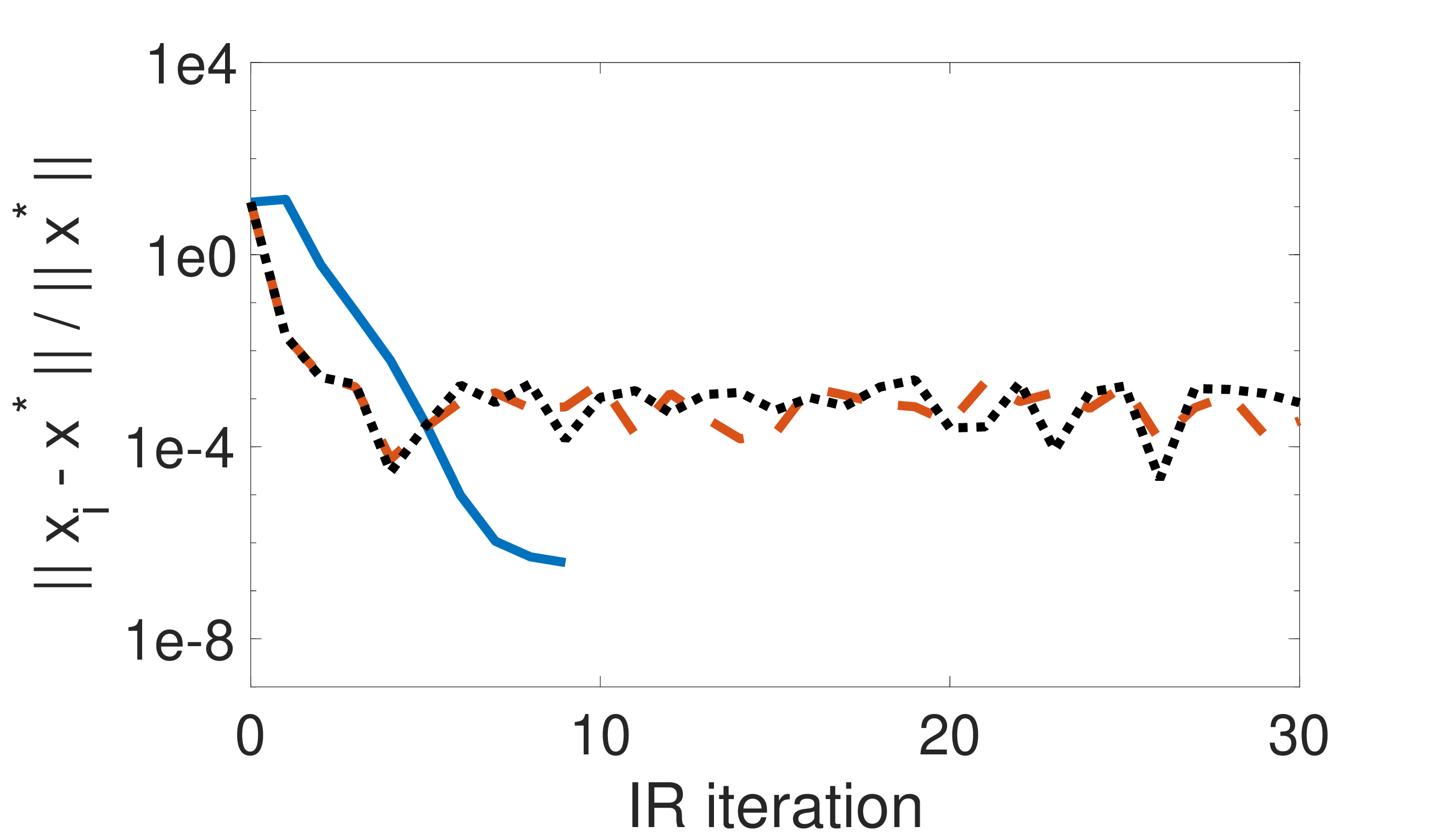}
  \caption{$\kappa_2(A)=$1e6, $\Vert r^* \Vert_2=$1e-6}
\end{subfigure}
    \caption{Relative error in $x$ at every IR iteration using iterative solvers for combinations of large and small $\kappa_2(A)$ and $\Vert r^* \Vert_2$; $u_r$ is set to double and $u$ is set to single.}
    \label{fig:x_conv_iterative}
\end{figure}

\begin{figure}
    \centering
\begin{subfigure}[t]{0.45\linewidth}
  \centering
 \includegraphics[width=\linewidth]{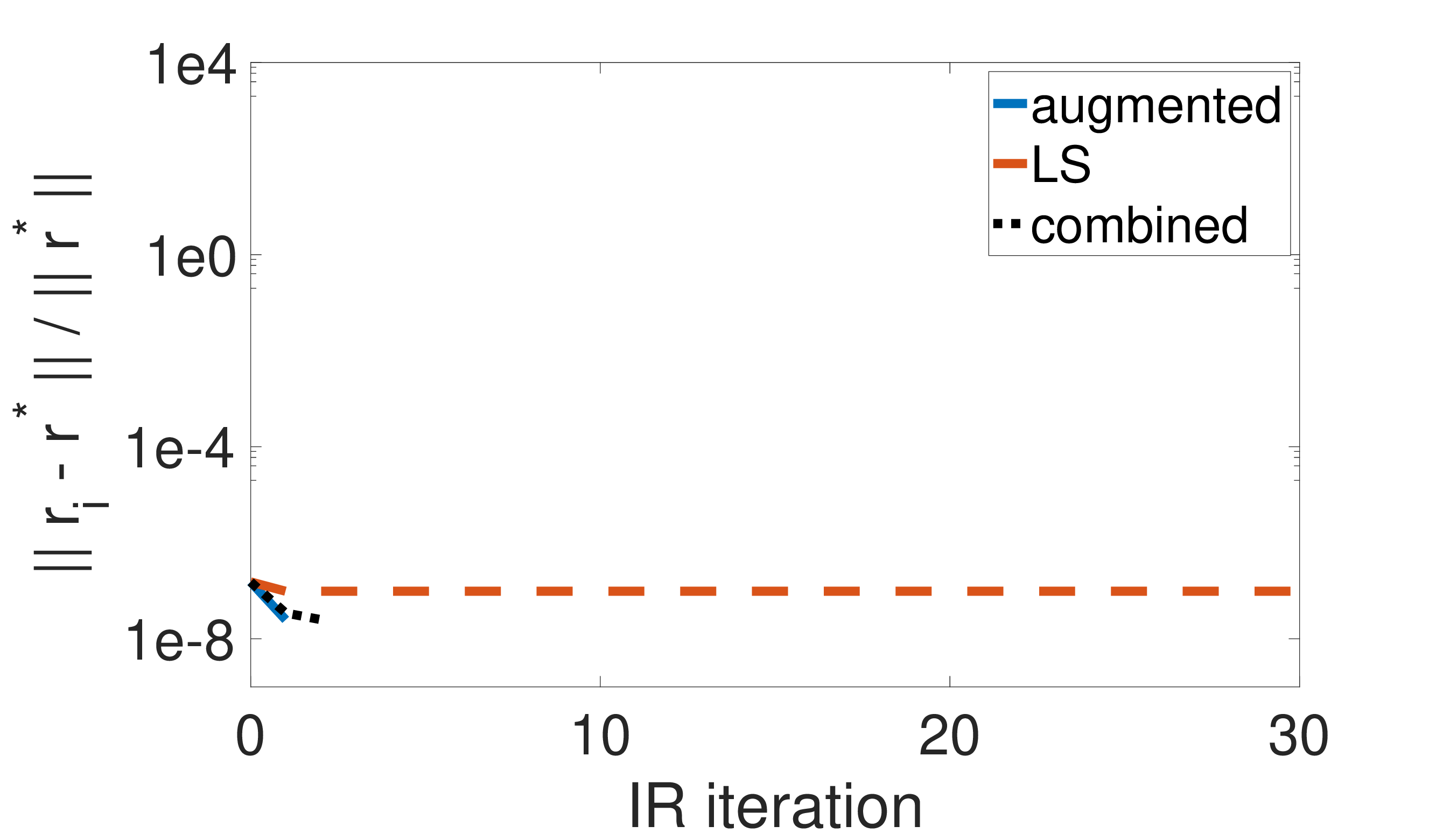}
  \caption{$\kappa_2(A)=$1e1, $\Vert r^* \Vert_2=$1e0 }
\end{subfigure}
\begin{subfigure}[t]{0.45\linewidth}
  \centering
 \includegraphics[width=\linewidth]{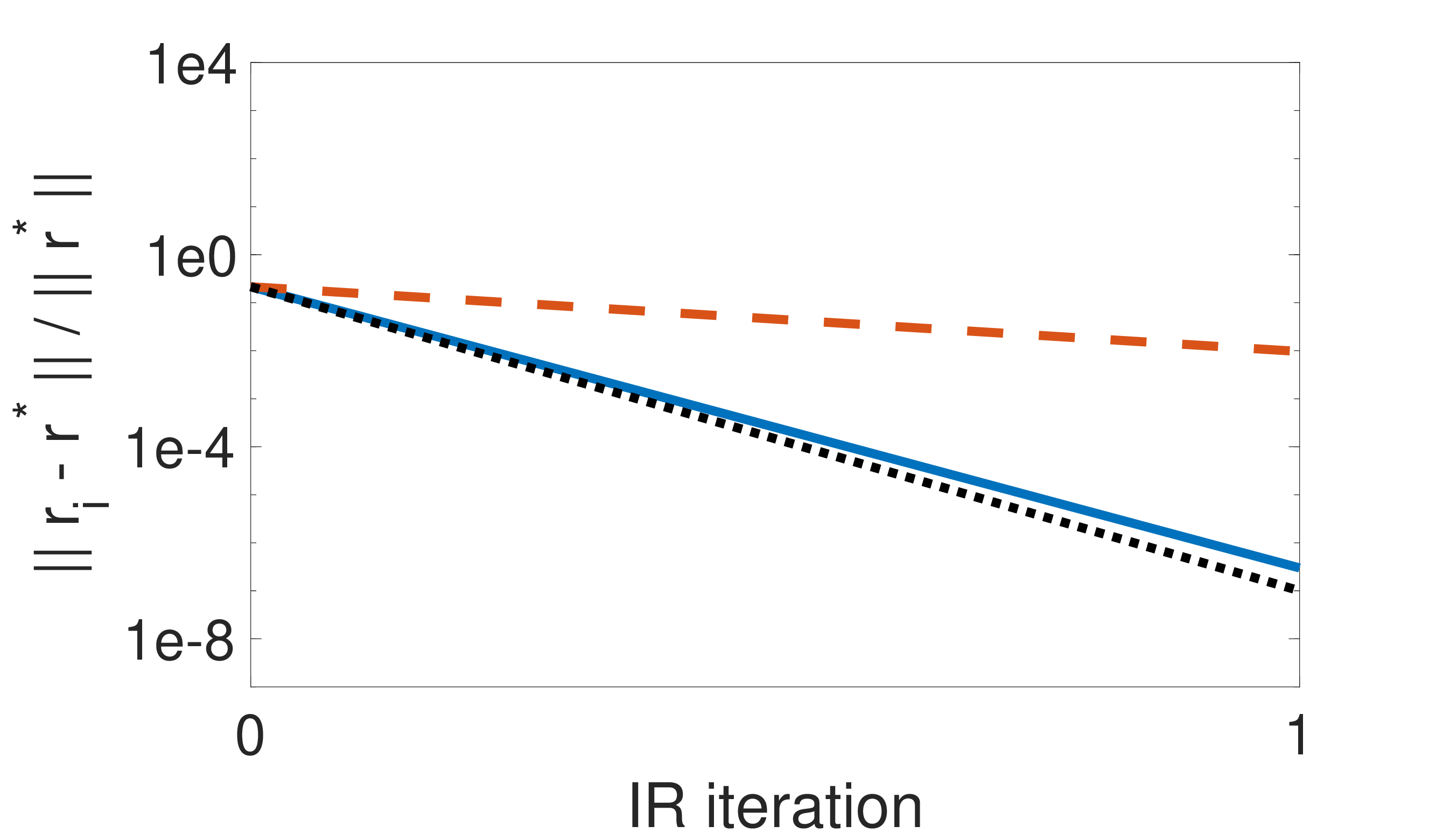}
  \caption{$\kappa_2(A)=$1e1, $\Vert r^* \Vert_2=$1e-6}
\end{subfigure}
\begin{subfigure}[t]{0.45\linewidth}
  \centering
 \includegraphics[width=\linewidth]{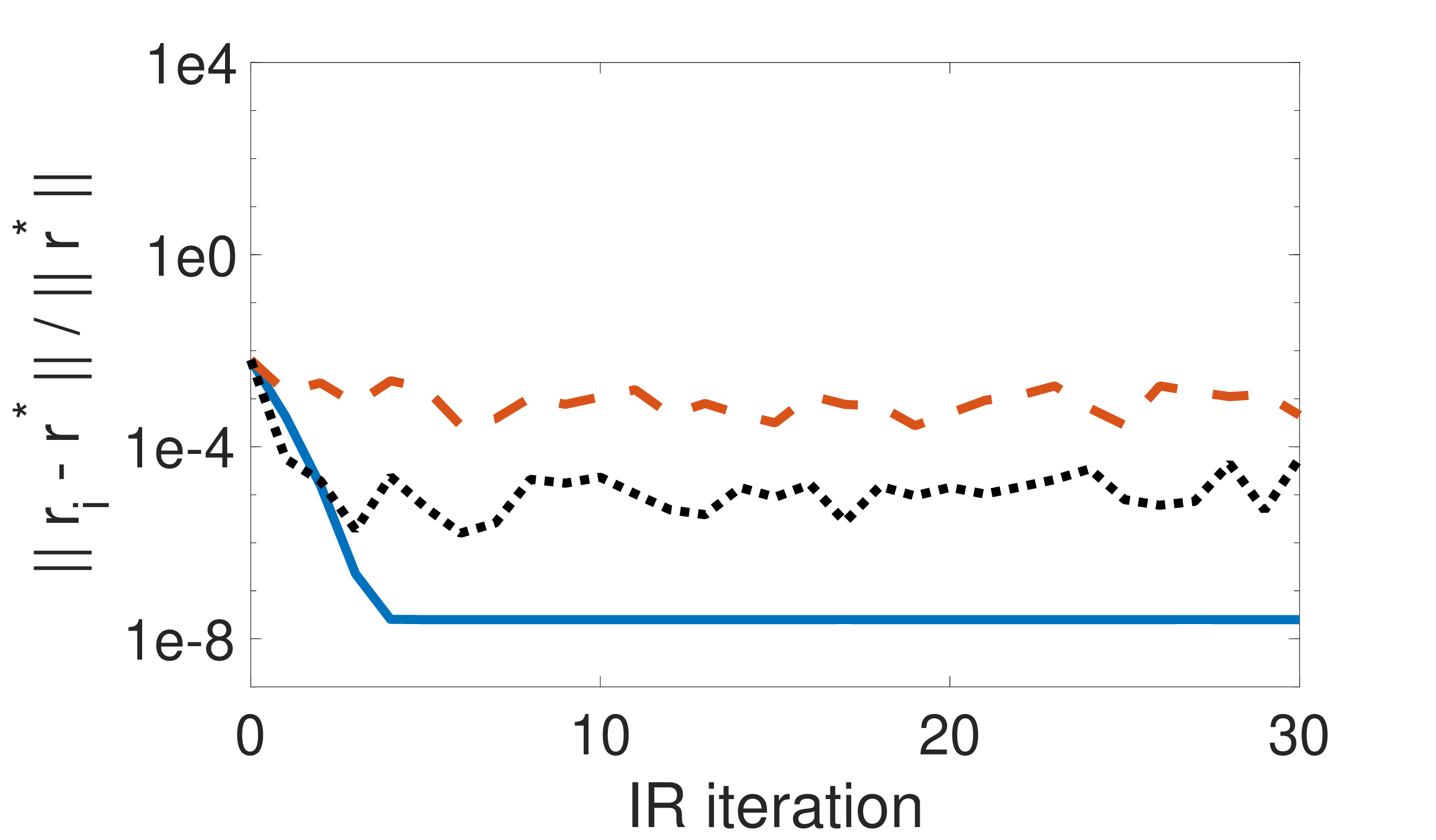}
  \caption{$\kappa_2(A)=$1e6, $\Vert r^* \Vert_2=$1e0 }
\end{subfigure}
\begin{subfigure}[t]{0.45\linewidth}
  \centering
 \includegraphics[width=\linewidth]{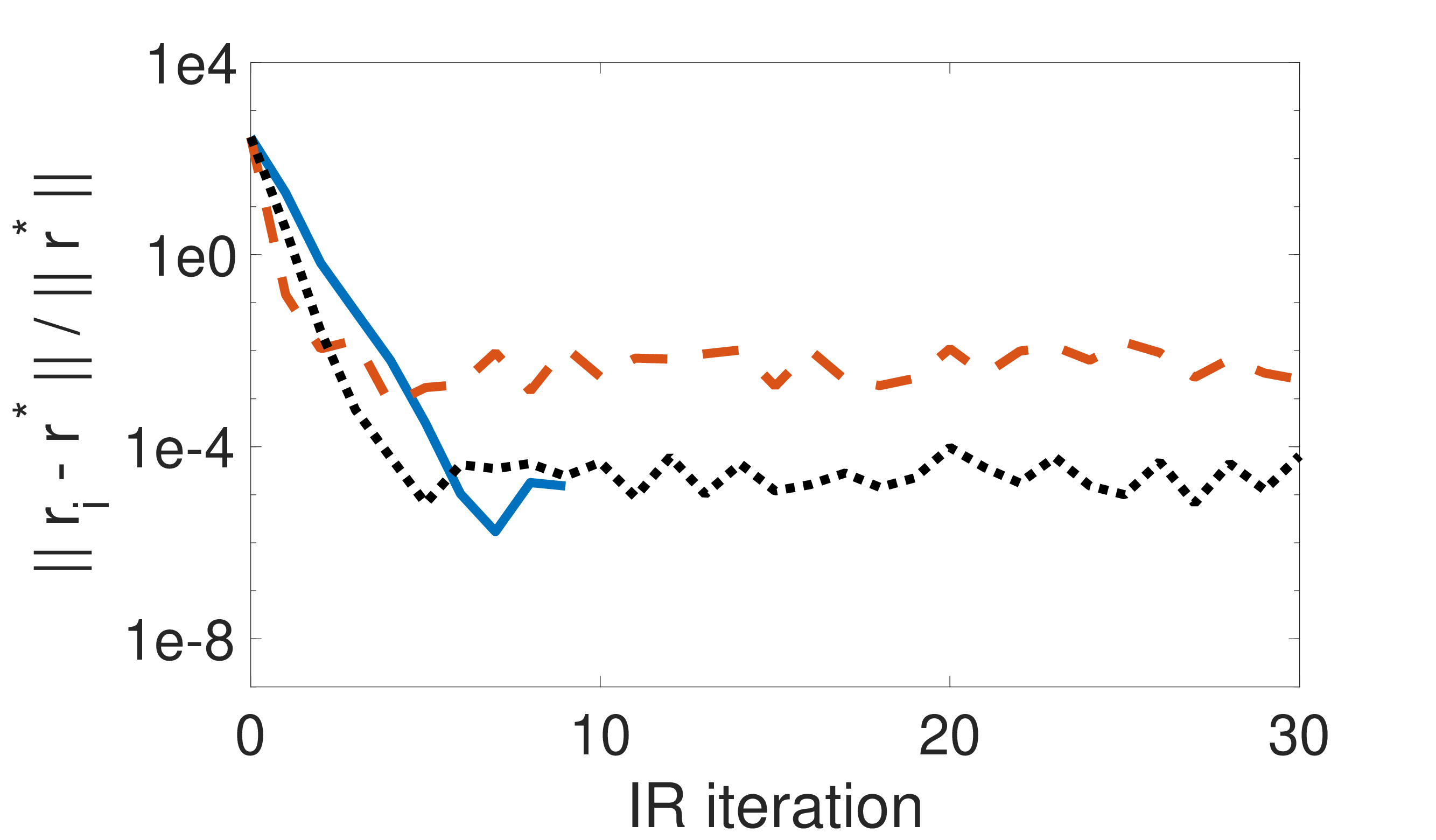}
  \caption{$\kappa_2(A)=$1e6, $\Vert r^* \Vert_2=$1e-6}
\end{subfigure}
    \caption{As in Figure~\ref{fig:x_conv_iterative}, but for the relative error in $r$.}
    \label{fig:r_conv_iterative}
\end{figure}

\section{Summary}\label{sec:summary}

We discussed the least-squares system, semi-normal equations, and augmented system approaches to two-precision iterative refinement for least-squares problems. The focus was on refining the solution $x$, although the augmented system approach also explicitly refines the least-squares residual $r$.  The methods require a similar amount of computation in the high precision $u_r$; see Table~\ref{tab:computational_cost}. Depending on the relative cost of these computations a practitioner can gauge the cost of obtaining a more accurate solution than given by, e.g., a backward-stable solver or fixed precision LSIR.

In addition to reformulating existing results to enable a comparison of the different approaches, we provided a new analysis for the forward error in the two-precision semi-normal equations approach. 
 We further showed how to combine the LS approach, which can use iterative least-squares solvers, and the augmented system approach, which also refines the residual.
\begin{table}[]
    \centering
    \begin{tabular}{c|ccc| c}
    \multirow{2}{*}{LSIR approach}   &  \multicolumn{3}{c|}{$u_r$} & $u$ \\
                    &  $Az$ & $A^T z$& flops & flops \\
      \hline
        LS system & yes & no & $2mn$ & $4mn + n^2$\\
        Semi-normal equations & yes & yes & $4mn - n$ &  $2n^2+n$\\
        Augmented system & yes & yes & $4mn +m -n $ & $8mn + 2n^2$
    \end{tabular}
    \caption{Computational cost per LSIR iteration for the three LSIR approaches in two precisions $u_r$ and $u$ with $u_r \ll u$ when $A$ is dense and Householder QR is used to solve the update equations. We specify the number of flops and whether a matrix-vector product with $A$ and $A^T$ has to be computed in the higher precision.}
    \label{tab:computational_cost}
\end{table}
The main observations are summarized in the following.
\begin{itemize}
    \item We do not have theoretical convergence guarantees for the LS approach and know that for some right-hand side vectors the approach may not be able to refine the solution to the level of the working precision, namely, when the least-squares residual is large. Numerical experiments show that the method does not converge when $\kappa(A) > u^{-1/2}$.
    \item The semi-normal equations approach converges for problems with $\kappa(A) \leq u^{-1/2}$ and is not sensitive to the size of the least-squares residual.
    \item Conditions for LSIR convergence in the solution and the residual when using the augmented system approach are available. Convergence for problems with high condition number (even when $\kappa(A) > u^{-1/2}$) and large residual can be obtained with a direct solver; however, a suitable preconditioner is crucial if iterative solvers are to be used. In either case, the situation can be further improved by the use of an optimal scaling.  
\end{itemize}
Thus the practitioner may choose to use the LS system approach if the problem is well-conditioned and the least-squares residual is small. If it is known that the problem is well-conditioned but the residual is large or unknown, then the semi-normal equations approach may be used. If the user has no knowledge of the conditioning and the size of the residual or it is known that the problem is ill-conditioned, then the augmented system approach is the most promising. The combined approach may be interesting in the cases where both $x$ and $r$ are to be refined and parallel iterative least-squares solvers are preferred. We summarize this advice in a decision tree in Figure~\ref{fig:tree}.

\begin{figure}[h]
\centering
\begin{tikzpicture}[scale=0.8,font=\small, edge from parent fork down,
  edge from parent/.style={->,black,thick,draw},
  level 1/.style={sibling distance=8cm, level distance=60pt},
  level 2/.style={sibling distance=6cm},
  level 3/.style={sibling distance=4cm}]
  \node (root) [treenode] {Refine both $x$ and $r$?} 
    child{ node [treenode] {Is $A$ ill-conditioned?} 
    child {node [treenode] {Is residual large?} 
      child {node [root] {LS} edge from parent node[above] {No}}  
      child {node [treenode] {Can compute accurate $R$-factor?} 
        child {node  [root]  {Semi-normal} edge from parent node[right] {Yes} } edge from parent node[above] {Yes/don't know} } edge from parent node[above] {No} } 
    child {node [root]  {Augmented} edge from parent node[align=left,above] {Yes/ \\ don't know}}
     edge from parent node[above] {Just $x$}}
    child{ node [treenode] {Can only run iterative \\ least-squares solver?}  
    child{ node [root] {Combined}  edge from parent node[right] {Yes} }
    edge from parent node[above] {Yes} }
     ;
     \draw[->,black,thick] (root-2) -|  node [text width=2.5cm,at start,above] {No} (root-1-2);
 \draw[->,black,thick] (root-1-1-2) -|  node [text width=2.5cm,midway,right] {No} (root-1-2);
\end{tikzpicture}
\caption{Decision tree for choosing the IR approach to refine the solution $x$ and possibly the residual $r$ based on available information about the problem.} \label{fig:tree}
\end{figure}

\bibliographystyle{siam}
\bibliography{biblio}

\begin{thebibliography}{10}

\bibitem{advanpix}
{\em Advanpix multiprecision computing toolbox for {M}{A}{T}{L}{A}{B}}.
\newblock \url{http://www.advanpix.com}.

\bibitem{amestoy2024five}
{\sc P.~Amestoy, A.~Buttari, N.~J. Higham, J.-Y. L’Excellent, T.~Mary, and
  B.~Vieubl{\'e}}, {\em Five-precision {GMRES}-based iterative refinement},
  SIAM Journal on Matrix Analysis and Applications, 45 (2024), pp.~529--552.

\bibitem{arioli1989augmented}
{\sc M.~Arioli, I.~S. Duff, and P.~P. de~Rijk}, {\em On the augmented system
  approach to sparse least-squares problems}, Numerische Mathematik, 55 (1989),
  pp.~667--684.

\bibitem{avron2010blendenpik}
{\sc H.~Avron, P.~Maymounkov, and S.~Toledo}, {\em Blendenpik: Supercharging
  {LAPACK}'s least-squares solver}, SIAM Journal on Scientific Computing, 32
  (2010), pp.~1217--1236.

\bibitem{bjorck1967iterative}
{\sc {\AA}.~Bj{\"o}rck}, {\em Iterative refinement of linear least squares
  solutions {I}}, BIT Numerical Mathematics, 7 (1967), pp.~257--278.

\bibitem{bjorck1967solving}
\leavevmode\vrule height 2pt depth -1.6pt width 23pt, {\em {S}olving linear
  least squares problems by {G}ram-{S}chmidt orthogonalization}, BIT Numerical
  Mathematics, 7 (1967), pp.~1--21.

\bibitem{bjorck1978comment}
\leavevmode\vrule height 2pt depth -1.6pt width 23pt, {\em Comment on the
  iterative refinement of least-squares solutions}, Journal of the American
  Statistical Association, 73 (1978), pp.~161--166.

\bibitem{bjorck1987stability}
\leavevmode\vrule height 2pt depth -1.6pt width 23pt, {\em Stability analysis
  of the method of seminormal equations for linear least squares problems},
  Linear Algebra and its Applications, 88 (1987), pp.~31--48.

\bibitem{bjorck1996numerical}
\leavevmode\vrule height 2pt depth -1.6pt width 23pt, {\em Numerical methods
  for least squares problems}, SIAM, 1996.

\bibitem{bjorck1994solution}
{\sc {\AA}.~Bj{\"o}rck and C.~C. Paige}, {\em Solution of augmented linear
  systems using orthogonal factorizations}, BIT Numerical Mathematics, 34
  (1994), pp.~1--24.

\bibitem{businger1965linear}
{\sc P.~Businger and G.~H. Golub}, {\em Linear least squares solutions by
  {H}ouseholder transformations}, Numerische Mathematik, 7 (1965),
  pp.~269--276.

\bibitem{carson2024mixed}
{\sc E.~Carson and I.~Dau\v{z}ickait\.{e}}, {\em {M}ixed precision sketching
  for least-squares problems and its application in {GMRES}-based iterative
  refinement}, arXiv preprint arXiv:2410.06319,  (2024).

\bibitem{carson2018accelerating}
{\sc E.~Carson and N.~J. Higham}, {\em Accelerating the solution of linear
  systems by iterative refinement in three precisions}, SIAM Journal on
  Scientific Computing, 40 (2018), pp.~A817--A847.

\bibitem{carson2020three}
{\sc E.~Carson, N.~J. Higham, and S.~Pranesh}, {\em Three-precision
  {GMRES}-based iterative refinement for least squares problems}, SIAM Journal
  on Scientific Computing, 42 (2020), pp.~A4063--A4083.

\bibitem{chen2023half}
{\sc Y.~Chen, X.~Ji, and J.~Nagy}, {\em Half-precision {K}ronecker product
  {S}{V}{D} preconditioner for structured inverse problems}, arXiv preprint
  arXiv:2311.15393,  (2023).

\bibitem{demmel2009extra}
{\sc J.~Demmel, Y.~Hida, E.~J. Riedy, and X.~S. Li}, {\em Extra-precise
  iterative refinement for overdetermined least squares problems}, ACM
  Transactions on Mathematical Software (TOMS), 35 (2009), pp.~1--32.

\bibitem{epperly2024fast}
{\sc E.~N. Epperly}, {\em Fast and forward stable randomized algorithms for
  linear least-squares problems}, SIAM Journal on Matrix Analysis and
  Applications, 45 (2024), pp.~1782--1804.

\bibitem{golub1965numerical}
{\sc G.~Golub}, {\em Numerical methods for solving linear least squares
  problems}, Numerische Mathematik, 7 (1965), pp.~206--216.

\bibitem{golub1966note}
{\sc G.~H. Golub and J.~H. Wilkinson}, {\em Note on the iterative refinement of
  least squares solution}, Numerische Mathematik, 9 (1966), pp.~139--148.

\bibitem{high:ASNA2}
{\sc N.~J. Higham}, {\em Accuracy and Stability of Numerical Algorithms},
  Society for Industrial and Applied Mathematics, Philadelphia, PA, USA,
  second~ed., 2002.

\bibitem{higham2021exploiting}
{\sc N.~J. Higham and S.~Pranesh}, {\em Exploiting lower precision arithmetic
  in solving symmetric positive definite linear systems and least squares
  problems}, SIAM Journal on Scientific Computing, 43 (2021), pp.~A258--A277.

\bibitem{li2024double}
{\sc H.~Li}, {\em Double precision is not necessary for {LSQR} for solving
  discrete linear ill-posed problems}, Journal of Scientific Computing, 98
  (2024), p.~55.

\bibitem{paige1982lsqr}
{\sc C.~C. Paige and M.~A. Saunders}, {\em {LSQR}: An algorithm for sparse
  linear equations and sparse least squares}, ACM Transactions on Mathematical
  Software (TOMS), 8 (1982), pp.~43--71.

\bibitem{rokhlin2008fast}
{\sc V.~Rokhlin and M.~Tygert}, {\em A fast randomized algorithm for
  overdetermined linear least-squares regression}, Proceedings of the National
  Academy of Sciences, 105 (2008), pp.~13212--13217.

\bibitem{saad2003iterative}
{\sc Y.~Saad}, {\em Iterative methods for sparse linear systems}, SIAM, 2003.

\bibitem{scott2022computational}
{\sc J.~Scott and M.~T\r{u}ma}, {\em A computational study of using black-box
  {QR} solvers for large-scale sparse-dense linear least squares problems}, ACM
  Transactions on Mathematical Software (TOMS), 48 (2022), pp.~1--24.

\bibitem{von1947numerical}
{\sc J.~Von~Neumann and H.~Goldstine}, {\em {N}umerical inverting of matrices
  of high order}, Bulletin of the American Mathematical Society, 53 (1947),
  pp.~1021--1099.

\end{thebibliography}

\end{document}